\documentclass{amsart}
\usepackage{tikz}
\usetikzlibrary{calc,arrows.meta}
\usepackage{xcolor}
\usepackage{amssymb,latexsym,amsmath,extarrows}
\usepackage{graphicx,mathrsfs,comment}
\usepackage[margin=1in, centering]{geometry}
\usepackage{hyperref,url}
\usepackage{cite}

\renewcommand{\a}{\alpha}
\newcommand{\s}{\psi}

\usepackage{mathabx}

\newcommand{\mcM}{\mathcal{M}}

\makeatletter
\newcommand\@avprod[2]{%
  {\sbox0{$\m@th#1\prod$}%
   \vphantom{\usebox0}%
   \ooalign{%
     \hidewidth
     \smash{\vrule height\dimexpr\ht0+1pt\relax depth\dimexpr\dp0+1pt\relax}%
     \hidewidth\cr
     $\m@th#1\prod$\cr
   }%
  }%
}
\newcommand{\avprod}{\mathop{\mathpalette\@avprod\relax}\displaylimits}

\usepackage{amstext}
\usepackage{bbm}

\numberwithin{equation}{section}

\newtheorem{theorem}{Theorem}[section]
\newtheorem{lemma}[theorem]{Lemma}

\newtheorem{proposition}[theorem]{Proposition}
\newtheorem{remark}[theorem]{Remark}

\newtheorem{definition}[theorem]{Definition}

\newcommand{\al}{\alpha}

\newcommand{\ga}{\gamma}
\newcommand{\Ga}{\Gamma}
\newcommand{\de}{\delta}

\newcommand{\De}{\Delta}
\newcommand{\e}{\epsilon}

\newcommand{\ka}{\kappa}

\newcommand{\si}{\sigma}

\newcommand{\om}{\omega}

\newcommand{\wt}{\widetilde}
\newcommand{\wh}{\widehat}

\newcommand{\bI}{{\bf I}}

\newcommand{\bfe}{{\bf e}}

\newcommand{\BR}{{\rm{Br}}}
\newcommand{\R}{\mathbb{R}}

\newcommand{\T}{\mathbb{T}  }

\newcommand{\I}{\mathbb{I}}

\newcommand{\rap}{{\rm RapDec}(R)}

\newcommand{\supp}{\textup{supp}}

\newcommand{\bfv}{\mathbf{v}}
\newcommand{\lam}{\lambda}
\newcommand{\bfn}{\mathbf{n}}

\newcommand{\cD}{\mathcal{D}}
\newcommand{\cU}{\mathcal{U}}

\newcommand{\cJ}{\mathcal{J}}

\begin{document}

\title[]{Sharp local smoothing estimates for curve averages}

\author{Shengwen Gan} \address{ Shengwen Gan\\  Department of Mathematics\\ Sun Yat-sen University\\ Guangzhou, 510275, P.R. China} \email{shengwengan2018@gmail.com}

\author{Dominique Maldague} \address{ Dominique Maldague\\  Department of Mathematics\\ University of Cambridge, UK}\email{dominiquemaldague@gmail.com}

\author{Changkeun Oh}\address{ Changkeun Oh\\ Department of Mathematical Sciences and RIM, Seoul National University, Republic of Korea} \email{changkeun.math@gmail.com}

\begin{abstract}
We prove sharp local smoothing estimates for curve averages in all dimensions. As a corollary, we prove the sharp $L^p$ boundedness of the helical maximal operator in $\R^4$, which was previously known only for $\R^2$ and $\R^3$.  We also improve  previously known results in higher dimensions. {The main new ingredient is a novel wave envelope estimate adapted to moment curves, which is a powerful tool in the proof of the local smoothing estimate.}
\end{abstract}

\maketitle



\section{Introduction}

Consider a $C^{\infty}$  curve $\gamma_n: [0,1] \rightarrow \mathbb{R}^n$ and the {dilated averaging} operator
 \begin{equation}
    A_tf(x):=\int_{\R}f(x-t\gamma_n(s))\chi(s)ds.
\end{equation}
Here, $\chi(s)$ is a {fixed} smooth bump function supported on $[0,1]$. We say that the curve $\gamma_n$ is nondegenerate if 
\begin{equation}\label{12}
    \det{(\gamma_n'(s), \gamma_n''(s), \ldots, \gamma_n^{(n)}(s) )} \neq 0
\end{equation}
for all $s \in [0,1]$. {The prototypical example of the curve is the moment curve $\gamma_n(s)=(s,s^2,\ldots,s^n)$.}
Our main results are sharp local smoothing estimates for curve averages in all dimensions.

\begin{theorem}\label{mainthm}
    For $2 \leq p <\infty$ and $\sigma < \sigma(p,n):= \min \{\frac1n,\frac1n(\frac12+\frac2p),\frac2p \}$, we have
    \begin{equation}\label{0703.13}
        \Big( \int_{1}^2 \|A_tf\|_{L_{\sigma}^p(\R^n) }^p \,dt \Big)^{\frac1p} \lesssim_{p,\gamma_n,\chi} \|f\|_{L^p(\R^n)}
    \end{equation}
    for all functions $f \in L^p(\R^n)$.
\end{theorem}

The local smoothing estimate was first observed by \cite{MR1098614}, who noticed that the $L^p$ boundedness of Bourgain's circular maximal function follows from a local smoothing estimate.
The main application of Theorem \ref{mainthm} is the $L^p$ boundedness of a helical maximal operator.
Define the helical maximal operator by
\begin{equation}
    M_{\gamma_n}f(x):=\sup_{t>0}|A_tf(x)|.
\end{equation}
The following is a consequence of Theorem \ref{mainthm}. This implication is standard; see Section 2 of \cite{MR4861588}.

{\begin{theorem}\label{0703.thm12}
Let $n \ge 2$. Then the maximal operator $M_{\gamma_n}$ is bounded on $L^p(\mathbb{R}^n)$; that is,
\[
\| M_{\gamma_n} f \|_{L^p(\mathbb{R}^n)} \lesssim_{p,\gamma_n,\chi}  \| f \|_{L^p(\mathbb{R}^n)},
\]
whenever
\[
p >
\begin{cases}
n, & n = 2,3,4, \\[4pt]
2(n-2), & n \ge 5.
\end{cases}
\]
\end{theorem}}

The helical maximal operator has a long history. Stein considered the spherical maximal operator 
\begin{equation}
    M_{\mathrm{sph}}f(x):= \sup_{t>0} \Big|\int_{S^{n-1} }f(x-tw) d\sigma^{n-1}(w) \Big|
\end{equation}
where $d\sigma^{n-1}$ is the normalized surface measure on the unit sphere $S^{n-1} \subset \R^n$. In \cite{MR420116},
Stein proved the sharp $L^p$ boundedness of the operator for dimensions greater than two. The two-dimensional case remained open for a while and was proved by Bourgain \cite{MR874045}. 

In the unpublished survey of Christ from the late 1980s\footnote{See page 62 of \cite{MR2288738}.},  he posed the question of determining the range of $p$ for which the $L^p$ boundedness of the helical maximal operator holds true. This question can be thought of as a natural extension of Bourgain's theorem to higher dimensions. It is believed that the operator $M_{\gamma_n}f$ is $L^p$ bounded if and only if $p>n$. {Bourgain's argument bounding the circular maximal function also confirms that $M_{\gamma_2}$ is $L^p$ bounded if and only if $p>2$}. For $n=3$, after partial progress in \cite{MR2288738}, it was independently proved by \cite{MR4411734} and \cite{MR4861588}.  Our theorem proves the conjecture for $n=4$. For $n \geq 5$, prior to our work, the best known bound was obtained by \cite{MR4549710}, and our theorem improves theirs. {The range of $p$ that we obtain for the $L^p$ boundedness of $M_{\gamma_n}$ is the best possible using the local smoothing estimate, as observed in \cite{BH}. }

\medskip

As mentioned, the local smoothing estimate implies the $L^p$ boundedness of the helical maximal operator for some range of $p$.
The sharp local smoothing estimate for $n=2$ was proved in the celebrated work by \cite{MR4151084}\footnote{More precisely, a square function estimate for a cone in $\R^3$ was proved by \cite{MR4151084}, and the sharp local smoothing estimate is a corollary of their square function estimate. We refer to \cite{MR1173929} for the implication.}, and has been open for $n  \geq 3$. By testing {standard} examples, one can see that the local smoothing estimate \eqref{0703.13} fails without the condition $\sigma < \min \{ \frac1n, \frac2p \}$. {Recently, \cite{BH} modified an example in \cite{MR4340226} to demonstrate an additional restriction on $\sigma$, which led to the conjecture solved by Theorem \ref{mainthm}.} 
{Note that there are two critical exponents of $p$ in Theorem \ref{mainthm}: $p=4$ and $p=4n-4$.} Before our work, the best known bound for $n \geq 3$ was obtained by \cite{MR4549710}, where they proved the estimate \eqref{0703.13} for $p \geq 4n-2$ and $\si<2/p$. Our theorem gives sharp local smoothing estimates for all ranges of $p$ in all dimensions. 

 \medskip

Let us discuss why Theorem \ref{mainthm} is called the local smoothing estimate. We can compare it with a fixed-time estimate: for $2 \leq p < \infty$ and $\sigma < \min\{\frac1p, \frac1n(\frac12+\frac{1}{p})\}$,
\begin{equation}
    \|A_1 f \|_{L^p_{\sigma}(\R^n)} \leq C_{p,\gamma_n,\chi}\|f\|_{L^p(\R^n)}.
\end{equation}
This estimate is called the $L^p$ Sobolev regularity estimate.
For $n=2$, this was proved by \cite{MR1330238}, and it was open for a while in higher dimensions. The cases $n=3,4$ were proved by \cite{MR2288738} and \cite{MR4340226}, respectively, and the case $n \geq 5$ was proved by \cite{MR4549710}. In Theorem \ref{mainthm}, we have an additional integration over $t$, and this gives an additional smoothing effect. This is why Theorem \ref{mainthm} is called the local smoothing estimate.

\medskip

\subsection{Main obstacle}

We explain the main obstacle to proving Theorem \ref{0703.thm12} in dimensions $\geq 4$. 
We write the average operator as  
\begin{equation}
    A_t f(x) = (f \delta_{t=0})* \mu_{\Gamma}(x,t),
\end{equation}
where $\mu_\Ga$ is a smooth surface measure on $\Ga=\{t(\ga_n(s),1): t\in [1,2],\, s \in [0,1]\}$.
{Formally,} taking the Fourier transform in the $(x,t)$ variables yields
\begin{equation}\label{0905.111}
    \widehat{A_tf}(\xi,\xi_{n+1}) = \widehat{f}(\xi) \widehat{\mu_{\Gamma}}(\xi,\xi_{n+1}),
\end{equation}
where
\begin{equation}
    \widehat{\mu_{\Gamma}}(\xi,\xi_{n+1})=\iint e^{- i t(\gamma_n(s),1) \cdot (\xi,\xi_{n+1})} \psi_1(s) \psi_2(t) \, dsdt
\end{equation}
and $\psi_1,\psi_2$ are smooth cutoff functions.
The key task is to estimate the multiplier $\widehat{\mu_{\Gamma}}$.  Let us explain why it is difficult to analyze this multiplier for dimensions $n\geq 4$.
\medskip

Suppose that $|(\xi,\xi_{n+1})| \sim R$.
Let us first discuss the $n=2$ case. By a stationary phase argument, one may see that
\begin{equation}
    |\widehat{\mu_{\Gamma}}(\xi_1,\xi_2,\xi_3)| \approx R^{-\frac12}\chi_{\mathcal{C}}(\xi_1,\xi_2,\xi_3) + \mathrm{Error}
\end{equation}
where $\mathcal{C}$ is a small neighborhood of a cone in $\mathbb{R}^3$.
In other words, the essential support of $|\widehat{\mu_{\Gamma}}|$ forms a cone and has essentially the same values on it. This gives a satisfactory decay rate for the function.
In the three-dimensional case,  the decay rate of $\widehat{\mu_{\Gamma}}$ depends heavily on the choice of the point $(\xi,\xi_{4})$. To deal with this challenge, \cite{MR4340226} used the Van der Corput lemma, which is a classical approach to estimating oscillatory integrals associated with moment curves. More precisely, applying the lemma to the $s$-variable, we have $|\widehat{\mu_{\Gamma}}(\xi,\xi_4)| \lesssim R^{-\frac13}$. They noticed that this decay rate is sharp only when $(\xi,\xi_4)$ lies in a small neighborhood of a two-dimensional cone in $\mathbb{R}^4$. After projecting it onto the $\xi$-plane, they observed that $\xi$ must lie in a small neighborhood of a two-dimensional cone in $\mathbb{R}^3$.
This leads to considering the dyadic decomposition of the frequencies $\xi$ according to the distance to the cone. This gives a geometric description of the decay rate for $\widehat{\mu_{\Gamma}}$.

In higher dimensions, determining the {relevant stationary points} $(\xi,\xi_{n+1})$ requires solving a system of equations corresponding to the vanishing of certain polynomials, {which is a barrier to extending this method} to higher dimensions $n \geq 4$. 
\medskip

The main novelty of our paper is to introduce a new geometric framework to describe the behavior of $\widehat{\mu_{\Gamma}}$. 
We use a \textit{wave packet approach} to estimate oscillatory integrals, thereby avoiding the use of the Van der Corput lemma and the analysis of the zeros of the derivatives of the phase function. We explain the ideas in detail in the next subsection.

\subsection{Ideas of the proof of Theorem \ref{mainthm} for \texorpdfstring{$n=3$}{n=3}}

Let us give a sketch of the proof of Theorem \ref{mainthm} for the three-dimensional case.
The first step is to calculate $\widehat{\mu_{\Gamma}}$. We devise a {wave packet approach} to estimate the oscillatory integral $\widehat{\mu_{\Gamma}}$. We may fix $R>1$ and analyze it in the region $B^4_R(0)\setminus B^4_{R/2}(0)$. 

Let us explain the idea. Denote by $\{ I: |I|=R^{-1/2}\}$ a finitely overlapping cover of $[0,1]$ and $\{\chi_I\}$ the associated partition of unity. For each interval $I$, we denote $\Ga_I:=\{ t(\ga_3(s),1):t\in [1,2], s\in I \}$. Let
\begin{equation}
    \widehat{\mu_{\Gamma_I}}(\xi,\xi_{4}):=\iint e^{- i t(\gamma_3(s),1) \cdot (\xi,\xi_{4})} \chi_I(s) \psi_1(s)  \psi_2(t) \, dsdt. 
\end{equation}
 We then have
$
    \widehat{\mu_{\Gamma}}= \sum_{|I|=R^{-1/2} } \widehat{\mu_{\Gamma_I}}$.
Morally speaking, $\mu_{\Ga_I}\approx 1_{\Ga_I}$.
Since $\Ga_I$ is contained in a box of dimensions $R^{-1}\times R^{-1}\times  R^{-1/2}\times 1$, by the uncertainty principle, $\wh \mu_{\Ga_I}$ is essentially supported in a {dual} box {centered at the origin} of dimensions $R\times R\times R^{1/2}\times 1$, which we denote by $P_I$.
We have
\begin{equation}
    \widehat{\mu_{\Gamma_I}}(\xi,\xi_{4}) = \widehat{\mu_{\Gamma_I}}(\xi,\xi_{4}) \chi_{P_I}(\xi,\xi_4) + \mathrm{Error}.
\end{equation}
So we have the expression
\begin{equation}
    \widehat{\mu_{\Gamma}}(\xi,\xi_4)= \sum_{|I|=R^{-1/2} } \widehat{\mu_{\Gamma_I}} (\xi,\xi_4) \chi_{P_I}(\xi,\xi_4) + \mathrm{Error}.
\end{equation}

\begin{figure}
\begin{tikzpicture}

\begin{scope}
   [x={(6cm,0)},
    y={({cos(45)*.5cm},{sin(45)*.5cm})},
    z={({cos(90)*6cm},{sin(90)*6cm})},line join=round, black]
  \draw[] (0,0,0) -- (0,0,1) -- (0,1,1);
   \draw[] (0,0,0) -- (1,0,0) -- (1,1,0);

  \draw[] (1,0,0) -- (1,0,1) -- (1,1,1) -- (1,1,0) -- cycle;
 \draw[] (0,0,1) -- (1,0,1) -- (1,1,1) -- (0,1,1) -- cycle;
 \end{scope}  
 
\begin{scope}
   [shift={(1.5cm,1.5cm)},x={(3cm,0)},
    y={(0,0)},
    z={({cos(90)*3cm},{sin(90)*3cm})},line join=round, black]
  \draw[] (0,0,0) -- (0,0,1) -- (1,0,1) -- (1,0,0) --cycle;
\end{scope}

\begin{scope}
   [x={(1.5cm,0)},
    y={({cos(45)*.5cm},{sin(45)*.5cm})},
    z={({cos(90)*6cm},{sin(90)*6cm})},line join=round, black]
  \draw[] (0,0,0) -- (0,0,1) -- (0,1,1);
   \draw[] (0,0,0) -- (1,0,0);
  \draw[] (1,0,0) -- (1,0,1);
 \draw[] (0,0,1) -- (1,0,1) -- (1,1,1) -- (0,1,1) -- cycle;
 \end{scope}  

\begin{scope}
   [shift={(4.5cm,0)},x={(1.5cm,0)},
    y={({cos(45)*.5cm},{sin(45)*.5cm})},
    z={({cos(90)*6cm},{sin(90)*6cm})},line join=round, black]
  \draw[] (0,0,0) -- (0,0,1) -- (0,1,1);
   \draw[] (0,0,0) -- (1,0,0);
  \draw[] (1,0,0) -- (1,0,1);
 \draw[] (0,0,1) -- (1,0,1) -- (1,1,1) -- (0,1,1) -- cycle;
 \end{scope}

\begin{scope}
   [shift={(2.75cm,0)},x={(0.5cm,0)},
    y={({cos(45)*.5cm},{sin(45)*.5cm})},
    z={({cos(90)*6cm},{sin(90)*6cm})},line join=round, black]
  \draw[] (0,0,0.75) -- (0,0,1) -- (0,1,1);
   \draw[] (0,0,0) -- (0,0,0.25);
  \draw[] (1,0,0) -- (1,0,0.25);
  \draw[] (1,0,0.75) -- (1,0,1);
 \draw[] (0,0,1) -- (1,0,1) -- (1,1,1) -- (0,1,1) -- cycle;
 
 \end{scope} 

\node[scale=0.8] at (0.8,3)
{$P_{I,\textup{high}}$};

\node[scale=0.8] at (3,3)
{$P_{I,\textup{low}}$};

\node[scale=0.8] at (3,5.3)
{$P_{I,R^{-\frac{1}{3}}}$};

\node[scale=0.8] at (2,5.2)
{$P_{I,\lambda}$};

 
\end{tikzpicture}
\caption{High/low and intermediate-scale decomposition of $P_I$.}
\label{HLdecomposition}
\end{figure}

We next partition $P_I$ as follows. 
\[ P_I= P_{I,\textup{high}} \sqcup P_{I,\textup{low}}\sqcup \bigsqcup_{R^{-1/3}\le \lam\le 1 : \mathrm{dyadic} } P_{I,\lam}. \]
See Figure \ref{HLdecomposition}. In the figure, we only draw the directions with lengths $R,R,R^{1/2}$. $P_{I,\textup{high}}$ is the outermost part with dimensions $\sim R\times R\times R^{1/2}\times 1$. $P_{I,\textup{low}}$ is $P_I\cap B^4_{R/2}(0)$. (Here, high/low refers to how far away it is from the origin.) For each dyadic number $R^{-1/3}\le \lam\le 1$, $P_{I,\lam}$ is of dimensions $\sim R\times R\lam \times R^{1/2}\times 1$. Choose a partition of unity so that  
\[ \chi_{P_I}= \chi_{P_{I,\textup{high}}} +\chi_{P_{I,\textup{low}}}+\sum_{R^{-1/3}\le \lam\le 1} \chi_{P_{I,\lam}}. \]
We hence obtain that
\[ \wh \mu_{\Ga}= \sum_I \wh \mu_{\Ga_I} \chi_{P_{I,\textup{high}}}+\sum_I \wh \mu_{\Ga_I} \chi_{P_{I,\textup{low}}}+\sum_\lam \sum_I \wh \mu_{\Ga_I} \chi_{P_{I,\lam}}+\textup{Error}.  \]

We consider the three terms separately. Since we restrict our domain to the region $B^4_R(0)\setminus B^4_{R/2}(0)$, we don't need to estimate the second term. If the first term dominates, noting that $\|\wh \mu_{\Ga_I}\|_\infty\lesssim |I|=R^{-1/2}$, we have 
\[ |\wh \mu_\Ga|\lessapprox R^{-1/2} \sum_I  \chi_{P_{I,\textup{high}}}.  \]
One may observe that $\{P_{I,\textup{high}}\}_I$ are essentially disjoint. Hence, we get
\[ |\wh \mu_\Ga|\lessapprox R^{-1/2}   \chi_{\cup_I P_{I,\textup{high}}}. \]

The intermediate scenario is when some $\lam$-term dominates:
\begin{equation}\label{cancel}
    |\wh \mu_\Ga |\lessapprox |\sum_I \wh \mu_{\Ga_I}\chi_{P_{I,\lam}}|. 
\end{equation} 
A naive upper bound for $|\wh\mu_\Ga|$ is $\sup_I \|\wh \mu_{\Ga_I}\|_\infty \sum_I \chi_{P_{I,\lam}}$. However, we can do better.
A key observation is that by grouping different $I$'s, we get cancellation. Here is the precise argument. Denote by $\{ I': |I'|=(R\lam)^{-1/2} \}$ a finitely overlapping cover of $[0,1]$ and $\{\chi_{I'}\}$ the associated partition of unity. Similarly, denote $\Ga_{I'}:=\{ t(\ga_3(s),1):t\in[1,2],s\in I' \}$. For each $|I|=R^{-1/2}$, we let $\wt P_{I,\lam}$ be a box of dimensions $R\times R\lam\times (R\lam)^{1/2}\times 1$, which is obtained from $P_{I,\lam}$ by replacing its third side length with $(R \lambda)^{1/2}$. Later in the paper, we show that for each $|I'|=(R\lam)^{-1/2}$, the boxes $\{ \wt P_{I,\lam} \}_{I\subset I'}$ are comparable, which we may unambiguously denote by $Q_{I',\lam}$; also $\{Q_{I',\lam}\}_{I'}$ are essentially disjoint. Another fact we will show is that
\[ \sum_{I\subset I'} \wh\mu_{\Ga_I}\chi_{P_{I,\lam}} \]
is essentially supported in $Q_{I',\lam}$.
Assuming this result, we have
\[ |\sum_{I\subset I'} \wh\mu_{\Ga_I}\chi_{P_{I,\lam}}|\approx |\sum_{I\subset I'} \wh\mu_{\Ga_I}\chi_{P_{I,\lam}}\chi_{Q_{I',\lam}}|=|\sum_{I\subset I'} \wh\mu_{\Ga_I}|\chi_{Q_{I',\lam}}\lesssim (R\lam)^{-1/2}\chi_{Q_{I',\lam}}. \]
Hence, since $\{Q_{I',\lam}\}_{I'}$ are essentially disjoint, we have
\[ |\wh \mu_\Ga|\lessapprox (R\lam)^{-1/2}\chi_{\cup_{I'}Q_{I',\lam}}. \]
This gives a satisfactory bound for $\widehat{\mu}_{\Gamma}$. Compared with \eqref{cancel}, its support lies in the smaller region $\cup_{I'}Q_{I',\lam}$, instead of $\cup_{I}P_{I,\lam}$.
\\


We have finished obtaining the estimate of $\widehat{\mu_{\Gamma} }$. Applying the estimate of $\widehat{\mu_{\Gamma}}$ to \eqref{0905.111} and using the standard interpolation argument, Theorem \ref{mainthm} can be viewed as
\begin{equation}\label{09.19.115}
    \Big\|  \big( \hat{f}(\xi) \chi_{\cup_{I'} Q_{I',\lam} } (\xi,\xi_4) \big)^{\vee}    \Big\|_{L^p(\R^3\times [0,1]) } \leq C_{p,\e} R^{\e} (R\lam)^{\frac12}  R^{-\frac13(\frac12+\frac2p)} \|f\|_{L^p(\R^3)}
\end{equation}
for any $f$ with $\supp(\wh f)\subset \{|\xi|\sim R\}$, $\lam \in [R^{-1/3},1]$ and $p=4,8$. To put this inequality into the standard form studied in the literature, we  perform a rescaling by a factor of $R$ in the physical space and a rescaling by a  factor of $R^{-1}$ in the frequency space. Denote $q_{I',\lam}:=R^{-1} Q_{I',\lam}$. We also localize the domain of the integral to a ball of radius $R$. After these steps, \eqref{09.19.115} is reduced to proving
\begin{equation}\label{0914.117}
    \Big\|  \big( \hat{f}(\xi) \chi_{\cup_{I'}  q_{I',\lam} } (\xi,\xi_4) \big)^{\vee}    \Big\|_{L^p(B_R^4) } \leq C_{p,\e} R^{\e} R^{\frac1p}(R\lam)^{\frac12}  R^{-\frac13(\frac12+\frac2p)} \|f\|_{L^p(\R^3)}
\end{equation}
for any $f$ with $\supp(\wh f)\subset \{|\xi|\sim 1\}$, $\lam \in [R^{-1/3},1]$ and $p=4,8$.

We would like to discuss the proof of \eqref{0914.117}.
Let us discuss the case $p=4$ first. In the literature, square function estimates have been used to obtain local smoothing estimates of the type \eqref{0914.117}.  In \cite{MR4151084}, the authors proved a square function estimate for a cone in $\R^3$ and, via the Nikodym maximal function estimate, ultimately derived the $L^4$ local smoothing estimate for a cone in $\R^3$.
However, it is not hard to see that the $L^4$ local smoothing estimate can also be deduced directly from the wave envelope estimate, without relying on the Nikodym maximal function estimate. By applying their wave envelope estimate repeatedly, we obtain \eqref{0914.117} for $p=4$ and any $\lam$.

To obtain the estimate \eqref{0914.117} for $p=8$, we use a new {$(\ell^q,L^p)$} wave envelope estimate for $p > 4$. 
In the cases of the parabola in $\mathbb{R}^2$ and the cone in $\mathbb{R}^3$, the wave envelope estimates are equivalent to the square function estimates. However, it was unclear whether such an equivalence holds for general nondegenerate curves in $\mathbb{R}^n$ for $n \geq 3$.
Square function estimates for nondegenerate curves in $\mathbb{R}^n$ have been obtained in \cite{MR4794594, guth23}, but in their work, these do not follow from any known $L^p$ wave envelope estimates for such curves. We prove sharp $L^p$ wave envelope estimates for $p>4$ (Theorem~\ref{0501.thm12}), and deduce \eqref{0914.117} for $p=8$. Neither our Theorem~\ref{0501.thm12} nor the square function estimates for nondegenerate curves in $\mathbb{R}^n$ seem to imply each other, as can be seen from the difference in their sharp exponent ranges\footnote{In their work, square function estimates were established for $2 \leq p \leq n(n+1)/2 + 1$, and the inequality fails for $p > n(n+1)/2 + 1$.} for $p$.
 The proof of the $L^p$ wave envelope estimates relies on ideas developed in \cite{MR4794594}, which, in turn, build on earlier work by \cite{MR4151084, MR4721026}.

\subsection{Remarks}

Our approach is Fourier analytic. There is a geometric approach to proving boundedness of maximal functions; see \cite{MR1626711, hickman2025improvedlpboundsstrong, zahl2025maximalfunctionsassociatedfamilies} and references therein. {However, the approach does not seem to be developed enough yet to prove Theorem \ref{0703.thm12} for $n \geq 3$.}

There has been a lot of work generalizing maximal functions to a multi-parameter setting. See \cite{zahl2025maximalfunctionsassociatedfamilies, MR4800578, chen2025multiparametercinematiccurvature} for references. We also refer to \cite{MR1388870, MR4731854} for the $L^p \rightarrow L^q$ mapping properties of maximal operators.

Our wave envelope estimates for nondegenerate curves in $\R^n$ (Theorem \ref{0501.thm12}) imply sharp Bochner-Riesz type estimates for the curves (Theorem \ref{2025.04.25.thm11}). We introduce the estimates and give a proof of the implication in the next sections.

\subsection{Acknowledgements.}

Changkeun Oh
was supported by the New Faculty Startup Fund from Seoul National University, the POSCO Science Fellowship of POSCO TJ Park Foundation, and the National Research Foundation of Korea (NRF) grant funded by the Korea government (MSIT)  RS-2024-00341891.
We would like to thank David Beltran, Shaoming Guo, Larry Guth, Sanghyuk Lee, Sewook Oh, Andreas Seeger, and Joshua Zahl for valuable comments.

\subsection{Organization of the paper.}
The proof of Theorem \ref{mainthm} proceeds in two main steps. The first step is a Fourier-side reduction: after writing the averaging operator in terms of the cone measure associated with the curve, we analyze the Fourier decay of this measure and decompose the relevant frequency region into suitable planks. This reduction, carried out in Section 3, shows that Theorem \ref{mainthm} follows from the main estimate, Theorem \ref{05.29.thm28}. The second step is to prove this main estimate using wave envelope estimates and a high/low analysis. Section 2 states the wave envelope estimates for nondegenerate curves and cones, and also formulates the Bochner--Riesz type estimate which will be proved later. Section 4 develops the Kakeya-type estimates for the intermediate-scale planks that arise in the reduction. Section 5 reduces the wave envelope estimates to certain $\ell^{p/2}$ function estimates, while Section 6 sets up the technical ingredients for the high/low method, including wave packet decompositions, interpolation, and pruning. In Section 7, we prove the required $\ell^{p/2}$ function estimate for curves by induction on the dimension, using an auxiliary cone estimate which is proved in Section 8. These ingredients are then assembled in Section 9 to prove Theorem \ref{05.29.thm28}, completing the proof of Theorem \ref{mainthm}. 

The dependencies among the main estimates are somewhat indirect, so we
record the following schematic chain of implications. Here we suppress the
Kakeya-type inputs from Section \ref{sec3} and the standard reductions in
Sections \ref{section4} and \ref{0627.sec6}:
\[
\begin{gathered}
\text{Proposition \ref{1020.prop313}, together with the induction in Section \ref{0703.sec6}}
\Longrightarrow
\text{Propositions \ref{0621.prop51} and \ref{0621.prop52}} \\
\Longrightarrow
\text{Theorems \ref{25.01.30.thm51} and \ref{25.04.26.thm13}}
\Longleftrightarrow
\text{Theorems \ref{0501.thm12} and \ref{05.03.thm13}} \\
\Longrightarrow
\text{Theorem \ref{05.29.thm28}}
\Longrightarrow
\text{Theorem \ref{mainthm}}.
\end{gathered}
\]

Finally, Section 10 proves the Bochner--Riesz estimate stated in Section 2, and Section 11 discusses its sharpness.

\subsection{Notations.}\label{sec13.723}\hfill

\bigskip

\noindent
$\bullet$ 
Given $a \in \R^n$, we use the notation
$B_R^{n}(a):=\{x \in \R^n: |x-a| <R  \}$.

\noindent
$\bullet$ For a sequence $\{ a_i\}_{i=1}^{j}$ we introduce
\begin{equation}
\avprod_{i=1}^{j}a_i := \prod_{i=1}^j |a_i|^{\frac1j}.
\end{equation}

\noindent
$\bullet$ We use $\rap$ to denote a quantity of order $O_M(R^{-M})$ for any $M$.
\medskip

\noindent
$\bullet$ We use $A\lesssim B$ to mean that $A\le C B$ for some universal constant $C$. We use $A\lesssim_{a_1,\dots,a_m}B$ to mean that $A\le C_{a_1,\dots,a_m} B$ for some universal constant $C_{a_1,\dots,a_m}$ depending on parameters $a_1,\dots,a_m$.
\medskip

\noindent
$\bullet$ For a fixed scale $R>1$, we use $A\lessapprox B$ to denote $A\le C_\e R^\e B$ for any $\e>0$.
\medskip

\noindent
$\bullet$ We use $A\ll B$ to denote that $A\le cB$ for a sufficiently small constant $c>0$.
\medskip

\noindent
$\bullet$ For a rectangular box $P$, we use $W_P$ to denote a weight function essentially $\sim 1$ on $P$ and with a rapidly decaying tail outside $P$; we use $w_P$ to denote $|P|^{-1}W_P$, which is the $L^1$ normalization of $W_P$. The precise definition is given in Section \ref{section4}. We may use $\phi_P$ or $\chi_P$ to denote a smooth bump function adapted to $P$ (with compact support).
\medskip

\noindent
$\bullet$ For a nonnegative function $G$, define
\begin{equation}
    \|f\|_{L^p(G)}:= \Big( \int_{\R^n} |f(x)|^p G(x)\,dx \Big)^{\frac1p}.
\end{equation}

\section{Wave envelope estimates}\label{0701.sec2}

In this section,
we introduce a wave envelope estimate for a nondegenerate curve in $\R^n$. This is a key ingredient of the proof of Theorem \ref{mainthm}. The wave envelope estimate for a curve in $\R^2$ was first (implicitly) introduced by \cite{MR4151084} in the study of local smoothing estimates for the wave equation. To state our result, we need  some notation.
For each $R \geq 1$, define the anisotropic neighborhood $\mcM^n(\gamma_n; R) \subset \R^n$ of a nondegenerate curve $\gamma_n$ by
\begin{equation}\label{caps1}
    \mcM^n(\gamma_n;R):=\Big\{ \gamma_n(s) + \sum_{j=1}^n \lambda_j \gamma_n^{(j)}(s): s \in [0,1], \;\;\; |\lambda_j| \leq R^{-\frac{j}{n}}  \Big\}.
\end{equation}
For \(s_\theta\in R^{-\frac1n}\mathbb Z\), define \(\theta\in\Theta^n(R)\) by
\begin{equation}\label{0623.16}\theta :=
\left\{
\gamma_n(s)+\sum_{j=1}^n \lambda_j\gamma_n^{(j)}(s):
s\in [0,1]\cap [s_\theta,s_\theta+2R^{-1/n}],
\ |\lambda_j|\le R^{-\frac{j}{n}}
\right\}   
\end{equation}
and $s_{\theta} \in R^{-\frac1n}\mathbb{Z}$.
Note that $\Theta^n(R)$ is a finitely overlapping cover of $\mcM^n(\gamma_n;R)$, and the cardinality of the collection $\Theta^n(R)$ is comparable to $R^{\frac1n}$.
Let $\psi_{\theta}$ be a smooth partition of unity subordinate to the covering $\{\theta\}$, and define $\hat{f}_{\theta}:=\hat{f}\psi_{\theta}$, and $\mathrm{BR}_R(f):=\sum_{\theta} f_{\theta}$.

For each $\tau \in \Theta^n(s^{-n})$, denote the dual rectangular box of $\tau$ by $\tau^*$. Then $\tau^*$ has dimension $s^{-1} \times s^{-2} \times \cdots \times s^{-n}$ and is centered at the origin. Define $U_{\tau,R}$ to be a scaled version of $\tau^*$ with dimension $s^{n-1}R \times s^{n-2}R \times \cdots \times R$, which roughly contains $\theta^*$ for each $\theta \subset \tau$. Let $U \| U_{\tau,R}$ be a tiling of $\R^n$ by translates of $U_{\tau,R}$ and call each $U$ a wave envelope. Define the $l^q$-function corresponding to the wave envelope $U$ by
\begin{equation}
    S_{U,q}^{curve}f(x):= \Big(\sum_{\theta \subset \tau}|f_{\theta}(x)|^q w_U(x) \Big)^{\frac1q}.
\end{equation}
The function $w_U$ is  essentially supported on $U$, and  $L^1$-normalized in the sense that $\|w_U\|_1 \sim 1$.
    We refer to Subsection \ref{sec41} for the definition of $w_U$.  Here is a wave envelope estimate for a nondegenerate curve $\gamma_n$ in $\R^n$. 

\begin{theorem}\label{0501.thm12} Let $n \geq 2$. For  $4 \leq p \leq n^2+n-2$, $R \geq 1$, and $\epsilon>0$, we have
       \begin{equation*}
        \|f\|_{L^p(\R^n)} \leq C_{\e}R^{\e}   R^{\frac1n(\frac12-\frac2p)}  \Big(\sum_{ \substack{ R^{-\frac1n} \leq s \leq 1: \\ dyadic } } \sum_{\tau \in \Theta^n(s^{-n}) } \sum_{U \| U_{\tau,R}  } |U| \|S_{U,{\frac{p}{2}}}^{curve}f\|_{L^{\frac{p}{2}}(\R^n)}^p \Big)^{\frac1p}
    \end{equation*}
    for any Schwartz functions $f: \R^n \rightarrow \mathbb{C}$ whose Fourier supports are in $\mcM^n(\gamma_n;R)$.
    
\end{theorem}

Note that Theorem \ref{0501.thm12} was previously known for $n=2$ by \cite{MR4151084}. Our theorem can be thought of as a generalization of a wave envelope estimate to higher dimensions. One may see that the range of $p$ is sharp by using a bush example. We refer to Section \ref{sec10} for the example.

As a byproduct, we prove sharp Bochner-Riesz type estimates for nondegenerate curves in all dimensions. Recall that $\mathrm{BR}_R(f):=\sum_{\theta \in \Theta^n(R) } f_{\theta}$.

\begin{theorem}\label{2025.04.25.thm11}
Let $n \geq 2$.
For $2 \leq p < \infty$ and $\e>0$, we have
\begin{equation}
    \|\mathrm{BR}_R(f)\|_{L^p} \leq C_{\e} R^{\e} \big( 1+R^{\frac1n(\frac{1}{2}-\frac{2}{p}) }  +R^{\frac1n}R^{-\frac{n+1}{2p}-\frac{1}{np}} \big) \|f\|_{L^p}
\end{equation}
for any Schwartz functions $f:\R^n \rightarrow \mathbb{C}$.
\end{theorem}

The Bochner-Riesz estimate for a nondegenerate curve in $\R^2$ was proved by \cite{MR320624}.
For $n \geq 3$, the Bochner-Riesz estimate for a nondegenerate curve in $\mathbb{R}^n$ was first studied by \cite{MR753444, MR772001}. We refer to Section \ref{sec10} for the sharpness of Theorem \ref{2025.04.25.thm11}. Note that there are two critical exponents: $p=4$ and $p=n^2+n-2$. These are precisely the critical exponents of Theorem \ref{0501.thm12}. \\

To prove Theorem \ref{0501.thm12},
we introduce a wave envelope estimate for a cone in $\R^n$ generated by a nondegenerate curve. 
Let $\gamma_n$ be a nondegenerate curve in $\mathbb{R}^n$.
Define the thickened cone $\Gamma_n(\gamma_n;R) \subset \mathbb{R}^n$ generated by $\ga_n$ to be
\begin{equation}
    \Gamma_n(\gamma_n;R):=\Big\{  \sum_{j=1}^{n} \lambda_j \gamma_n^{(j)}(s): s \in [0,1], \;\; |\lambda_i| \leq   R^{-\frac{i-1}{n-1}} , \;\; \frac12 \leq |\lambda_1| \leq 1 \Big\}.
\end{equation}
Write $\Gamma_n(\gamma_n;R) \subset \bigcup_{\theta \in \Xi_n(R)} \theta$ where $\Xi_n(R)$ is the collection of subsets $\theta$ of the form
\begin{equation}\label{caps2}
    \theta:= \Big\{ \sum_{j=1}^{n} \lambda_j \gamma_n^{(j)}(s): \, s \in [0,1] \cap [s_{\theta}, s_{\theta}+2R^{-\frac{1}{n-1}}], \;\; |\lambda_i| \leq    R^{-\frac{i-1}{n-1}}, \;\;  \frac12 \leq |\lambda_1| \leq 1 \Big\}
\end{equation}
and $s_{\theta} \in R^{-\frac{1}{n-1}}\mathbb{Z}$. Note that the cardinality of $\Xi_n(R)$ is comparable to $R^{\frac{1}{n-1}}$.

We split the cone into two sheets according to the sign of \(\lambda_1\).
More precisely, all elements of \(\Xi_n(R)\) are understood to carry an additional
sign label \(\iota\in\{+1,-1\}\), and in the definition of each such element we replace
\(\frac12\le |\lambda_1|\le 1\) by \(\frac12\le \iota\lambda_1\le 1\).
Thus each \(\theta\in\Xi_n(R)\) is a one-sheet piece of the cone. In particular, for
\(\tau\in\Xi_n(s^{-(n-1)})\), the dual box \(\tau^*\) is the dual rectangular box of this
one-sheet piece.
Then $\tau^*$ has dimension $1 \times s^{-1} \times s^{-2} \times \cdots \times s^{-(n-1)}$. Define $U_{\tau,R}$ to be a scaled version of $\tau^*$ with dimension $s^{n-1}R \times s^{n-2}R \times \cdots \times R$, which roughly contains $\theta^*$ for each $\theta \subset \tau$. Let \(\{\psi_\theta\}_{\theta\in\Xi_n(R)}\) be a smooth partition of unity
subordinate to the covering \(\Xi_n(R)\), and define
\(\widehat f_\theta:=\widehat f\,\psi_\theta\). Let $U \| U_{\tau,R}$ be a tiling of $\R^n$ by translates of $U_{\tau,R}$ and call each $U$ a wave envelope. Define the $l^q$-function corresponding to the wave envelope $U$ by
\begin{equation}
    S_{U,q}^{cone}f(x):= \Big(\sum_{\theta \subset \tau}|f_{\theta}(x)|^q w_U(x) \Big)^{\frac1q}.
\end{equation}

Now we state the wave envelope estimate for the cone in $\R^n$.

\begin{theorem}\label{05.03.thm13}  Let $n \geq 3$. For  $4 \leq p \leq n^2-n-2$, $R \geq 1$, and $\epsilon>0$, we have
       \begin{equation*}
        \|f\|_{L^p(\R^n)} \leq C_{\e}R^{\e} R^{\frac{1}{n-1}(\frac12-\frac2p)} \Big(  \sum_{ \substack{ R^{-\frac{1}{n-1}} \leq s \leq 1: \\ dyadic } } \sum_{\tau \in \Xi_n(s^{-(n-1)}) } \sum_{U \| U_{\tau,R}  } |U| \|S_{U,{\frac{p}{2}}}^{cone}f\|_{L^{\frac{p}{2}}(\R^n)}^p \Big)^{\frac1p} 
    \end{equation*}
    for any Schwartz functions $f: \R^n \rightarrow \mathbb{C}$ whose Fourier supports are in $\Gamma_n(\gamma_n;R)$.
\end{theorem}

Note that Theorem \ref{05.03.thm13} was previously known for $n=3$ by \cite{MR4151084}. As mentioned above, Theorem \ref{0501.thm12} and Theorem \ref{05.03.thm13} are closely related. More precisely, Theorem \ref{0501.thm12} for $n-1$ implies Theorem \ref{05.03.thm13} for $n$, and Theorem \ref{05.03.thm13} for $n$ is an ingredient in the proof of Theorem \ref{0501.thm12} for $n$. We refer to Section \ref{0703.sec6} for the discussion.
\\

\subsection{Comparison to previous work.}

The $L^4$ wave envelope estimates were known only for the parabola and the cone in $\mathbb{R}^3$, as established in \cite{MR4151084}. Subsequently, \cite{maldague2022amplitudedependentwaveenvelope} proved an $L^4$ amplitude-dependent wave envelope estimate for the cone in $\mathbb{R}^3$, which was used to derive a small cap decoupling inequality for the cone. This, in turn, implies a small cap decoupling result for the moment curve in $\mathbb{R}^3$ via the work of \cite{guth2022small}. As a side remark, we note that a small cap decoupling for the cone in $\mathbb{R}^3$ has applications in projection theory and the Gauss circle problem; see \cite{MR4663630, li2023improvementgaussscircleproblem} for details. This demonstrates the power of wave envelope analysis, and leads to the question of what other set-ups permit wave envelope interpretations. Since the identification of the wave envelope estimate was a key step towards proving the sharp $L^4$ square function estimate for the cone in $\mathbb{R}^3$, a natural question is to ask if a wave envelope estimate could be used to establish the sharp $L^7$ square function estimate for the moment curve in $\R^3$. Precisely, the question is whether the expression 
\begin{equation}\label{naiveL7}  \int_{\R^3}|\sum_{\theta\in\Theta^3(R)}|f_\theta|^2|^{7/2} \approx  \sum_U |U|\|S_{U,2}^{curve}f\|_{L^7}^7 \end{equation}
holds for some collection of essentially distinct wave envelopes $U$. By \emph{wave envelope}, we mean a set that, after translation to the origin, is comparable to the convex hull of $\theta^*$, for some neighboring subcollection of $\theta\in\Theta^n(R)$. 
It turns out that \eqref{naiveL7} cannot hold for genuine wave envelopes since moment curve planks can stack into tubes, which are larger, convex sets but not wave envelopes. Even if it were possible to identify all of the ways that planks could collect to dominate the right hand side of \eqref{naiveL7}, it is not clear how that would be useful to establish the $L^7$ square function estimate. Part of the argument in \cite{MR4794594} that establishes the sharp $L^7$ estimate uses an iteration inspired by the wave envelope argument of \cite{MR4151084}. However, the ``wave envelopes" that appear in that iteration are actually closer to the wave envelopes which appear in the lower dimensional problem for the parabola. This is due to the fact that working in $L^7$ drastically limits the types of ``high" sets one should think about in the ``high/low'' frequency decomposition. A significant new idea which is exploited in Theorem~\ref{0501.thm12} is to focus on an $\ell^{p/2}$ expression in place of the $\ell^2$ expression in square function estimates, and to simultaneously measure the function in $L^{p/2}$. Compared to the $L^7$ square function estimate in $n=3$, this replaces the critical expression 
\[ \int_{\R^3}|\sum_{\theta\in\Theta^3(R)}|f_\theta|^2|^{7/2} \] 
with the critical expression 
\begin{equation}\label{novel}  \int_{\R^3}|\sum_{\theta\in\Theta^3(R)}|f_\theta|^5|^{2}, \end{equation}
which is an $L^2$ expression within a powerful $L^{10}$-based inequality. By wave packet analysis, the functions $\sum_{\theta\in\Theta^3(R)}|f_\theta|^2$ and $\sum_{\theta\in\Theta^3(R)}|f_\theta|^5$ have essentially the same geometric behavior on the spatial side and comparable Fourier supports. The fact that the expression \eqref{novel} is $L^2$-based allows for a much more generous choice of high-frequency part in the high/low analysis, and leads to a genuine wave envelope estimate for moment curve planks in $\mathbb{R}^3$ (as well as higher dimensions).

\medskip

\section{Reduction of the local smoothing estimate to the main estimate}\label{sec2}

In this section, we carry out a Fourier-side reduction of Theorem \ref{mainthm}
to the main estimate, Theorem \ref{05.29.thm28}. Since the main estimate requires some notation, let us postpone the statement to the end of the section. 
Recall that $\gamma_n:I \rightarrow \mathbb{R}^n$ is a smooth nondegenerate curve.
We fix $n$ and denote $\ga(s)=\ga_n(s)$ to simplify the notation. After an affine change of parameter, we may assume that the range of $s$ is $[0,1]$.
By abusing the notation, we use the notation
\[\ga:=\{ \ga(s): 0\le s\le 1\}\]
to denote the nondegenerate curve, 
and use \[\Ga:=\{ t(\ga(s),1): 1\le t\le 2, \; 0\le s\le 1 \}\]
to denote the two-dimensional cone in $\R^{n+1}$ generated by $\ga$. We remark that $s$ plays the role of the angular parameter, while $t$ plays the role of the radial parameter.

Let $\chi(s,t)$ be a smooth bump function supported in $[0,1]\times [1,2]$.
We can modify the average operator as
\[Af(x,t):=\int_\R f(x-t\ga(s))\chi(s,t) \mathrm{d}s.\] 
We claim that
\begin{equation}\label{verify}
   \int_\R f(x-t\ga(s))\chi(s,t) \mathrm{d}s= (f\de_{t=0})*\mu_\Ga(x,t). 
\end{equation} 
Here, $\mu_\Ga$ is defined so that
\begin{equation}\label{defmu}
    \wh \mu_\Ga(\xi,\xi_{n+1})=\int_{\R^2} e^{-it(\ga(s),1)\cdot (\xi,\xi_{n+1})}\chi(s,t)\mathrm{d}s\mathrm{d}t. 
\end{equation} 
We now verify \eqref{verify}. Take Fourier transform on both sides of \eqref{verify}. The left hand side is 
\begin{equation}
    \begin{split}
        \int &f(x-t\ga(s))\chi(s,t)e^{-i(x\cdot\xi+t\xi_{n+1})}\mathrm{d}s\mathrm{d}x\mathrm{d}t \\&=\int f(x)\chi(s,t)e^{-i x\cdot\xi} e^{-i(t\ga(s)\cdot\xi+t\xi_{n+1})}\mathrm{d}s\mathrm{d}x\mathrm{d}t=\wh f(\xi) \wh \mu_\Ga(\xi,\xi_{n+1})
    \end{split}
\end{equation}
which equals the Fourier transform of the right hand side.

One can compute
\begin{equation}\label{0709.22}
    \begin{split}
        \mu_\Ga(z) &= \int_{\R^{n+1}} \int_{\R^2}e^{-i ( t(\ga(s),1)-z )\cdot (\xi,\xi_{n+1})}\chi(s,t)\mathrm{d}s \mathrm{d}t \mathrm{d}\xi \mathrm{d}\xi_{n+1}
        \\&
        =\int_{\R^2} \de_{z=t(\ga(s),1)}\chi(s,t)\mathrm{d}s\mathrm{d}t.
    \end{split}
\end{equation}
If $\chi(s,t)=1_{[0,1]\times [1,2]}(s,t)$, then the above expression is $1_\Ga$.
So morally speaking, we can think of the measure $\mu_{\Gamma}$ as a ``smooth" surface measure supported on the two-dimensional cone $\Gamma$.

Let us recall  Theorem \ref{mainthm}.
\[ \Big(\int_1^2\|Af(x,t)\|_{L^p_\si(\R^n)}^p \, dt \Big)^\frac1p\lesssim \|f\|_{L^p(\R^n)}. \]
By a standard Littlewood-Paley decomposition and localization argument, it is reduced to proving the following: There exists $\e>0$ such that
\begin{equation}
    \Big(\int_{B_1^n\times [1,2]}|Af(x,t)|^p \, \mathrm{d}x\mathrm{d}t \Big)^{\frac1p} \lesssim R^{-\si-\e}\|f\|_{L^p(W_{B_1^{n}} )},
\end{equation}
for every $f$ satisfying $\supp \wh f\subset B^n_R\setminus B^n_{R/2}$ and $\sigma < \sigma(p,n)$.

The main goal of this section is to compute $\wh\mu_\Ga(\xi,\xi_{n+1})$ for $|\xi|\sim R $. 
We will explicitly compute $\wh{\mu_\Ga}$ (see Theorem \ref{fourierdecay}). This is one of the crucial steps in the proof of Theorem \ref{mainthm}. Based on these decay rates, we partition the frequency space into conical regions and estimate each region separately.

Introduce 
$\bfv_1(s):=\frac{1}{\sqrt{|\gamma(s)|^2+1}}(\gamma(s),1)$,
which is the unit vector in the flat direction of the cone $\Gamma$. Since
$\gamma$ is nondegenerate, the vectors
$\bfv_1(s),\bfv_1'(s),\ldots,\bfv_1^{(n)}(s)$
are linearly independent in $\mathbb R^{n+1}$. Applying the Gram--Schmidt
process to this ordered family gives an orthonormal frame
\[
        \bfv_1(s),\bfv_2(s),\ldots,\bfv_{n+1}(s),
\]
which we call the Frenet frame associated with $\bfv_1$. Equivalently,
\[
        \operatorname{span}\{\bfv_1(s),\ldots,\bfv_j(s)\}
        =
        \operatorname{span}\{\bfv_1(s),\bfv_1'(s),\ldots,\bfv_1^{(j-1)}(s)\},
        \qquad 1\le j\le n+1 .
\]
We shall use the reversed frame
\[
        \bfe_i(s):=\bfv_{n+2-i}(s),\qquad 1\le i\le n+1.
\]
Thus $\bfe_{n+1}(s)$ is parallel to $(\gamma(s),1)$. The properties of the Frenet basis will be discussed in Lemma \ref{0704.lem29} and \ref{Taylorlem}.

We actually view $\{\bfe_1(s):s\in[0,1]\}$ as a curve in the frequency space.
Later, we will use the frame $\bfe_1(s),\dots,\bfe_{n+1}(s)$ to decompose the frequency space. To set up  the scales,  let $R\gg 1$ be the largest scale. We will also use $0<\de_1,\dots,\de_{n-1}\le 1$ to define the intermediate scales. We will discuss a dyadic decomposition at scale $R$.

\subsection{Bump functions adapted to a rectangular box}\label{subsection21}
We introduce a notion used frequently throughout the paper. Let $P=\prod_{i=1}^n[-a_i,a_i]$ be a rectangular box in $\R^n$. We say $\phi_P$ is a bump function adapted to $P$ (or simply $\phi_P$ is adapted to $P$), if
\begin{enumerate}
    \item $\phi_P$ is supported in the $C$-dilation of $P$. Here, $C$ is a fixed constant.
    \item $|\partial^\al\phi_P(x_1,\dots,x_n)|\lesssim_\al \prod_{i=1}^n  a_i^{-\alpha_i} (1+\frac{|x_i|}{a_i})^{-\al_i}$ holds for any multi-index $\al$.
\end{enumerate}

For a general rectangular box $P$, we  define
adapted functions by applying a rigid motion sending P to a box of the above form and then pulling back the definition.

It is easy to see the following facts.
\begin{enumerate}
    \item If $\phi_1,\dots,\phi_m$ are bump functions adapted to $P$, then $\prod_{i=1}^m\phi_i$ is a bump function adapted to $P$.

    \item If $P\subset Q$, and $\phi_P, \phi_Q$ are adapted to $P,Q$ (respectively), then $\phi_P\phi_Q$ is adapted to $P$.
\end{enumerate}
 If $X=\cup P$ is  a finite union of rectangular boxes, we say $\phi_X$ {is adapted to} $X$ if it can be written as  $\phi_X=\sum_{P\subset X} \phi_P$ where each $\phi_P$ is adapted to $P$.

\subsection{Dyadic decomposition for a hollow rectangular box}
The goal of this subsection is to suitably decompose the hollow rectangular
box
\begin{equation}\label{0704.142}
   [-1,1]^{n-1} \setminus  [-\frac 12,\frac 12]^{n-1}
\end{equation}
into anisotropic sub-boxes. Such decompositions have already appeared in \cite{gan2022restricted} (for $n=3$) and \cite{gan2024restricted} (for general $n$), so we will sketch the idea. The two important features of the sub-boxes in the decomposition are the size (corresponding to the parameter $\vec\de$) and the separation from the origin (corresponding to the parameter $\vec\nu$).

Let us describe the size of the sub-boxes. Fix $R\gg 1$.
Choose a sequence of dyadic
numbers
\[
    \vec\de=(\de_1,\ldots,\de_{n-1}), \qquad 0<\de_i\le 1,
\]
which satisfies the numerology
\begin{equation}\label{admissible-numerology}
    \de_1^n\de_2^{n-1}\cdots \de_{n-1}^2=R^{-1}.
\end{equation}

It is useful to introduce the associated $\rho$-sequence
\begin{equation}\label{rhoseq}
    \rho_k:=\de_1^k\de_2^{k-1}\cdots\de_k,
    \qquad 1\le k\le n-1,
\end{equation}
and also set 
$\rho_0:=1 $ and 
\[\rho_n:=\de_1^n\de_2^{n-1}\cdots\de_{n-1}^2=R^{-1}.\]
The sub-boxes in the decomposition will be of size
comparable to
\[
    \rho_0\times \rho_1\times \ldots\times \rho_{n-2}.
\]

\begin{definition}[Admissible $(\vec\de,\vec\nu)$]\label{admissibledenu}
Let $\vec\de=(\de_1,\ldots,\de_{n-1})$ be a sequence of dyadic numbers, and
$\vec\nu=(\nu_1,\ldots,\nu_{n-1})\in\{-1,0,1\}^{n-1}$.  We say that
$(\vec\de,\vec\nu)$ is \textbf{admissible} at scale $R$ if \eqref{admissible-numerology} holds, and $\nu_i=0\Leftrightarrow \de_i=1$. Equivalently, a nontrivial scale
$\de_i<1$ records that the $i$th coordinate lies in one of the two outer
dyadic annuli, while $\nu_i=0$ means that this coordinate stays in the
central part.
\end{definition}

We have the following crucial decomposition lemma.

\begin{lemma}
    There exists a decomposition
    \begin{equation}\label{decomphollow}
    [-1,1]^{n-1}\Big\backslash [-\frac{1}{2},\frac{1}{2}]^{n-1}=\bigsqcup_{\vec\de,\vec\nu}s[\vec\de,\vec\nu]. 
\end{equation} 
Here, $(\vec\de,\vec\nu)$ are admissible at scale $R$; each $s[\vec\de,\vec\nu]$ is the collection of points $(a_1,\dots,a_{n-1})\in \R^{n-1}$ such that (recalling \eqref{rhoseq})
\begin{equation}\label{deltaplank00}
\begin{split}
\begin{cases}
    |a_i|\lesssim \rho_{i-1} & \textup{if~}\nu_i=0,\\
    a_i\sim \rho_{i-1} & \textup{if~} \nu_i=1,\\
    a_i\sim -\rho_{i-1} &\textup{if~} \nu_i=-1,
\end{cases}\qquad 1\le i\le n-1.\\
\end{split}
\end{equation}
($a_i\sim \pm \rho_{i-1}$ means that $a_i$ lies in a
fixed dyadic interval of that size and with the indicated sign.) 

\end{lemma}

\begin{proof}
We prove the following slightly stronger statement. Let
\(0<\ell_i<1\), \(1\leq i\leq n-1\). Then
\[
[-1,1]^{n-1}\setminus \prod_{i=1}^{n-1}[-\ell_i,\ell_i]
=\bigsqcup_{\vec\delta,\vec\nu} s[\vec\delta,\vec\nu],
\]
where the pieces have the form in \((3.10)\), with implicit constants depending
only on the \(\ell_i\)'s. Taking \(\ell_i=1/2\) gives the lemma.

We argue by induction on \(n\). When \(n=2\),
\[
[-1,1]\setminus[-\ell_1,\ell_1]
=[-1,-\ell_1)\sqcup(\ell_1,1].
\]
These are the two pieces corresponding to
\(\delta_1=R^{-1/2}\) and \(\nu_1=-1,1\), respectively.

Assume the result is known in dimension \(n-1\). Decompose
\[
[-1,1]^{n-1}\setminus \prod_{i=1}^{n-1}[-\ell_i,\ell_i]
\]
into the following three disjoint parts:
\begin{enumerate}
    \item $\ [-1,1]\times\Big([-1,1]^{n-2}\setminus
\prod_{i=2}^{n-1}[-\ell_i,\ell_i]\Big)$,

\item
$(\ell_1,1]\times \prod_{i=2}^{n-1}[-\ell_i,\ell_i]$,

\item $[-1,-\ell_1)\times \prod_{i=2}^{n-1}[-\ell_i,\ell_i]$.

\end{enumerate}

For part (1), apply the induction hypothesis to the last \(n-2\)
coordinates, so that this portion is decomposed into pieces $\bigsqcup s[\vec\tau,\vec\om]$. For each admissible pair \((\vec\tau,\vec\omega)\),
set
\[
\vec\delta=(1,\vec\tau),\qquad \vec\nu=(0,\vec\omega).
\]
Then the resulting $s[\vec\de,\vec\nu]=[-1,1]\times s[\vec\tau,\vec\omega]$ form a decomposition.

We next treat part (2); part (3) is identical. Decompose
\[
\prod_{i=2}^{n-1}[-\ell_i,\ell_i]
\]
into dyadic anisotropic annuli
\[
\prod_{i=2}^{n-1}[-\ell_i(2\delta_1)^{i-1},
 \ell_i(2\delta_1)^{i-1}]
\setminus
\prod_{i=2}^{n-1}[-\ell_i\delta_1^{i-1},
 \ell_i\delta_1^{i-1}],
\]
where \(R^{-1/n}\lesssim \delta_1\leq 1/2\) is dyadic, together with the
remaining central box
\[
\prod_{i=2}^{n-1}[-\ell_i R^{-(i-1)/n},\ell_i R^{-(i-1)/n}].
\]
For a fixed \(\delta_1\), rescale the last \(n-2\) coordinates by
\(a_i=\delta_1^{i-1}b_i\). Applying the induction hypothesis at scale
\(R\delta_1^n\) gives admissible \((n-2)\)-tuples
\(\vec\tau=(\delta_2,\ldots,\delta_{n-1})\), satisfying
\[
\delta_2^{\,n-1}\delta_3^{\,n-2}\cdots \delta_{n-1}^2
=(R\delta_1^n)^{-1}.
\]
Hence
\[
\delta_1^n\delta_2^{\,n-1}\cdots \delta_{n-1}^2=R^{-1},
\]
so \((\delta_1,\vec\tau)\) satisfies \eqref{admissible-numerology}. After undoing the
rescaling, the corresponding pieces are exactly of the form
\[
s[(\delta_1,\vec\tau),(1,\vec\omega)].
\]
The remaining central box is the piece with
\[
\vec\delta=(R^{-1/n},1,\ldots,1),\qquad
\vec\nu=(1,0,\ldots,0).
\]
Thus the part (2) is decomposed into admissible pieces. 
\end{proof}

\bigskip

Rescaling \eqref{decomphollow} by factor $R$, we have
\[ [-R,R]^{n-1}\Big\backslash [-\frac{R}{2},\frac{R}{2}]^{n-1}=\bigsqcup_{\vec\de,\vec\nu}S[\vec\de,\vec\nu]. \]
Here, each $S[\vec\de,\vec\nu]$ is the collection of points $(a_1,\dots,a_{n-1})\in\R^{n-1}$ such that
\begin{equation}\label{deltaplank0}
\begin{split}
\begin{cases}
    |a_i|\lesssim R\rho_{i-1} & \textup{if~}\nu_i=0,\\
    a_i\sim R\rho_{i-1} & \textup{if~} \nu_i=1,\\
    a_i\sim -R\rho_{i-1} &\textup{if~} \nu_i=-1,
\end{cases}\qquad 1\le i\le n-1,\\
\end{split}
\end{equation}
Also define
$P[\vec\de,\vec\nu]$ to be the collection of points $(a_1,\dots,a_{n+1})\in\R^{n+1}$ such that
\begin{equation}\label{deltaplank}
\begin{split}
\begin{cases}
    |a_i|\lesssim R\rho_{i-1} & \textup{if~}\nu_i=0,\\
    a_i\sim R\rho_{i-1} & \textup{if~} \nu_i=1,\\
    a_i\sim -R\rho_{i-1} &\textup{if~} \nu_i=-1,
\end{cases}\qquad 1\le i\le n-1,\\
\textup{and~}|a_n|\lesssim R\rho_{n-1},
\qquad |a_{n+1}|\lesssim R\rho_n=1.
\end{split}
\end{equation}
Here, the first $n-1$ coordinates of $P[\vec\de,\vec\nu]$ coincide with $S[\vec\de,\vec\nu]$.

\begin{definition}
    Let $s\in[0,1]$. Suppose $X\subset \R^{n+1}$. We define
    \[ X(s):=\Big\{\sum_{i=1}^{n+1} a_i\bfe_i(s): (a_1,\dots,a_{n+1})\in X\Big\}. \]
    In most cases, $X$ is a rectangular box. For example,
    if $\bI=I_1\times \dots\times I_{n+1}\subset \R^{n+1}$, where each $I_i\subset \R$ is an interval. Then
    \[ \bI(s):=\Big\{\sum_{i=1}^{n+1} a_i\bfe_i(s): a_i\in I_i\Big\}. \]
\end{definition}
By changing coordinates, we obtain the $(\vec\de,\vec\nu)$-plank $P[\vec\de,\vec\nu](s)$ that is associated with the Frenet coordinate. More precisely, we define
\begin{equation}
    P[\vec\de,\vec\nu](s):= \Big\{ \sum_{i=1}^{n+1}a_i \bfe_i(s): (a_1,\ldots,a_{n+1}) \in P[\vec\de, \vec\nu]   \Big\}.
\end{equation}

\begin{lemma}[Frenet geometry]\label{0704.lem29}
    Recall the Frenet coordinate $\bfe_1(s),\dots,\bfe_{n+1}(s)$. Define
    the nonzero functions
    \[ \kappa_j(s):=\langle \bfe_j'(s),\bfe_{j+1}(s)\rangle \ \textup{for~}j=1,\dots,n. \]
    Then one has the formula
    \begin{equation}\label{frenetformula}
        \begin{split}
            &\bfe_1'(s)=\kappa_1(s)\bfe_2(s),\\
            &\bfe_i'(s)=-\kappa_{i-1}(s)\bfe_{i-1}(s)+\kappa_i(s)\bfe_{i+1}(s),\ i=2,\dots,n,\\
            &\bfe'_{n+1}(s)=-\kappa_n(s)\bfe_n(s).
        \end{split}
    \end{equation}
Since the curve is fixed, we assume $|\ka_i(s)|\sim 1$ throughout the paper.
\end{lemma}

Lemma \ref{0704.lem29} is well-known, so we omit the proof here.

\begin{remark}
    {\rm
    Recall that the geometric objects we want to study are boxes of the form \eqref{caps2}. They are expressed using the basis $\{\ga'_n,\dots,\ga_n^{(n)}\}$. We may replace the basis by $\{\bfe_1,\dots,\bfe_n\}$, the Frenet coordinate for the curve $\ga_n$. Noting that the coefficients $\lam_j$ are decreasing, one can show that the original box is comparable to the resulting box, which is comparable to a $(\vec\de,\vec\nu)$-plank. This type of argument appeared in  previous work (see, for example, \cite[Lemma 8.2]{MR4861588}).
    }
\end{remark}

\medskip

\begin{lemma}[Taylor expansion] \label{Taylorlem}
    For $|\De|\ll 1$ and $1\le i<k\le n+1$, we have
    \begin{equation}\label{taylorexpansion}
        \langle \bfe_i(s+\De),\bfe_k(s)\rangle= \frac{\De^{k-i}}{(k-i)!}\kappa_i(s)\cdots\kappa_{k-1}(s)+O(|\De|^{k-i+1}). 
    \end{equation} 
Similarly, for $1\le k<i\le n+1$, we have
    \begin{equation}\label{taylorexpansion2}
        \langle \bfe_i(s+\De),\bfe_k(s)\rangle= \frac{(-\De)^{i-k}}{(i-k)!}\kappa_k(s)\cdots\kappa_{i-1}(s)+O(|\De|^{i-k+1}). 
    \end{equation} 
Also,     
\begin{equation}\label{taylorexpansion3}
        \langle \bfe_k(s+\De),\bfe_k(s)\rangle= 1+O(|\De|^{2}). 
    \end{equation}
\end{lemma}
\begin{proof}
Let us prove \eqref{taylorexpansion}.
    We use  Taylor's theorem
    \[ \bfe_i(s+\De)=\sum_{j=0}^{k-i} \frac{\De^j}{j!}\bfe_i^{(j)}(s)+O(|\De|^{k-i+1}). \]
Using \eqref{frenetformula}, we see that $\bfe_i^{(j)}(s)$ is a linear combination of $\{\bfe_l(s)\}_{1\le l\le i+j}$. Since $\bfe_k(s)$ is orthogonal to $\bfe_l(s)$ for $l<k$, we have
\[ \langle \sum_{j=1}^{k-i-1} \frac{\De^j}{j!}\bfe_i^{(j)}(s),\bfe_k(s)\rangle=0. \]
Writing $\bfe_i^{(k-i)}(s)$ as a linear combination of $\{\bfe_l(s)\}_{1\le l\le k}$ and noting that the coefficient of $\bfe_k(s)$ is $\ka_i(s)\cdots\ka_{k-1}(s)$, we have
\[ \langle  \frac{\De^{k-i}}{(k-i)!}\bfe_i^{(k-i)}(s),\bfe_k(s)\rangle=\frac{\De^{k-i}}{(k-i)!}\kappa_i(s)\cdots\kappa_{k-1}(s). \]
This finishes the proof of \eqref{taylorexpansion}. The proofs of the other two identities are similar.
\end{proof}

As a next step,
we will exploit the geometric positions of the $(\vec\de,\vec\nu)$-planks. 
We will show that if $s,s'$ are close, then they are \textit{comparable}. To define ``comparable", we need the following anisotropic dilation which enlarges the plank  while preserving the property that, whenever $\nu_i \neq 0$, the $i$th edge remains separated from the origin.
\begin{definition}[Anisotropic dilation of $(\vec\de,\vec\nu)$-planks]\label{dedilation}
Let $0< C<\infty$. Recall $P[\vec\de,\vec\nu]$ in \eqref{deltaplank}.
We define the $C$-dilation of $P[\vec\de,\vec\nu](s)$, denoted by $CP[\vec\de,\vec\nu](s)$ to be a plank consisting of points
 $\sum_{i=1}^{n+1}a_i\bfe_i(s)$,
where
\begin{equation} 
\begin{split}
\begin{cases}
    |a_i|\le CR\rho_{i-1} & \textup{if~}\nu_i=0,\\
    a_i\in (\frac{1}{C+1}R\rho_{i-1},(1+C) R\rho_{i-1}] & \textup{if~} \nu_i=1,\\
    a_i\in [-(1+C) R\rho_{i-1},-\frac{1}{C+1}R\rho_{i-1}) &\textup{if~} \nu_i=-1;
\end{cases}\\
\textup{and~}|a_n|\le  CR\rho_{n-1};
|a_{n+1}|\le  C.
\end{split}
\end{equation}
    
\end{definition}

\begin{definition}[Isotropic dilation]\label{050704.212}
For a box $U$ and a number $C>0$,
     define the isotropic dilation $\textup{Dil}_C(U)=CU$, which is the dilation with respect to the center of $U$ by a factor of $C$. 
\end{definition}

Throughout the paper, unless otherwise specified, when we place a constant \(C\)
in front of a \((\vec\delta,\vec\nu)\)-plank or a similar object, we mean the
anisotropic dilation in the sense of Definition \ref{dedilation}; otherwise, we mean the
isotropic dilation.

\begin{definition}\label{defcompa}
    We say a set $U$ is comparable to $P[\vec\de,\vec\nu](s)$ if
    \[ C^{-1}P[\vec\de,\vec\nu](s)\subset U\subset CP[\vec\de,\vec\nu](s). \]
    Here, $C$ is a universal constant.
\end{definition}

The following lemma tells us that if we slightly change the basis, the resulting $(\vec\de,\vec\nu)$-planks are morally the same.

\begin{lemma}\label{structurelem1} There exists a constant $K\gg 1$ independent of the parameter $R$ such that the following holds.
Let $(\vec\de,\vec\nu)$ be admissible at scale $R$. Let $s,s'\in[0,1]$. If $|s-s'|\le K^{-1}\prod_{i=1}^{n-1}\de_i$, then $P[\vec\de,\vec\nu](s)$ and $P[\vec\de,\vec\nu](s')$ are comparable.
\end{lemma}

\begin{proof}
Let 
\[
    \delta:=\prod_{j=1}^{n-1}\delta_j
\]
and write $s'=s+\Delta$, with $|\Delta|\leq K^{-1}\delta$.  We prove that
$P[\vec\delta,\vec\nu](s')\subset C P[\vec\delta,\vec\nu](s)$; the reverse inclusion is identical.

Set
\[
    L_i:=R\rho_{i-1}\quad (1\leq i\leq n),\qquad L_{n+1}:=1=R\rho_n,
\]
where
$\rho_i$ is given by \eqref{rhoseq}. Let
\begin{equation}\label{zformula}
    z=\sum_{i=1}^{n+1}a_i' \bfe_i(s')\in P[\vec\delta,\vec\nu](s').
\end{equation}
    
Thus $|a_i'|\lesssim L_i$ for all $i$, and for $1\leq i\leq n-1$ with
$\nu_i\neq 0$, the coefficient $a_i'$ has the prescribed sign and satisfies
$|a_i'|\sim L_i$.

Write $z$ in the basis at $s$:
\[
    z=\sum_{k=1}^{n+1}a_k \bfe_k(s),
    \qquad 
    a_k=\langle z,\bfe_k(s)\rangle .
\]
By Lemma~\ref{Taylorlem}, for $i\neq k$ we have
\[
    |\langle \bfe_i(s+\Delta),\bfe_k(s)\rangle|
    \lesssim |\Delta|^{|i-k|},
\]
while
\[
    \langle \bfe_k(s+\Delta),\bfe_k(s)\rangle=1+O(|\Delta|^2).
\]
Hence, using \eqref{zformula},
\[
    |a'_k-a_k|=\langle z,\bfe_k(s+\De)-\bfe_k(s)\rangle
    \lesssim |a_k'||\Delta|^2
      +\sum_{i\neq k} |a_i'| |\Delta|^{|i-k|}.
\]
We claim that every term on the right is $O(K^{-1}L_k)$.  Indeed, if $i<k$, then by the definition of the $\rho$-sequence,
\[
    L_i\delta^{k-i}\leq L_k.
\]
If $i>k$, then $L_i\leq L_k$.  Since $|\Delta|\leq K^{-1}\delta\leq K^{-1}$, we obtain
\[
    |a_i'||\Delta|^{|i-k|}
    \lesssim K^{-1}L_k
\]
for every $i\neq k$.  The diagonal error $|a_k'||\Delta|^2$ is also
$\lesssim K^{-1}L_k$.  Therefore,
\[
    |a'_k-a_k|\lesssim K^{-1}L_k .
\]
Choosing $K$ sufficiently large, depending only on the curve and the fixed implicit
constants in the definition of the planks, this error is small compared with $L_k$.

Consequently, for $1\leq k\leq n-1$, if $\nu_k=0$, then $|a_k|\lesssim L_k$; if
$\nu_k=\pm1$, then $a_k$ has the same sign as $a'_k$ and $|a_k|\sim L_k$.
For $k=n,n+1$, we also have $|a_k|\lesssim L_k$.  Thus
\[
    z\in C P[\vec\delta,\vec\nu](s).
\]
This proves $P[\vec\delta,\vec\nu](s')\subset C P[\vec\delta,\vec\nu](s)$. 
\end{proof}

We first prove an auxiliary disjointness lemma, which will imply the finite-overlap estimate.

\begin{lemma}\label{disjlem}
    For any $C\ge 1$, there exists $K> 1$ so that the following holds. Fix an admissible $(\vec\de,\vec\nu)$. 
    For all $1\le l\le n-1$, and all $s,s'$ satisfying
    $K\prod_{i=1}^{l}\de_i\le |s-s'|\le K^{-1} \prod_{i=1}^{l-1}\de_i$,  $CP[\vec\de,\vec\nu](s)$ and $CP[\vec\de,\vec\nu](s')$ are disjoint.
\end{lemma}

\begin{proof}
We may assume $C=1$; the general case can be handled similarly by choosing  a larger $K$. We argue by contradiction. Suppose that
\[
z\in P[\vec \delta,\vec\nu](s)\cap P[\vec \delta,\vec\nu](s').
\]
Write $s'=s+\Delta$, where
\[
K\prod_{i=1}^l\delta_i\le |\Delta|\le K^{-1}\prod_{i=1}^{l-1}\delta_i .
\]
If $\delta_l=1$, then since $K>1$, the above range is empty. Hence we may
assume $\delta_l<1$. By admissibility, $\nu_l=\pm1$, and therefore the $l$-th
coefficient of $z$ in the frame at $s'$ satisfies
\[
|a_l|\sim R\rho_{l-1}.
\]
Write
\[
z=\sum_{i=1}^{n+1} a_i \bfe_i(s').
\]

We look at the $(l+1)$-th coordinate of $z$ in the frame at $s$, namely
$\langle z,\bfe_{l+1}(s)\rangle$. Since $z\in P[\vec \delta,\vec\nu](s)$, we have
\[
|\langle z,\bfe_{l+1}(s)\rangle|\lesssim R\rho_l .
\]
On the other hand, expanding the expression using Lemma~\ref{Taylorlem}, we get
\[
 \langle z,\bfe_{l+1}(s)\rangle=\sum_{i=1}^{l}\bigg(a_i \frac{\De^{l+1-i}}{(l+1-i)!}\kappa_i(s)\cdots\kappa_{l}(s)+a_i O(\De^{l+2-i})\bigg)+O(R\rho_{l}). 
\]
Here the terms with $i\ge l+1$ contribute $O(R\rho_l)$.

We claim that the terms with $i<l$ are negligible compared with $a_l\Delta$.
Indeed, using
\[
|\Delta|\le K^{-1}\prod_{j=1}^{l-1}\delta_j,
\]
we have, for $i<l$,
\[
|a_i||\Delta|^{l+1-i}
\le
K^{-(l-i)}
R\rho_{i-1}
\Big(\prod_{j=1}^{l-1}\delta_j\Big)^{l-i}
|\Delta|
\ll
R\rho_{l-1}|\Delta|
\sim |a_l||\Delta|,
\]
provided $K$ is sufficiently large. Since $|\kappa_l(s)|\sim1$, the dominant term
is $a_l\Delta\kappa_l(s)$. Hence
\[
|a_l||\Delta|\lesssim R\rho_l .
\]
Using $|a_l|\sim R\rho_{l-1}$, we obtain
\[
|\Delta|\lesssim \frac{\rho_l}{\rho_{l-1}}
=
\prod_{i=1}^l\delta_i .
\]
Choosing $K$ larger than the implicit constant contradicts
\[
|\Delta|\ge K\prod_{i=1}^l\delta_i .
\]
Therefore the two planks are disjoint.
\end{proof}

Next, we prove a finite-overlap lemma.

\begin{lemma}\label{finoverlem} 
Let $C\ge 1$.
Let $(\vec\de,\vec\nu)$ be admissible at scale $R$. Let $\de=\prod_{i=1}^{n-1}\de_i$, and $\I_\de$ be a set of $\de$-separated points in $[0,1]$. Then $\{CP[\vec\de,\vec\nu](s)\}_{s\in\I_\de}$ are $O_C(1)$-overlapping. In other words,
\[ \Big\|\sum_{s\in\I_\de}1_{CP[\vec\de,\vec\nu](s)} \Big\|_\infty \lesssim_C 1. \]
\end{lemma}

\begin{proof}
Let $K$ be the constant in Lemma \ref{disjlem}, chosen for the dilation constant $C$.
Write
\[
    P(s):=CP[\vec\de,\vec\nu](s).
\]
Set
\[
    \beta_1:=1,\qquad 
    \beta_l:=\prod_{i=1}^{l-1}\de_i \quad (2\le l\le n).
\]
Thus $\beta_n=\de$.

For $0<\beta\le 1$, let $\cJ_\beta$ be a finitely overlapping cover of $[0,1]$
by intervals of length $\beta$.  If $J\subset [0,1]$, write
\[
    \I_\de(J):=\I_\de\cap J.
\]
We claim that for each $1\le l\le n-1$,
\begin{equation}\label{multiscale-overlap-step}
    \sup_{J\in \cJ_{K^{-1}\beta_l}}
    \Big\|\sum_{s\in \I_\de(J)}1_{P(s)}\Big\|_\infty
    \lesssim_C
    \sup_{J'\in \cJ_{K\beta_{l+1}}}
    \Big\|\sum_{s\in \I_\de(J')}1_{P(s)}\Big\|_\infty .
\end{equation}
Indeed, fix $J\in\cJ_{K^{-1}\beta_l}$ and decompose $J$ into intervals of
length $K\beta_{l+1}$.  If two such intervals are not neighboring, then any two
parameters $s,s'$ chosen from them satisfy
\[
    K\beta_{l+1}\le |s-s'|\le K^{-1}\beta_l .
\]
By Lemma \ref{disjlem}, the corresponding planks $P(s)$ and $P(s')$ are disjoint.
Hence, at any fixed point, only $O_C(1)$ neighboring intervals can contribute.
This proves \eqref{multiscale-overlap-step}.

Now we iterate this estimate. First, by the triangle inequality,
\[
    \Big\|\sum_{s\in\I_\de}1_{P(s)}\Big\|_\infty
    \lesssim_C
    \sup_{J\in \cJ_{K^{-1}\beta_1}}
    \Big\|\sum_{s\in \I_\de(J)}1_{P(s)}\Big\|_\infty .
\]
Applying \eqref{multiscale-overlap-step} successively for
$l=1,\ldots,n-1$, and using only the fact that $K$ is fixed once $C$ is fixed,
we obtain
\[
    \Big\|\sum_{s\in\I_\de}1_{P(s)}\Big\|_\infty
    \lesssim_C
    \sup_{J\in \cJ_{K\beta_n}}
    \Big\|\sum_{s\in \I_\de(J)}1_{P(s)}\Big\|_\infty .
\]
Since $\beta_n=\de$ and $\I_\de$ is $\de$-separated, each interval
$J\in\cJ_{K\de}$ contains only $O_C(1)$ points of $\I_\de$. Therefore
\[
    \sup_{J\in \cJ_{K\de}}
    \Big\|\sum_{s\in \I_\de(J)}1_{P(s)}\Big\|_\infty
    \lesssim_C 1.
\]
This proves
\[
    \Big\|\sum_{s\in\I_\de}1_{CP[\vec\de,\vec\nu](s)}\Big\|_\infty
    \lesssim_C 1,
\]
as desired.
\end{proof}

\medskip

The next lemma compares $(\vec\de,\vec\nu)$-planks at different scales. Recall Definition \ref{050704.212} of isotropic dilation.

\begin{lemma}\label{conv}
There exist $K, C\ge 1$ such that the following holds.
    Let $(\vec\de,\vec\nu)$ be admissible at scale $R$, and $\prod_{i=1}^{n-1}\de_i\le \si <1$. 
    Then there exists a pair $(\vec\de',\vec\nu')$ that is admissible at scale $r\le R$, such that $\prod_{i=1}^{n-1}\de_i'=\si$ and
    \[  \bigcup_{s':|s'-s|\le K^{-1}\si}P[\vec\de,\vec\nu](s')  \subset \textup{Dil}_\frac{R}{r} \big(CP[\vec\de',\vec\nu'](s)\big)  \]
    for any $s\in[0,1]$. Here, $P[\vec\de',\vec\nu']$ is a plank with dimensions $r\times r\rho'_{1}\times\dots\times r\rho'_{n-1}\times 1$ as defined in \eqref{deltaplank}, where $\rho_i'=(\de'_1)^i(\de_2')^{i-1} \cdots\de'_i$ and $r=(\de_1')^{-n}(\de_2')^{-n+1}\cdots (\de'_{n-1})^{-2}$.
\end{lemma}

\begin{proof}
Let
\[
        \de:=\prod_{i=1}^{n-1}\de_i .
\]
Since $\de\le \si<1$, we can choose an index $1\le l\le n-1$ such that
\[
        \de_1\cdots \de_l \le \si < \de_1\cdots \de_{l-1}.
\]
Set
\[
        \de_l'=\frac{\si}{\de_1\cdots \de_{l-1}},
        \qquad
        \vec\de'=(\de_1,\ldots,\de_{l-1},\de_l',1,\ldots,1),
\]
and
\[
        \vec\nu'=(\nu_1,\ldots,\nu_l,0,\ldots,0).
\]
Then $\de_l\le \de_l'<1$, and hence $(\vec\de',\vec\nu')$ is admissible at
 scale
\[
        r=(\de_1')^{-n}(\de_2')^{-(n-1)}\cdots(\de_{n-1}')^{-2}.
\]
Moreover,
\[
        \prod_{i=1}^{n-1}\de_i'=\si,
        \qquad
        r\le R.
\]

We first compare the planks with the same angular parameter. By the choice of
$\vec\de'$, the side lengths of $P[\vec\de,\vec\nu](s')$ are no larger than the
corresponding side lengths of
\[
        \textup{Dil}_{R/r}\big(P[\vec\de',\vec\nu'](s')\big).
\]
Indeed, for the first $l$ coordinates the signs are unchanged, while for the
remaining coordinates we have set $\nu_i'=0$, so the enlarged plank contains the
original signed pieces. Thus, after increasing the implicit constant if necessary,
\[
        P[\vec\de,\vec\nu](s')
        \subset
        \textup{Dil}_{R/r}\big(C'P[\vec\de',\vec\nu'](s')\big).
\]

Now assume $|s'-s|\le K^{-1}\si$. Since
\[
        \prod_{i=1}^{n-1}\de_i'=\si,
\]
Lemma~\ref{structurelem1} applied to the admissible pair
$(\vec\de',\vec\nu')$ gives
\[
        C'P[\vec\de',\vec\nu'](s')
        \subset
        CP[\vec\de',\vec\nu'](s).
\]
Combining the two inclusions gives
\[
        P[\vec\de,\vec\nu](s')
        \subset
        \textup{Dil}_{R/r}\big(CP[\vec\de',\vec\nu'](s)\big)
\]
whenever $|s'-s|\le K^{-1}\si$. Taking the union over such $s'$ proves the
lemma.
\end{proof}

\bigskip

\subsection{Computing the decay rate}
The goal of this subsection is to obtain an expression for $\wh\mu_\Ga (\xi,\xi_{n+1})$ for $|\xi|\sim R$. Recall that
\[ \wh \mu_\Ga(\xi,\xi_{n+1})=\int_{[0,1]\times [1,2] } e^{-it(\ga(s),1)\cdot (\xi,\xi_{n+1})}\chi(s,t)\mathrm{d}s\mathrm{d}t, \]
where $\chi(s,t)$ is a smooth bump function supported in $[0,1]\times [1,2]$.
The following algorithm is designed to compute the function $\widehat{\mu}_{\Gamma}$.

Denote by \(\mathcal D\)  the set of dyadic \((n-1)\)-tuples \(\vec\delta\)
satisfying the numerology \eqref{admissible-numerology}.
Fix a sequence $R^{-\frac12}=\De_1<\De_2<\dots<\De_{M+1}=R^{-\frac1n}$. Here, $\De_k=R^{-\frac12+(\frac12-\frac1n)\frac{k-1}{M}}$ and  $M$ is a sufficiently large integer, which will be determined later\footnote{We use the fact that $M$ is sufficiently large several times; we use it in the proof of Lemma \ref{cancellation} and also in Subsection \ref{subsection710}}.  

\bigskip

\begin{definition}
    Define $\De(\vec\de):=\prod_{i=1}^{n-1}\de_i$.
\end{definition}

 One can see that $\vec\de_{\text{min}}:=(1,\dots,1,R^{-1/2})$ and $\vec\de_{\text{max}}:=(R^{-1/n},\dots,1)$ satisfy
\[\inf_{\vec\de\in\cD}\De(\vec\de)=\De(\vec\de_{\text{min}})=R^{-1/2},\ \ \sup_{\vec\de\in\cD}\De(\vec\de)=\De(\vec\de_{\text{max}})=R^{-1/n}.\]

\begin{definition}

Define
\[
\mathcal D_1:=\{\vec\delta\in\mathcal D:\Delta(\vec\delta)\in[\Delta_1,\Delta_2]\},
\]
and for \(2\le k\le M\), define
\[
\mathcal D_k:=\{\vec\delta\in\mathcal D:\Delta(\vec\delta)\in(\Delta_k,\Delta_{k+1}]\}.
\]

\end{definition}

\bigskip

\begin{definition}[Cluster of planks]\label{Cluster}
For $1\le k\le M$, define
    \begin{equation}
    Q_k =\bigsqcup_{\vec\de\in \sqcup_{k\le j\le M}\cD_j } \Big( \bigsqcup_{\vec\nu:(\vec\de, \vec\nu) \textrm{admissible} }  S[\vec\de,\vec\nu] \Big). 
\end{equation}
\end{definition}
We have the nested inclusions $Q_1\supset Q_2\supset\dots\supset Q_M$ as shown in Figure \ref{figQk}. $Q_1=[-R,R]^{n-1}\setminus [-\frac{R}{2},\frac{R}{2}]^{n-1}$. As $k$ increases, $Q_k$ becomes smaller.

\begin{figure}[ht]
\centering
\begin{tikzpicture}[x=1cm,y=1cm,line join=round,line cap=round]
  \definecolor{qone}{RGB}{93,170,245}
  \definecolor{qk}{RGB}{160,112,70}
  \definecolor{qkp}{RGB}{180,60,255}

  \draw[qone,line width=.55pt] (0,0) rectangle (5.15,5.15);
  \node[qone] at (4.85,4.70) {$Q_1$};

  \draw[qk,line width=.55pt] (2.10,3.25) rectangle (3.10,4.82);
  \draw[qk,line width=.55pt] (2.10,.32) rectangle (3.10,1.90);
  \draw[qk,line width=.55pt] (.42,2.15) rectangle (2.35,3.15);
  \draw[qk,line width=.55pt] (2.90,2.15) rectangle (4.75,3.15);
  \node[qk] at (4.50,2.88) {$Q_k$};

  \draw[qkp,line width=.55pt] (2.42,3.25) rectangle (2.78,4.38);
  \draw[qkp,line width=.55pt] (2.42,.80) rectangle (2.78,1.90);
  \draw[qkp,line width=.55pt] (1.04,2.40) rectangle (2.00,2.90);
  \draw[qkp,line width=.55pt] (3.15,2.40) rectangle (3.95,2.90);
  \node[qkp,scale=.55] at (4.02,2.54) {$Q_{k+1}$};

  \draw[black,line width=.55pt,fill=white] (1.42,1.42) rectangle (3.72,3.72);
  \node at (2.57,2.57)
    {$\left[-\frac{R}{2},\frac{R}{2}\right]^{n-1}$};
\end{tikzpicture}
\caption{Nested clusters $Q_1\supset Q_k\supset Q_{k+1}$ outside the central box $[-\frac{R}{2},\frac{R}{2}]^{n-1}$.}
\label{figQk}
\end{figure}

\begin{definition}[Thickened planks]\label{thickplank}
We define the cluster of rectangular boxes thickened in the $n$-th and $(n+1)$-st directions.
    \[ Q_k[L,J]:=Q_k\times [-L,L]\times [-J,J]. \]
    Similarly, we define
    \[ P[\vec\de,\vec\nu;L,J]:=S[\vec\de,\vec\nu]\times [-L,L]\times [-J,J]. \]
\end{definition}

We see that
\begin{equation}
    Q_k[L,J] =\bigsqcup_{\vec\de\in \sqcup_{k\le j\le M}\cD_j, \vec\nu }P[\vec\de,\vec\nu;L,J]. 
\end{equation}
We  call each sub-rectangular box on the RHS a \textbf{block} of the LHS. If $X$ is a finite union of blocks $X=\sqcup P$, we say that $\phi_X$ is adapted to $X$ if $\phi_X=\sum_{P\subset X} \phi_P$ where each $\phi_P$ is adapted to $P$.

We have
\begin{equation}\label{QQ}
    Q_k[L,J]\Big\backslash Q_{k+1}[L,J]=\bigsqcup_{\vec\de\in  \cD_k, \vec\nu }P[\vec\de,\vec\nu;L,J].
\end{equation}

\bigskip

We partition the range of the variable $s$. Let $K$ be the number in Lemma \ref{structurelem1}. For $1\le k\le M$, let $\bI_k=\{I_k\}$ be a set of disjoint intervals of length $K^{-1}\De_k$ that form a partition of $[0,1]$.

For $1\le k\le M$, we let $\{\varphi_{I_k}\}_{I_k\in\bI_k}$ be a set of bump functions such that $\varphi_{I_k}$ is supported in $2I_k$ and $\sum_{I_k}\varphi_{I_k}=1$ in $[0,1]$. We can also assume the derivative estimates
\begin{equation}\label{adapt}
    \|\partial^m \varphi_{I_k}\|_\infty\lesssim_m (K^{-1}\De_k)^{-m},
\end{equation}
for any $m\in \mathbb N$. Thus, $\varphi_{I_k}$ is a bump function  adapted to $I_k$.

For $I_i\in\bI_i$, we define
\begin{equation}\label{psiI}
\varphi_{I_k,I_{k+1},\dots,I_{M}}=\prod_{i=k}^M\varphi_{I_i}. 
\end{equation} 
We see that $\varphi_{I_k,I_{k+1},\dots,I_M}$ also satisfies the derivative estimate \eqref{adapt}. Thus, it is also a bump function adapted to $I_k$.

It is convenient to introduce the following notion. For $I_i\in \bI_i, I_{i+l}\in   \bI_{i+l}$ ($l\ge 1)$, we write 
$I_i<I_{i+l}$ if $\supp\varphi_{I_i}\cap \supp\varphi_{I_{i+l}}\neq\emptyset$. We see that \eqref{psiI} is nonzero only when $I_k<\cdots<I_M$.

Define
\begin{equation}\label{mupsi}
    \wh\mu_{I_k,\dots,I_M}(\xi,\xi_{n+1}):=\int_{[0,1]\times [1,2]} e^{-it(\ga(s),1)\cdot(\xi,\xi_{n+1})}\chi(s,t) \varphi_{I_k,\dots,I_M}(s)\mathrm{d}s\mathrm{d}t.
\end{equation}
We recall that $\bfe_{n+1}(s)=\frac{(\ga(s),1)}{\sqrt{|\ga(s)|^2+1}}$.
Noting that $\sum_{I_1<\dots<I_M}\varphi_{I_1,\dots,I_M}=1$ on $[0,1]$, we have the decomposition
\begin{equation}\label{decomposition}
\wh\mu_\Ga=\sum_{I_1<\dots<I_M}\wh\mu_{I_1,\dots,I_M}.
\end{equation}

From \eqref{mupsi}, one sees that $\mu_{I_1,\dots,I_M}$ is supported in a constant enlargement of 
\[\Ga(I_1):=\{ t(\ga(s),1):s\in I_1, t\in [1,2] \}.\] Let $s_I$ denote the center of $I$. Recall that $I_1$ is an interval of length $\lesssim R^{-\frac12}$. Since $\Ga(I_1)$ is contained in a translate of the rectangular box
\begin{equation}\label{gaI1}
    \Big\{ \sum_{i=1}^{n+1}a_i\bfe_i(s_{I_1}): |a_i|\lesssim R^{-1} (1\le i\le n-1), |a_n|\lesssim R^{-1/2}, |a_{n+1}|\lesssim 1 \Big\}, 
\end{equation} 
we see that $\wh \mu_{I_1,\dots,I_M} 1_{|\xi|\sim R}$ is essentially supported in the dual box of \eqref{gaI1} minus $[-\frac{R}{2},\frac{R}{2}]^{n+1}$, i.e., $  Q_1(s_{I_1})$. 
\begin{lemma}
For sufficiently small $\beta>0$,
    $\wh\mu_{I_1,\dots,I_M}1_{|\xi|\sim R}$ is essentially supported in \[
    Q_1[R^{1/2+\beta},R^{2\beta}](s_{I_1}).
\] 
Equivalently, it is essentially supported in
\[
    \left\{
        \sum_{i=1}^{n+1} a_i\bfe_i(s_{I_1}):
        |a_i|\le R\ (1\le i\le n-1),\
        |a_n|\le R^{\frac12+\beta},\
        |a_{n+1}|\le R^{2\beta}
    \right\}
    \setminus
    \left[-\frac{R}{2},\frac{R}{2}\right]^{n+1}.
\]
\end{lemma}

\begin{proof}
Write $\zeta=(\xi,\xi_{n+1})$. Recall that
\[
    \wh\mu_{I_1,\dots,I_M}(\zeta)
    =
    \int_{[0,1]\times[1,2]}
        e^{-it(\ga(s),1)\cdot\zeta}
        \chi(s,t)\varphi_{I_1,\dots,I_M}(s)
        \,\mathrm{d}s\,\mathrm{d}t.
\]
We restrict attention to $\zeta\in B_R^{n+1}$. Repeated integration by parts in $s$ shows that,
\begin{equation}\label{errorterm}
    \wh\mu_{I_1,\dots,I_M}(\zeta)=\rap,
\end{equation}
unless
\[
    |\ga'(s)\cdot\xi|
    \lesssim R^{1/2+\beta/2}
    \qquad\text{for some }s\in I_1.
\]
Since $|I_1|\lesssim R^{-1/2}$ and $|\xi|\lesssim R$, this implies
\[
    |\ga'(u)\cdot\xi|
    \lesssim R^{1/2+\beta/2} \qquad\text{for all }u\in I_1.
\]
    
Similarly, repeated integration by parts in $t$ gives rapid decay
unless
\[
    |(\ga(s),1)\cdot\zeta|
    \lesssim R^{\beta/2}
    \qquad\text{for some }s\in I_1.
\]
Hence, 
\[
\begin{split}
    |(\ga(s_{I_1}),1)\cdot\zeta|
    &\le
    |(\ga(s),1)\cdot\zeta|
    +\left|\int_s^{s_{I_1}}\ga'(u)\cdot\xi\,\mathrm{d}u\right| \\
    &\lesssim
    R^{\beta/2}
    +R^{-1/2}R^{1/2+\beta/2}
    \lesssim R^{\beta/2}.
\end{split}
\]
Therefore, \eqref{errorterm} holds unless
\begin{equation}\label{essentiallysupp}
    |\ga'(s_{I_1})\cdot\xi|
    \lesssim R^{1/2+\beta/2},
    \qquad
    |(\ga(s_{I_1}),1)\cdot\zeta|
    \lesssim R^{\beta/2}.
\end{equation}
Equivalently, $\wh\mu_{I_1,\dots,I_M}$ is essentially supported in \eqref{essentiallysupp}.

Recall that
\[
    \bfe_{n+1}(s)\parallel(\ga(s),1),
    \qquad
    \bfe_n(s)
    =a_s(\ga'(s),0)+b_s\bfe_{n+1}(s),
\]
where $|a_s|\sim1$ and $|b_s|\lesssim1$. Since
$\bfe_1(s_{I_1}),\dots,\bfe_{n+1}(s_{I_1})$ is an orthonormal
frame and $|\zeta|\lesssim R$, the preceding estimates imply that
$\wh\mu_{I_1,\dots,I_M}1_{|\xi|\sim R}$ is essentially supported in
\[
    \left\{
        \sum_{i=1}^{n+1} a_i\bfe_i(s_{I_1}):
        |a_i|\le R\ (1\le i\le n-1),\
        |a_n|\le R^{1/2+\beta},\
        |a_{n+1}|\le R^{2\beta}
    \right\}
    \setminus
    \left[-\frac{R}{2},\frac{R}{2}\right]^{n+1}.
\]
In particular, this set is contained in
\[
    Q_1[R^{1/2+\beta},R^{2\beta}](s_{I_1}).
\]
\end{proof}

Let 
\begin{equation}\label{defbeta}
    \beta:=\frac{1}{M^{1/2}}
\end{equation}
 from now on. Then $\frac1M\ll\beta\ll1$.
 
\medskip

Let $\eta_R$ be a smooth bump function that equals $1$ on $|\xi|\sim R$.
By \eqref{decomposition} and the above lemma,
\begin{equation}\label{decompose}
\begin{split}
    \wh\mu_\Ga(\xi,\xi_{n+1})\eta_R(\xi)
    &=
    \sum_{I_1<\dots<I_M}
        \wh\mu_{I_1,\dots,I_M}(\xi,\xi_{n+1}) \\
    &\qquad\times
        \phi_{Q_1[R^{1/2+\beta},R^{2\beta}](s_{I_1})}
            (\xi,\xi_{n+1})
        \eta_R(\xi)
    +\rap,
\end{split}
\end{equation}
where
$\phi_{Q_1[R^{1/2+\beta},R^{2\beta}](s_{I_1})}$ is adapted to
$Q_1[R^{1/2+\beta},R^{2\beta}](s_{I_1})$.

Our goal is to show that $\phi_{Q_1[R^{1/2+\beta},R^{2\beta}](s_{I_1})}$ in \eqref{decompose} can be replaced by the characteristic function of
boxes with narrower supports in the $\bfe_n$ direction. 
We need the following lemma, which is similar to Lemma \ref{structurelem1}. 
\begin{lemma}\label{comp}
Let $2\le m\le M$.
    If $|s-s'|\le K^{-1} \De_{m}$ (where $K$ is a large constant) and $\vec\de\in \cD_j$ for $m\le j\le M$, then for any $\De_m^{-1}R^{\beta}\le L\le \De_{m-1}^{-1}R^\beta$, $P[\vec\de,\vec\nu;L,R^{2\beta}](s)$ and $P[\vec\de,\vec\nu;L,R^{2\beta}](s')$ are comparable. In particular, it holds for $L=\De_{m-1}^{-1}R^\beta$ and $ \De_{m}^{-1}R^\beta$.
\end{lemma}

\begin{proof}
    If we replace $(L,R^{2\beta})$ with $(\De(\vec\de)^{-1},1)$, then the rectangular boxes become $P[\vec\de;\vec\nu](s)$ and $P[\vec\de;\vec\nu](s')$. By Lemma \ref{structurelem1}, they are comparable. We now thicken their two shortest sides and show that the comparability is preserved. We follow the same steps as in Lemma \ref{structurelem1}.

Write $s'=s+\Delta$, with $|\Delta|\leq K^{-1}\De_m$ by assumption.  We prove that
$P[\vec\delta,\vec\nu;L,R^{2\beta}](s')\subset C P[\vec\delta,\vec\nu;L,R^{2\beta}](s)$; the reverse inclusion is identical.

Set
\[
    L_i:=R\rho_{i-1}\quad (1\leq i\leq n-1),\qquad L_n=L,\qquad L_{n+1}:=R^{2\beta}.
\]
Let
\[
    z=\sum_{i=1}^{n+1}a_i' \bfe_i(s')\in P[\vec\delta,\vec\nu;L_n,L_{n+1}](s').
\]
Thus $|a_i'|\lesssim L_i$ for all $i$, and for $1\leq i\leq n-1$ with
$\nu_i\neq 0$, the coefficient $a_i'$ has the prescribed sign and satisfies
$|a_i'|\sim L_i$.

Write $z$ in the basis at $s$:
\[
    z=\sum_{k=1}^{n+1}a_k \bfe_k(s),
    \qquad 
    a_k=\langle z,\bfe_k(s)\rangle .
\]
By the same reasoning as in the proof of Lemma \ref{structurelem1}, we have
\[
    |a'_k-a_k|=\langle z,\bfe_k(s+\De)-\bfe_k(s)\rangle
    \lesssim |a_k'||\Delta|^2
      +\sum_{i\neq k} |a_i'| |\Delta|^{|k-i|}.
\]
We claim that every term on the right is $O(K^{-1}L_k)$; this will finish the proof by the same reasoning as in Lemma \ref{structurelem1}. It remains to prove
\[ |a_i'||\De|^{|k-i|}\lesssim K^{-1}L_k, \]
for $k=n,n+1$ and $i<k$;
the other cases have been proved in Lemma \ref{structurelem1}.

$\bullet$ If $k=n+1$, we need to prove
\[ L_i |\De|^{n+1-i}\lesssim K^{-1}R^{2\beta}. \]
The cases $i\le n-1$ have been verified in Lemma \ref{structurelem1} (with the right-hand side even $=1$). We just need to prove for $i=n$, i.e.,
\[ \De_{m-1}^{-1}R^\beta K^{-1}\De_m\lesssim K^{-1}R^{2\beta}. \]
This is true since $R^\beta\gg R^{\frac{1}{M-1}}=\De_m/\De_{m-1}$.

$\bullet$ If $k=n$, we need to prove
\[ L_i |\De|^{n-i}\lesssim K^{-1}\De_{m}^{-1}R^\beta. \]
Indeed, the proof of Lemma \ref{structurelem1} gives
\[ L_i|\De|^{n-i}\lesssim K^{-1}R\de_1^{n-1}\dots\de_{n-1}=K^{-1}\De(\vec\de)^{-1}. \]
Then we just need to note that $\De_m\le \De_j\le \De(\vec\de)$.
\end{proof}

Fix a sequence $I_1<\dots<I_M$. Let us state a key lemma. For $1\le k\le M$, we will construct functions $\phi_{k},\phi_k',\psi_k$ so that the following hold. We 
remark that $\phi_k'$ is not the derivative of $\phi_k$, but rather another function.

\begin{lemma}[Partition of unity]\label{partitionunity}
There exist functions 
\begin{equation}
    \{ \phi_{I_k} \}_{1\le k\le M, I_k\in\bI_k},\;\; \{   \phi'_{I_k}\}_{2\le k\le M, I_k\in\bI_k}, \;\;\{ \psi_{I_k,I_{k+1}} \}_{1\le k\le M-1, I_k\in\bI_k, I_{k+1}\in\bI_{k+1}, I_k<I_{k+1} }
\end{equation}
such that the following properties hold.
\begin{enumerate}
    \item $\phi_{I_k }$ is a bump function adapted to  $Q_k[\De_{k}^{-1}R^{\beta},R^{2\beta}](s_{I_k})$;
    \item $\phi'_{I_k }$ is a bump function adapted to  $Q_k[\De_{k-1}^{-1}R^{\beta},R^{2\beta}](s_{I_k})$;
\item $\psi_{I_k,I_{k+1} }$ is a bump function adapted to  \[\bigsqcup_{\vec\de\in  \cD_k  } \bigsqcup_{\vec{\nu}} P[\vec\de,\vec\nu;\De_{k}^{-1}R^{\beta},R^{2\beta}](s_{I_k});\]
    
    \item If $I_k<I_{k+1}$, then
    \begin{equation}\label{partphi}
\phi_{I_k}=\phi'_{I_{k+1}}+\psi_{I_k,I_{k+1}}.
    \end{equation} 
\item If $\zeta\in\supp (\phi'_{I_k}-\phi_{I_k})$, then
\begin{equation}\label{item5}
    |\zeta\cdot \bfe_n(s)|\gtrsim \De_{k}^{-1}R^\beta,
\end{equation} 
for all $|s-s_{I_k}|\le K^{-1}\De_k$.
   
\end{enumerate}
    
\end{lemma}

\begin{figure}[ht]
\centering
\begin{tikzpicture}[
  x=0.01243cm,
  y=-0.01243cm,
  line cap=butt,
  line join=miter,
  every node/.style={inner sep=0pt,outer sep=0pt,font=\normalsize}
]
  \path[use as bounding box] (0,0) rectangle (641,557);

  \tikzset{
    figline/.style={draw=black,line width=.35pt}
  }

  \draw[figline]
    (247,27)--(358,27)--(358,242)--(573,242)--(573,335)--
    (357,335)--(357,550)--(247,550)--(247,335)--(31,335)--
    (31,242)--(247,242)--cycle;

  \draw[figline] (247,27)--(284,3)--(396,3)--(358,27);
  \draw[figline] (396,3)--(396,219);
  \draw[figline] (247,159)--(358,159)--(396,136);

  \draw[figline] (31,242)--(69,218)--(247,218);
  \draw[figline] (176,242)--(213,218);
  \draw[figline] (176,242)--(176,335);

  \draw[figline] (358,242)--(396,219)--(611,219)--(573,242);
  \draw[figline] (573,242)--(611,219)--(611,312)--(573,335);
  \draw[figline] (573,335)--(611,312);
  \draw[figline] (438,242)--(477,219);
  \draw[figline] (438,242)--(438,335);

  \draw[figline] (247,405)--(357,405)--(394,382);
  \draw[figline] (394,335)--(394,527)--(357,550);

  \draw[figline]
    (282,95)--(323,95)--(323,269)--(497,269)--(497,308)--
    (323,308)--(323,482)--(282,482)--(282,308)--(108,308)--
    (108,269)--(282,269)--cycle;

  \draw[figline,-{Stealth[length=5pt,width=3.5pt]}] (432,38)--(477,7);

  \node at (306,65) {$\psi_{I_k,I_{k+1}}$};
  \node at (311,124) {$\phi'_{I_{k+1}}$};
  \node at (453,139) {$\phi_{I_k}$};
  \node at (493,33) {$e_n$};

\end{tikzpicture}
\caption{Schematic supports of $\phi_{I_k}$, $\phi'_{I_{k+1}}$, and $\psi_{I_k,I_{k+1}}$, with the $\bfe_n$-direction indicated.}
\label{cross}
\end{figure}

\begin{proof}
We first describe the geometry of the support of $\phi_{I_k},\phi_{I_k}',\psi_{I_k,I_{k+1}}$ heuristically. See Figure \ref{cross}. $\phi_{I_k}$  can be viewed heuristically as a bump function of the big cross minus the central square $[-\frac{R}{2},\frac{R}{2}]^{n+1}$; $\phi'_{I_{k+1}}$ can be viewed as a bump function of the small cross minus $[-\frac{R}{2},\frac{R}{2}]^{n+1}$; $\psi_{I_k,I_{k+1}}$ is morally a bump function of the difference of them; $\phi_{I_{k+1}}$ is obtained from $\phi'_{I_{k+1}}$ by making the support thinner on the $\bfe_n(s_{I_{k+1}})$ direction.

Write \(L_k=\De_k^{-1}R^\beta\). Set
\[
\phi_{I_1}:=\phi_{Q_1[L_1,R^{2\beta}](s_{I_1})},
\]
where the right hand side is the cutoff in \eqref{decompose}. 

Suppose \(\phi_{I_k}\) has been constructed and \(I_k<I_{k+1}\). By
\eqref{QQ},
\[
Q_k[L_k,R^{2\beta}](s_{I_k})
=
Q_{k+1}[L_k,R^{2\beta}](s_{I_k})
\sqcup
\bigsqcup_{\vec\de\in\cD_k,\vec\nu}
P[\vec\de,\vec\nu;L_k,R^{2\beta}](s_{I_k}).
\]
Since \(|s_{I_k}-s_{I_{k+1}}|\le K^{-1}\De_{k+1}\), Lemma
\ref{comp} (with $m=k+1$) shows that the blocks of the first term on the right hand side are comparable to the
corresponding blocks based at \(s_{I_{k+1}}\). Hence a smooth partition
subordinate to these two regions gives
\[
\phi_{I_k}
=
\phi'_{I_{k+1}}+\psi_{I_k,I_{k+1}},
\]
where \(\phi'_{I_{k+1}}\) and \(\psi_{I_k,I_{k+1}}\) are adapted to the regions
claimed in (2) and (3). Hence, \eqref{partphi} is verified.

It remains to narrow the support in the \(\bfe_n\)-direction. Choose a fixed smooth cutoff
\(\chi\), equal to \(1\) on a sufficiently large neighborhood of \(0\), and let
\[
\phi_{I_{k+1}}(\zeta)
=
\phi'_{I_{k+1}}(\zeta)\,
\chi\!\left(
\frac{\zeta\cdot\bfe_n(s_{I_{k+1}})}{L_{k+1}}
\right).
\]
Then \(\phi_{I_{k+1}}\) is adapted to
\(Q_{k+1}[L_{k+1},R^{2\beta}](s_{I_{k+1}})\). Moreover, for
\(\zeta\in \supp(\phi'_{I_{k+1}}-\phi_{I_{k+1}})\),
\begin{equation}\label{awayen}
    |\zeta\cdot\bfe_n(s_{I_{k+1}})|\ge C_0 L_{k+1}.
\end{equation}
Here, $C_0$ can be made large depending on the choice of $\chi$.

By Lemma \ref{comp} (with $m=k+1$), we see that when $|s-s_{I_{k+1}}|\le K^{-1}\De_{k+1}$,  the blocks of $Q_{k+1}[L_{k+1},R^{2\beta}](s)$ are comparable to the corresponding blocks of $Q_{k+1}[L_{k+1},R^{2\beta}](s_{I_{k+1}})$. When $C_0$ is large,
\eqref{awayen} tells us that $\zeta$ cannot lie in a large dilation of $Q_{k+1}[L_{k+1},R^{2\beta}](s_{I_{k+1}})$ and hence cannot lie in $Q_{k+1}[L_{k+1},R^{2\beta}](s)$.
This implies that
\[ |\zeta\cdot \bfe_n(s)|\ge L_{k+1} \]
for any $|s-s_{I_{k+1}}|\le K^{-1}\De_{k+1}$, which proves (5) in the $(k+1)$-case.

The iteration finishes the proof of the lemma.
\end{proof}

We next prove a cancellation lemma. This enables us to replace $\phi_{I_k}'$ with $\phi_{I_k}$.

\begin{lemma}\label{cancellation}
    For $k=2,\dots,M$, we have
\begin{equation}
   \wh \mu_{I_{k-1},\dots,I_M}(\xi,\xi_{n+1}) \big(\phi'_{I_k}(\xi,\xi_{n+1}) -\phi_{I_k}(\xi,\xi_{n+1}) \big)=\rap.
\end{equation}
\end{lemma}

\begin{proof}
    In the proof, we will use Lemma \ref{BGHSlem}, which was proved in \cite[Appendix D]{MR4861588}.
    By the support properties of $\phi_{I_k}'$ and $\phi_{I_k}$, together with \eqref{item5}, we  need to show that if $(\xi,\xi_{n+1})\in Q_k[\De_{k-1}^{-1}R^\beta,R^{2\beta}](s_{I_k})$ satisfies
\[|(\xi,\xi_{n+1})\cdot \bfe_n(s)|\gtrsim \De_{k}^{-1}R^\beta  \]
for all $s$ with $|s-s_{I_k}|\le K^{-1}\De_k$, then we have $|\wh \mu_{I_{k-1},\dots,I_M}(\xi,\xi_{n+1})|=\rap$.

Recall from \eqref{mupsi} that
\[ \wh\mu_{I_{k-1},\dots,I_M}(\xi,\xi_{n+1})=\int_{[0,1]\times [1,2]}e^{-it(\ga(s),1)\cdot(\xi,\xi_{n+1})}\chi(s,t) \varphi_{I_{k-1},\dots,I_M}(s)\mathrm{d}s\mathrm{d}t, \]
where $\varphi_{I_{k-1},\dots,I_M}$ is defined in \eqref{psiI}. Hence, $\varphi_{I_{k-1},\dots,I_M}$ is adapted to $I_{k-1}$. Fix $t\in [1,2]$ and denote $\phi(s)=t(\ga(s),1)\cdot(\xi,\xi_{n+1})$ and  $ a(s)=\chi(s,t)\varphi_{I_{k-1},\dots,I_M}(s)$. 
It remains to show, uniformly in $t$, that
\[ \Big|\int_\R e^{-i\phi(s)}a(s)\mathrm{d}s \Big|=\rap. \]
It suffices to verify the three conditions of Lemma \ref{BGHSlem}. Note that $\supp\ a\subset 2I_{k-1}$, which is an interval of length $K^{-1}\De_{k-1}$. We choose $r=R^{\frac{\beta}{2}}$ in Lemma \ref{BGHSlem}. $(\ga(s),1)=\sqrt{|\ga(s)|^2+1}\bfe_{n+1}(s)$ implies $(\ga'(s),0)=b_s\bfe_{n}(s)+c_s\bfe_{n+1}(s)$ with $|b_s|\sim 1, |c_s|\lesssim 1$. Lemma \ref{comp} (with $m=k$) shows that $(\xi,\xi_{n+1})$ lies in a dilation of $Q_k[\De_{k-1}^{-1}R^\beta,R^{2\beta}](s)$ and hence $|\bfe_{n+1}(s)\cdot(\xi,\xi_{n+1})|\lesssim R^{2\beta}\ll \Delta_k^{-1}R^\beta$. We therefore have
\[
|\phi'(s)|\sim |\bfe_n(s)\cdot(\xi,\xi_{n+1})|\gtrsim \Delta_k^{-1}R^\beta\ge \De_{k-1}^{-1}r
\]
for all $s\in \operatorname{supp} a$.

For any $j \geq 0$,
\begin{equation}
    \begin{split}
        |a^{(j)}(s)|\lesssim (K^{-1}\De_{k-1})^{-j}\lesssim r^{-j}|\phi'(s)|^j.
    \end{split}
\end{equation}
Since $\supp\ a\subset 2I_k$,
it remains to verify for $|s-s_{I_k}|\le K^{-1}\De_{k}$ that 
\begin{equation}\label{estphi}
    |\phi^{(j)}(s)|\lesssim r\De_{k-1}^{-j} \lesssim r^{-(j-1)}|\phi'(s)|^j \;\; \textup{~for~} j\ge 2. 
\end{equation}
Since the second inequality has already been established, it suffices to prove the first inequality. Suppose that $|s- s_{I_k}| \leq K^{-1}\Delta_{k}$.
Writing $\phi(s)=t b(s)\bfe_{n+1}(s)\cdot (\xi,\xi_{n+1})$, where $b(s)= \sqrt{|\ga(s)|^2+1 }$ is a bounded smooth function,
for any $2 \leq j\le n$, we have
\[ |\phi^{(j)}(s)|\lesssim \sum_{i=n+1-j}^{n+1}|\bfe_{i}(s)\cdot (\xi,\xi_{n+1})|. \]
Since $(\xi,\xi_{n+1})\in Q_k[\De_{k-1}^{-1}R^\beta,R^{2\beta}](s_{I_k})$, there exists $\vec\de\in \bigcup_{k\le l\le M}\cD_l$ and $\vec\nu$ such that
$(\xi,\xi_{n+1})\in P[\vec\de,\vec\nu;\De_{k-1}^{-1}R^\beta,R^{2\beta}](s_{I_k})$. By Lemma \ref{comp}, $(\xi,\xi_{n+1})$ is contained in a dilation of $P[\vec\de,\vec\nu;\De_{k-1}^{-1}R^\beta,R^{2\beta}](s)$ since $|s-s_{I_k}|\le K^{-1}\De_k$. Let $(\rho_i)$ be the $\rho$-sequence of $\vec\de$  in \eqref{rhoseq}. We have
\[ \sum_{i=n+1-j}^{n+1}|\bfe_{i}(s)\cdot (\xi,\xi_{n+1})|\lesssim \sum_{i=n+1-j}^{n-1}R\rho_{i-1}+\De_{k-1}^{-1}R^\beta+R^{2\beta}. \]
To show that this is bounded by $\lesssim r\De_{k-1}^{-j}$, we  need to show
\[ R\rho_{n-j}+\De_{k-1}^{-1}R^\beta+R^{2\beta}\lesssim R^{\frac{\beta}{2}}\De_{k-1}^{-j}. \]
We have
\[ R\rho_{n-j}=\de_1^{-n}\cdots\de_{n-1}^{-2}\cdot \de_1^{n-j}\cdots\de_{n-j}\le \De(\vec\de)^{-j}\le \De_k^{-j}\le R^{\frac{\beta}{2}}\De_{k-1}^{-j}. \]
The second last inequality holds because $\vec\de \in \cD_l$ for some $l\ge k$.
Since $j\ge 2$ and $\De_{k-1}\le R^{-\frac1n}\le R^{-\beta}$, we have $\De_{k-1}^{-1}R^\beta\le R^{\frac{\beta}{2}}\De_{k-1}^{-j}$. 
Finally, we have $\De_{k-1}^{-j}\ge \De_{k-1}^{-1}\ge R^\beta$.

This finishes the proof of \eqref{estphi} when $2\le j\le n$. For $j> n$, we have
\[ |\phi^{(j)}(s)|\lesssim \sum_{i=1}^{n+1}|\bfe_i\cdot(\xi,\xi_{n+1})|\lesssim r\De_{k-1}^{-n}\le r\De_{k-1}^{-j}. \]
Hence, the three conditions in Lemma \ref{BGHSlem} are satisfied. This finishes the proof.    
\end{proof}

\smallskip

We recall \eqref{decompose} as follows:
\begin{equation}\label{basecase}
    \eta_R \wh\mu_\Ga =\eta_R\sum_{I_1<\dots<I_M}\wh \mu_{I_1,\dots,I_M}\phi_{I_1}+\rap.
\end{equation}
We will derive the following result.

\begin{theorem}[Fourier decay rate of a cone generated by the nondegenerate curve]\label{fourierdecay}
We have
\begin{equation}
   \eta_R \wh \mu_\Ga=\eta_R\sum_{I_1<\dots<I_M}\wh\mu_{I_1,\dots,I_M}\bigg(\sum_{k=1}^{M-1} \psi_{I_k,I_{k+1}} + \phi_{I_M}\bigg)+\rap.
\end{equation}

\end{theorem}

\begin{proof}
We prove by induction that
\begin{equation}
   \eta_R \wh \mu_\Ga=\eta_R\sum_{I_1<\dots<I_M}\wh\mu_{I_1,\dots,I_M}\bigg(\sum_{i=1}^{k}  \psi_{I_i,I_{i+1}} + \phi_{I_{k+1}}\bigg)+\rap,
\end{equation}
for $k=0,\dots,M-1$.  \eqref{basecase} gives the base case $k=0$. Suppose it has been proved for $k-1$:
\begin{equation}
   \eta_R \wh \mu_\Ga=\eta_R\sum_{I_1<\dots<I_M}\wh\mu_{I_1,\dots,I_M}\bigg(\sum_{i=1}^{k-1}  \psi_{I_i,I_{i+1}} + \phi_{I_{k}}\bigg)+\rap.
\end{equation}
We deal with the term
\[  \sum_{I_1<\dots<I_M}\wh\mu_{I_1,\dots,I_M}\phi_{I_{k}}=\sum_{I_k<\dots<I_M}\wh\mu_{I_k,\dots,I_M}\phi_{I_{k}}.\]
Here, we used that $\sum_{(I_1,\dots,I_{k-1}):I_1<\dots<I_k}\wh\mu_{I_1,\dots,I_M}=\wh\mu_{I_k,\dots,I_M}$. 
By item (4) of Lemma \ref{partitionunity}, we write 
\[\sum_{I_k<\dots<I_M}\wh\mu_{I_k,\dots,I_M}\phi_{I_{k}}=\sum_{I_k<\dots<I_M}\wh\mu_{I_k,\dots,I_M}(\phi'_{I_{k+1}}+\psi_{I_k,I_{k+1}}).\] 
By Lemma \ref{cancellation}, 
\begin{align*}
    \sum_{I_k<\dots<I_M}\wh\mu_{I_k,\dots,I_M} \phi'_{I_{k+1}}=&\sum_{I_k<\dots<I_M}\wh\mu_{I_k,\dots,I_M} \phi_{I_{k+1}}+\rap  \\
    =&\sum_{I_1<\dots<I_M}\wh\mu_{I_1,\dots,I_M} \phi_{I_{k+1}}+\rap.
\end{align*} 
This finishes the proof for the case $k$.
\end{proof}

\smallskip

\subsection{Reduction of the local smoothing estimate}\label{subsection710}

Let $\sigma=\sigma_p=\frac1n$, when $2\le p\le4$; $\sigma=\frac1n(\frac12+\frac2p)$ when $4\le p\le 4(n-1)$; $\sigma=\frac2p$ when $p\ge 4(n-1)$.
Recall that our goal is to prove that for $2 \leq p \leq \infty$
\begin{equation}\label{0724.242}
    \Big\| \Big(\wh f(\xi) \wh \mu_\Ga(\xi,\xi_{n+1})\Big)^\vee \Big\|_{L^p(W_{B_1^{n+1}} )}\lesssim R^{-\sigma+\e} \|f\|_{L^p(W_{B_1^n})}.
\end{equation}
Here we assume $ \supp \wh f\subset B_R^n\setminus
B_{R/2}^n$ and $(\xi,\xi_{n+1})\in\R^n\times \R$.
When $p=2$ or $\infty$, \eqref{0724.242} is straightforward. By interpolation, it suffices to consider $p\in [4,4(n-1)]$.
We are going to make a sequence of reductions.

By Theorem \ref{fourierdecay}, we have
\[ \wh f(\xi) \wh \mu_\Ga=\wh f(\xi)\sum_{I_1<\dots<I_M} \wh\mu_{I_1,\dots,I_M}\bigg(\sum_{k=1}^{M-1}\psi_{I_k,I_{k+1}}+\phi_{I_M} \bigg)+\hat{f}(\xi) \rap . \]
By the triangle inequality, it suffices to fix $m \in [1,M-1]$ and prove
\[ \Big\| \Big(\wh f(\xi)\sum_{I_m<\dots<I_M} \wh\mu_{I_m,\dots,I_M} \psi_{I_m,I_{m+1}} \Big) ^\vee \Big\|_{L^p(W_{B_1^{n+1}})}\lesssim R^{-\sigma+\e} \|f\|_{L^p(W_{B_1^n})}. \]
The term including $\phi_{I_M}$ can be handled similarly as $\psi_{I_{M-1},I_M}$. 
For fixed $I_m$, there are $O(1)$ choices of $I_{m+1},\dots,I_M$ such that $I_m<\dots<I_M$. By the triangle inequality, we may choose  $I_{m+1},\dots,I_M$ as a function of $I_m$, and it suffices to prove
\[ \Big\| \Big(\wh f(\xi)\sum_{I_m} \wh\mu_{I_m,\dots,I_M} \psi_{I_m,I_{m+1}} \Big) ^\vee \Big\|_{L^p(W_{B_1^{n+1}})}\lesssim R^{-\sigma+\e} \|f\|_{L^p(W_{B_1^n})}. \]

For simplicity, we write
\[ \wh\mu_{I_m}=\wh\mu_{I_m,\dots,I_M}. \]
We have
\begin{equation}\label{muIm}
\begin{split}
    \wh\mu_{I_m}(\xi,\xi_{n+1})&=\int_{[0,1]\times [1,2]} e^{-it(\ga(s),1)\cdot(\xi,\xi_{n+1})}\chi(s,t) \chi_{I_m}(s)\mathrm{d}s\mathrm{d}t\\
    &=\int_{[s_{I_m}-\De_m,s_{I_m}+\De_m]\times [1,2]} e^{-it(\ga(s),1)\cdot(\xi,\xi_{n+1})}\chi(s,t) \chi_{I_m}(s)\mathrm{d}s\mathrm{d}t,
\end{split} 
\end{equation} 
where $\chi_{I_m}$ is a bump function adapted to $I_m$.
Since $\psi_{I_m,I_{m+1}}$ is adapted to
\[\bigsqcup_{\vec\de\in  \cD_m, \vec\nu }P[\vec\de,\vec\nu;\De_{m}^{-1}R^{\beta},R^{2\beta}](s_{I_m}),\]
we can write
\[ \psi_{I_m,I_{m+1}}=\sum_{\vec\de\in\cD_m,\vec\nu}\phi_{\vec\de,\vec\nu,s_{I_m}}, \]
where $\phi_{\vec\de,\vec\nu,s_{I_m}}$ is adapted to $P[\vec\de,\vec\nu;\De_{m}^{-1}R^{\beta},R^{2\beta}](s_{I_m})$.
Since there are $(\log R)^{O(1)}$ choices of $\vec\de$ and $O(1)$ choices of $\vec\nu$, we may fix an admissible $(\vec\de,\vec\nu)$ with $\vec\de\in\cD_m$, and it suffices to prove
\begin{equation}\label{becomes1}
    \Big\| \Big(\wh f(\xi) \sum_{I_m\in \bI_m}\wh\mu_{I_m} \phi_{\vec\de,\vec\nu,s_{I_m}} \Big)^\vee \Big\|_{L^p(W_{B_1^{n+1}})}\lesssim R^{-\sigma+\e} \|f\|_{L^p(W_{B_1^n})}.
\end{equation}
Denote 
\[ \de:=\De(\vec\de)=\prod_{i=1}^{n-1}\de_i. \]
We have $\de\in [\De_m,\De_{m+1})=[\De_m,R^{(\frac12-\frac1n)\frac{1}{M}}\De_m)$. Therefore,
$P[\vec\de,\vec\nu;\De_m^{-1}R^\beta,R^{2\beta}](s)$ and $P[\vec\de,\vec\nu](s)$ are the same up to a factor $R^{O(M^{-\frac12})}$, recalling \eqref{defbeta}.
We may choose $M^{-1}=\e^{10000}$ so that the factor is dominated by $R^\e$. Hence, we may omit the factor $R^{O(M^{-\frac12})}$, and \eqref{becomes1} becomes
\begin{equation}\label{becomes2}
    \Big\| \Big(\wh f(\xi) \sum_{I_m\in \bI_m}\wh\mu_{I_m} \phi_{P[\vec\de,\vec\nu](s_{I_m})} \Big)^\vee \Big\|_{L^p(W_{B_1^{n+1}})}\lesssim R^{-\sigma+\e} \|f\|_{L^p(W_{B_1^n})},
\end{equation}
where $\phi_{P[\vec\de,\vec\nu](s_{I_m})}$ is a bump function adapted to $P[\vec\de,\vec\nu](s_{I_m})$.

Recalling \eqref{muIm}, we use Minkowski's inequality for variables $s,t$. Then \eqref{becomes2} is bounded by
\begin{equation}
    \int_{[-\De_m,\De_m]\times [1,2]}\| F(\cdot,\cdot;s,t) \|_{L^p(W_{B_1^{n+1}})}\mathrm{d}s\mathrm{d}t
\end{equation}
where $F(x,x_{n+1};s,t)$ is defined by
\begin{equation*}
\begin{split}
    \Big(\wh f(\xi) \sum_{I_m\in \bI_m} e^{-it(\ga(s_{I_m}+s),1)\cdot (\xi,\xi_{n+1})} \chi(s_{I_m}+s,t)\chi_{I_m}(s_{I_m}+s)\phi_{P[\vec\de,\vec\nu](s_{I_m})} \Big)^\vee(x,x_{n+1}).
\end{split}
\end{equation*}
Hence, by taking $\sup$ over $s,t$ variables, \eqref{becomes2} boils down to showing 
\begin{equation}\label{becomes3}
   \Big\| \Big(\wh f(\xi) \sum_{I_m\in \bI_m}a_{I_m} e^{-it(\ga(s_{I_m}+s),1)\cdot (\xi,\xi_{n+1})} \phi_{P[\vec\de,\vec\nu](s_{I_m})} \Big)^\vee \Big\|_{L^p(W_{B_1^{n+1}})}\lesssim  \De_m^{-1} R^{-\sigma+\e} \|f\|_{L^p(W_{B_1^n})}
\end{equation}
for all $a_{I_m}\in \mathbb C$ with $|a_{I_m}|\lesssim 1$, $1\le t\le 2$ and $|s|\le \De_m$.

We further simplify our notation. Recall that $\delta \in [\Delta_m,\Delta_{m+1})$. By losing a factor $\lesssim \frac{\de}{\De_m}\lesssim R^{1/(M-1)}$, 
 we may restrict to a
subcollection of intervals \(I_m\) whose centers are \(\delta\)-separated.
Denote the set of these centers by \(I_\delta\).
 Then \eqref{becomes3} boils down to proving
\begin{equation}\label{becomes4}
     \Big\| \Big(\wh f(\xi) \sum_{s\in \I_\de}a_{s} e^{-i\bfn_s\cdot (\xi,\xi_{n+1})} \phi_{P[\vec\de,\vec\nu](s)} \Big)^\vee \Big\|_{L^p(W_{B_1^{n+1}})}\lesssim \de^{-1} R^{-\sigma+\e} \|f\|_{L^p(W_{B_1^n})},
\end{equation}
where $|a_s|\lesssim 1$, $\bfn_s\in \{ t(\ga(s'),1):t\in[1,2], s'\in [s-\de,s+\de] \}$.
Since we have
\[  \Big(\wh f(\xi) e^{-i\bfn_s\cdot (\xi,\xi_{n+1})} \phi_{P[\vec\de,\vec\nu](s)} \Big)^\vee=\Big(\wh f(\xi) \phi_{P[\vec\de,\vec\nu](s)} \Big)^\vee(\cdot-\bfn_s), \]
\eqref{becomes4} becomes
\begin{equation} \label{rescaling} 
    \Big\|  \sum_{s\in \I_\de}a_{s}\Big(\wh f(\xi)  \phi_{P[\vec\de,\vec\nu](s)} \Big)^\vee(\cdot-\bfn_s) \Big\|_{L^p(W_{B_1^{n+1}})}\lesssim \de^{-1} R^{-\sigma+\e} \|f\|_{L^p(W_{B_1^n})}.
\end{equation}

Since $P[\vec\de,\vec\nu](s)$ remains comparable when $s$ is perturbed within length $\de$, we may assume $\bfn_s\in \{t(\ga(s),1):t\in[1,2]\}$.
Finally, we rescale \eqref{rescaling}. We dilate physical space by a factor $R$, centered at the origin, and dilate frequency space by a factor $R^{-1}$. Let $p[\vec{\de},\vec{\nu}](s)=R^{-1}P[\vec{\delta},\vec{\nu}](s)$. Then $p[\vec{\de},\vec{\nu}](s)$ has dimensions $1\times \rho_1\times \dots\times \rho_{n-1}\times R^{-1}$. We let $\phi_{p[\vec{\de},\vec{\nu}](s)}$ denote a bump function adapted to $p[\vec{\de},\vec{\nu}](s)$. For $ s\in [0,1]$, define $\Gamma(s):=\{t(\gamma(s),1): t \in [1,2] \}$.
The main result is stated as follows:
\begin{theorem}[Main estimate]\label{05.29.thm28} Fix an admissible pair $(\vec\de,\vec\nu)$  at scale $R$.  Let $\de=\prod_{i=1}^{n-1}\de_i$. Let $\I_{\delta}$ be a $\delta$-separated subset of $[0,1]$. For $4 \leq p \leq 4(n-1)$ and $\e>0$, we have
    \begin{equation*}
        \Big\|  \sum_{s \in \I_{\delta}}a_s \Big(\wh f(\xi)\phi_{p[\vec{\de},\vec{\nu}](s)} \Big)^\vee(\cdot-R\bfn_s)  \Big\|_{L^p(W_{B_R^{n+1}})}\lesssim \de^{-1} R^{-1}  R^{\frac1p}R^{-\frac1n(\frac12+\frac2p)+\e}\|f\|_{L^p(W_{B_R^n})}
    \end{equation*}
    for any function $f$ with $\wh f\subset B_1^n\setminus B_{\frac12}^{n}$,  $|a_s|\lesssim 1$, and $\bfn_s \in \Ga(s)$.
\end{theorem}

We have shown that Theorem \ref{05.29.thm28} implies Theorem \ref{mainthm}. In Sections \ref{sec3}--\ref{0624.sec8}, we prove Theorem \ref{05.29.thm28}.

\bigskip

\section{Kakeya-type estimates for all intermediate scales}\label{sec3}

In this section, we establish Kakeya-type estimates for the intermediate-scale planks
constructed in Section \ref{sec2}. The main result is Proposition \ref{wee}. The proof follows the strategy of \cite{MR4151084}, adapted to the present family of planks. These estimates will later serve as the geometric input for passing from frequency-localized pieces to wave-envelope-type quantities.

Fix an admissible pair $(\vec\de,\vec\nu)$ at scale $R$, and let
\[
    \rho_k=\de_1^k\de_2^{k-1}\cdots\de_k,
    \qquad 1\le k\le n-1,
    \qquad \rho_0=1,\quad \rho_n=R^{-1}.
\]
Thus the associated planks $P[\vec\de,\vec\nu](s)$ have dimensions
\[
    R\times R\rho_1\times\cdots\times R\rho_{n-1}\times 1.
\]
Let $Q[\vec\de]$ be the translate of $P[\vec\de,\vec\nu]$ centered at the
origin:
\[
    Q[\vec\de]
    =
    \bigl\{(a_1,\ldots,a_{n+1}):
    |a_i|\lesssim R\rho_{i-1}\ (1\le i\le n),\ |a_{n+1}|\lesssim1\bigr\}.
\]
Write $\de=\de_1\cdots\de_{n-1}$. We decompose this box into pieces indexed
by the dyadic scales $\de\le\si\le1$.

\subsection{The intermediate-scale decomposition}

Our goal is to perform a high/low decomposition 
\[ Q[\vec\de]=\bigsqcup_{\de\le\si\le 1} Q[\vec\de;\si].  \]

\textbf{Intuition.} Before giving the precise definition, let us explain the intuition.
Let $\si$ be a dyadic number in $[\de,1]$. Morally speaking, we want to define 
\begin{equation}\label{intersectionQ}
    Q[\vec\de;\le\si](s)\approx \bigcap_{|s'-s|\le \si} Q[\vec\de](s'),
\end{equation}
and let $Q[\vec\de;\sigma](s)=Q[\vec\de;\le\si](s)\setminus Q[\vec\de;\le2\si](s)$. Here, $Q[\vec\de;\le\si]$ gets smaller as $\si$ gets bigger, since it takes more intersection.

Morally speaking, the set on the right hand side of \eqref{intersectionQ} is a rectangular box centered at the origin.
We also want to compute its size. Write $z \in Q[\vec\de](s+\si)$ in the form
\[ z=\sum_{i=1}^{n+1} a_i\bfe_i(s+\si). \]
Using Taylor expansion with Lemma \ref{Taylorlem}, we have $\bfe_i(s+\si)= \sum_{k=1}^{n+1} O(\si^{|k-i|})\bfe_k(s)$. Rewrite as follows.
\[ z\approx\sum_{i=1}^{n+1}(\sum_{k=1}^{n+1} \si^{|k-i|}a_k )\bfe_i(s).  \]
If $z$ also lies in $Q[\vec\de](s)$, then we expect 
\[ |\sum_{k=1}^{n+1} \si^{|k-i|}a_k|\le R\rho_{i-1}. \]
Hence, we expect the range:
\[ |a_k|\le R\min_{k\le i\le n+1} \{\si^{k-i}\rho_{i-1}\}. \]

The log-concavity of the sequence $(\rho_i)$ reduces the minimum to its two
endpoints, and hence
\begin{equation}\label{eq:compressed-common-box}
    |a_k|\le
    \min\bigl\{R\rho_{k-1},\si^{k-n-1}\bigr\},
    \qquad 1\le k\le n+1.
\end{equation}
This motivates the following definition of critical numbers.

For $1\le k\le n$, let
\[
    \si_k
    :=\left(\frac{\rho_n}{\rho_{k-1}}\right)^{\!1/(n+1-k)},
    \qquad \si_0:=1.
\]
In particular, $\si_{n-1}=\si_n=\de$.
We have
\[
    \de=\si_n=\si_{n-1}\le\si_{n-2}\le\cdots\le\si_1\le\si_0=1.
\]
As usual, we replace these scales by comparable dyadic numbers when needed. We now give the formal definition.

\begin{definition}\label{defQsigma}
For dyadic $\de\le\si\le1$, define
\[
    Q[\vec\de;\le\si]
    :=\bigl\{(a_i)_{i=1}^{n+1}:
    |a_i|\le \min\{R\rho_{i-1},\si^{i-n-1}\}\bigr\}.
\]
For $\si<1$, let
\[
    Q[\vec\de;\si]
    :=Q[\vec\de;\le\si]\setminus Q[\vec\de;\le2\si],
\]
and put $Q[\vec\de;1]:=Q[\vec\de;\le1]$.
\end{definition}

We have the decomposition 
\[Q[\vec\de]=\bigsqcup_{\si} Q[\vec\de;\si].\]

    If $\si\in (\si_k,\si_{k-1}]$, we have
    \[ Q[\vec\de;\le \si]:=
    \prod_{i=1}^{k-1}[-R\rho_{i-1},R\rho_{i-1}]\times \prod_{i=k}^{n+1}[-\si^{i-n-1},\si^{i-n-1}].
    \]
and 
    \[ Q[\vec\de;\si]:= \prod_{i=1}^{k-1}[-R\rho_{i-1},R\rho_{i-1}]\times\bigg(\prod_{i=k}^{n}[-\si^{i-n-1},\si^{i-n-1}]\Big\backslash\prod_{i=k}^{n}[-(2\si)^{i-n-1},(2\si)^{i-n-1}]\bigg)\times [-1,1]. \]

\bigskip

\subsection{Overlap and the Kakeya estimate}

\begin{lemma}\label{finiteoverlaplem}
Fix $\si\in ( \si_k,\si_{k-1} ]$.
    Let $\I_\si$ be a $\si$-separated subset of $[0,1]$. Then $\{Q[\vec\de;\si](s):s\in \I_\si\}$  is $(\log R)^{O(1)}$-overlapping. In other words, 
    \[ \Big\|\sum_{s\in\I_\si} 1_{Q[\vec\de;\si](s)}\Big\|_\infty\lesssim (\log R)^n. \]
\end{lemma}

\begin{proof}
    The strategy is to partition $Q[\vec\de;\si]$ into  a collection $\cU=\cU[\vec\de;\si]$ of $\lesssim(\log R)^n$ rectangular boxes:
     \[ Q[\vec\de;\si]=\bigsqcup_{U\in\cU}U. \]
It then suffices to  show that for each $U\in \cU$, $\{U(s):s\in\I_\si\}$ is $O(1)$-overlapping.

We first partition $Q[\vec\de;\si]$ into rectangular boxes
\begin{equation}\label{partitionV}
    Q[\vec\de;\si]=\bigsqcup_{k\le m\le n}V_{m,+}\sqcup V_{m,-}, 
\end{equation} 
where
\begin{equation}
\begin{split}
         V_{m,+}=\{ (a_i): |a_i|\le R\rho_{i-1} (1\le i\le k-1), |a_i|\le (2\si)^{i-n-1} (k\le i\le m-1),\\
         (2\si)^{m-n-1}<a_m\le \si^{m-n-1},|a_i|\le \si^{i-n-1} (m+1\le i\le n+1)   \}.
\end{split}
\end{equation} 
Changing the sign of $a_m$ gives $V_{m,-}$.
Next, we partition each $V_{m,+}$ (and $V_{m,-}$) further into pieces. 

\bigskip

We  discuss a general setting. Suppose that we are given dyadic numbers $r_1\ge r_2\ge\dots\ge r_m>0$ and a rectangular box
\[ D=\prod_{i=1}^{m-1}[-r_i,r_i]\times [\frac12 r_m,r_m]. \]
We claim that we can partition $D$ as:
\[ D=\bigsqcup_{\vec\mu,\vec\iota} D_{\vec\mu,\vec\iota}. \]
Here, $\vec\mu=(\mu_1,\dots,\mu_{m-1})$ are dyadic numbers with $\mu_{i+1}\le \mu_i\le r_i$ ($1\le i\le m-1$) and we set $\mu_m=r_m$; $\vec\iota=(\iota_1,\dots,\iota_{m-1})$ satisfies $\iota_i\in \{-1,0,1\}$, and $\iota_i=0$ if $\mu_i=\mu_{i+1}$;
and
\[ D_{\vec\mu,\vec\iota}=\prod_{i=1}^{m-1}I_i\times [\frac12 r_m,r_m], \]
where
\begin{equation}
    I_i=\begin{cases}
        [-\mu_i,\mu_i] & \iota_i=0,\\
        (\frac12 \mu_i,\mu_i] & \iota_i=1,\\
        [-\mu_i,-\frac12\mu_i) & \iota_i=-1.
    \end{cases}
\end{equation}
We obtain the sets $D_{\vec\mu,\vec\iota}$ by the following algorithm starting from the $(m-1)$-st component to the $1$st component. We briefly discuss it. Starting with the $(m-1)$th component, we partition $[-r_{m-1},r_{m-1}]$ into dyadic intervals
\[ [-r_{m-1},r_{m-1}]=\bigsqcup_{i=0}^l(2^{-i-1}r_{m-1},2^{-i}r_{m-1}]\sqcup [ -2^{-i}r_{m-1},-2^{-i-1}r_{m-1} )\bigsqcup [-r_m,r_m],  \]
where $2^{l+1}=\frac{r_{m-1}}{r_m}$.
If we choose $(2^{-i-1}r_{m-1},2^{-i}r_{m-1}]$ to be the $(m-1)$th component, we set $\mu_{m-1}=2^{-i}r_{m-1}, \iota_{m-1}=1$; If we choose $[ -2^{-i}r_{m-1},-2^{-i-1}r_{m-1} )$ to be the $(m-1)$th component, we set $\mu_{m-1}=2^{-i}r_{m-1}, \iota_{m-1}=-1$; If we choose $[-r_m,r_m]$ to be the $(m-1)$th component, we set $\mu_{m-1}=r_m, \iota_{m-1}=0$. We continue this process until we obtain $\vec\mu,\vec\iota$, and hence we obtain $D_{\vec\mu,\vec\iota}$.

\bigskip

Returning  to \eqref{partitionV},
we partition the first $m$ components of $V_{m,\pm}$ using the above algorithm:
\[ V_{m,\pm}=\bigsqcup_{\vec\mu,\vec\iota}V_{m,\pm;\vec\mu,\vec\iota}. \]
Hence, we obtain the partition
\[ Q[\vec\de;\si]=\bigsqcup_{k\le m\le n} \bigsqcup_{\vec\mu,\vec\iota}V_{m,+;\vec\mu,\vec\iota}\sqcup V_{m,-;\vec\mu,\vec\iota}=:\bigsqcup_{\cU} U. \]

Fix a $U=V_{m,+;\vec\mu,\vec\iota}$. We verify that $\{U(s):s\in \I_\si\}$ is $O(1)$-overlapping. The other sign choices are handled in the same way, so  we may assume all $\iota_i\ge 0$. It is convenient to set $\mu_m=\si^{m-n-1},\mu_{m+1}=\si^{m-n}$. We use the rough form of $U$ as
\begin{equation}\label{U}
    \{ (a_i): a_i \prec \mu_i (1\le i\le m-1), a_m\sim \si^{m-n-1}, |a_i|\lesssim \si^{i-n-1} (m+1\le i\le n+1)   \}, 
\end{equation} 
where $a_i \prec \mu_i$ means either $\sim$ or $\lesssim$, and if it means $\lesssim$, then we must have $\mu_i=\mu_{i+1}$.

\medskip
We prove the following lemma.
\begin{lemma}\label{emptylem}
Let $U$ be given by \eqref{U}.
    For fixed $1\le l\le m$, if $\frac{\mu_{l+1}}{\mu_l}\ll |s-s'|\ll \min\{ 1,\frac{\mu_2}{\mu_1},\dots,\frac{\mu_l}{\mu_{l-1}} \}$ (the upper bound is $1$ when $l=1$), then
\[ U(s)\bigcap U(s') =\emptyset. \]
\end{lemma}

\begin{proof}
If $\mu_{l+1}=\mu_l$, then there are no $s,s'$ such that $1\ll |s-s'|\ll 1$, so the conclusion automatically holds. Therefore, we assume $\mu_{l+1}<\mu_{l}$, and hence $a_l\sim \mu_l$.

    Write $s'=s+\De$ with $|\De|\ll \min\{ 1,\frac{\mu_2}{\mu_1},\dots,\frac{\mu_l}{\mu_{l-1}} \}$. Suppose, toward a contradiction, that $z\in U(s')\bigcap U(s)$. Write $z$ as
\[ z=\sum_{i=1}^{n+1} a_i\bfe_{i}(s+\De),
\]
where $a_i$ satisfies the range in \eqref{U}.

We compute
\begin{equation}\label{aldominate}
    \langle z,\bfe_{l+1}(s)\rangle=\sum_{i=1}^{l}\bigg(a_i \frac{\De^{l+1-i}}{(l+1-i)!}\kappa_i(s)\cdots\kappa_{l}(s)+a_i O(\De^{l+2-i})\bigg)+O(\mu_{l+1}). 
\end{equation}
Since $z\in U(s)$, we have $|\langle z,\bfe_{l+1}(s)\rangle|\lesssim \mu_{l+1}$. We will derive a contradiction by showing
\[ |\De|\lesssim \frac{\mu_{l+1}}{\mu_l}. \]

Note that $|\ka_i(s)|\sim 1$ and $|\De|\ll \min\{1,\frac{\mu_2}{\mu_1},\dots,\frac{\mu_l}{\mu_{l-1}}\}$. We claim that for $1\le i\le l-1$, we have
\[ |a_i||\De|^{l+1-i}\ll |a_l||\De|. \]
This is true since
\[ |a_i| |\De|^{l-i}\ll \mu_i \prod_{j=i}^{l-1} \frac{\mu_{j+1}}{\mu_j}=\mu_l\sim |a_l|. \]
Substituting this into \eqref{aldominate}, we see that the $a_l\De$ term dominates, and hence using $|\langle z,\bfe_{l+1}(s)\rangle|\lesssim \mu_{l+1}$, we get
\[ |a_l||\De|\lesssim \mu_{l+1}, \]
which implies
\[ |\De|\lesssim \frac{\mu_{l+1}}{|a_l|}\sim \frac{\mu_{l+1}}{\mu_l}. \]
This finishes the proof.
\end{proof}

It remains to show
\begin{equation}\label{finiteoverlap}
    \Big\|\sum_{s\in\I_\si}1_{U(s)}\Big\|_\infty \lesssim 1.
\end{equation}
The proof relies on Lemma \ref{emptylem}. Choose $K$ large enough so that the implicit constants in the  ``$\ll$" and $\gg$ conditions in Lemma \ref{emptylem} are accounted for. For each number $0<\beta\le 1$, let $\cJ_\beta$ be a collection of intervals of length $\beta$ that form a finitely overlapping cover of $[0,1]$. If $J\subset [0,1]$ is an interval, we denote $\I_\si(J):=\I_\si\cap J$.

For $2\le l\le m+1$, let $\beta_l=\min\{1,\frac{\mu_2}{\mu_1},\dots,\frac{\mu_l}{\mu_{l-1}}\}$. Define $\beta_1=1$. We have
\[ \Big\|\sum_{s\in\I_\si}1_{U(s)}\Big\|_\infty\lesssim K\sup_{J_1\in\cJ_{K^{-1}\beta_1}}\Big\|\sum_{s\in\I_\si(J_1)}1_{U(s)}\Big\|_\infty\lesssim K \sup_{J_2\in\cJ_{K\beta_2}}\Big\|\sum_{s\in\I_\si(J_2)}1_{U(s)}\Big\|_\infty. \]
Here, the first inequality is by the triangle inequality; in the second inequality, we used Lemma \ref{emptylem} with $l=1$. Similarly, we have
\begin{equation}
    \begin{split}
        \sup_{J_l\in \cJ_{K\beta_l}}\Big\|\sum_{s\in\I_\si(J_l)} 1_{U(s)}\Big\|_\infty &\lesssim K^2 \sup_{J_l\in \cJ_{K^{-1}\beta_l}}\Big\|\sum_{s\in\I_\si(J_l)} 1_{U(s)}\Big\|_\infty
        \\&
        \lesssim K^2\sup_{J_{l+1}\in \cJ_{K\beta_{l+1}}}\Big\|\sum_{s\in\I_\si(J_{l+1})} 1_{U(s)}\Big\|_\infty.
    \end{split}
\end{equation}
Iterating and noting that $K$ is a fixed constant, we get
\[ \Big\|\sum_{s\in\I_\si}1_{U(s)}\Big\|_\infty\lesssim \sup_{J\in \cJ_{\beta_{m+1}}}\Big\|\sum_{s\in\I_\si(J)}1_{U(s)}\Big\|_\infty \lesssim 1. \]
The last inequality follows from $\beta_{m+1}\le \frac{\mu_{m+1}}{\mu_m}=\si$, and hence $\#\I_\si(J)\lesssim 1$ for $J\in\cJ_{\beta_{m+1}}$. 
\end{proof}

We state a geometric result which is similar to Lemma \ref{structurelem1}.
\begin{lemma}\label{structurelem2}
Suppose $\si\in[\si_{k},\si_{k-1})$. Let $Q[\vec\de;\le\si]$ be defined in Definition \ref{defQsigma}. 
If $|s-s'|\ll \si$, then $Q[\vec\de;\le\si](s)$ and $Q[\vec\de;\le\si](s')$ are comparable.
\end{lemma}

\begin{proof}

We follow the same procedure as in the proof of Lemma \ref{structurelem1}. Some cases have been verified in the proof of Lemma \ref{structurelem1}. Let $z\in Q[\vec\de;\le\si](s')$. We just need to verify $|\langle z,\bfe_j(s)\rangle |\lesssim \si^{j-n-1}$ for $j\ge k$. By Taylor expansion, we just need to verify $R\rho_{i-1}\De^{j-i} \lesssim \si^{j-n-1}$ for any $i\le k-1, j\ge k, \De\ll \si$.  In other words, we verify $R\rho_{i-1}\lesssim \si^{i-n-1}$ for $i\le k-1$. This is true since $\si\le \si_{k-1}$.
\end{proof}

\bigskip

For $0<\si\le 1$, let $\I_\si$ be a maximal set of $\si$-separated points in $[0,1]$. We temporarily fix a $\vec\de$ and may assume all components of $\vec\nu$ are $\ge 0$.
To simplify our notation, we will temporarily drop $\vec\de$ and $\vec\nu$ from the notation in the rest of the section. (We may still need to keep $\vec\de$ in the later sections.) Hence, $P[\vec\de](s)$,$ Q[\vec\de]$,$Q[\vec\de;\si](s)$,$Q[\vec\de;\le\si](s)$ are  denoted by $P(s)$,$Q(s)$,$Q[\si](s)$,$ Q[\le\si](s)$. We also use the lowercase letters to denote their $R^{-1}$-dilation (centered at the origin): $p(s),q(s),q[\si](s), q[\le\si](s)$. For a rectangular box $U$, we use $U^*$ to denote the dual rectangular box (which is centered at the origin). For each $\si$ and $s'$, we tile $\R^{n+1}$ by translates $\{U\}$ of $q[\le \si](s')^*$, which we denote by $U\parallel (q[\le \si](s'))^*$. 

We have the following Kakeya-type estimate.

\begin{proposition}\label{wee}
    Suppose that $\{g_s\}_{s\in\I_\de}$ satisfy $\supp\wh g_s\subset q(s)$. Then we have
    \begin{equation}\label{weeineq}
        \int_{\R^{n+1}} \Big|\sum_{s\in\I_\de}g_s \Big|^{2}\lesssim \delta^{-\e} \sum_{\substack{ \de\le \si\le 1: \\ dyadic }  } \sum_{s'\in\I_\si} \sum_{U\parallel (q[\le \si](s'))^*} |U|(\int  \sum_{s\in\I_\de:|s-s'|\le \si}|g_{s}| w_U)^2.
    \end{equation}
\end{proposition}

Recall that $w_U\sim \frac{1}{|U|}$ in $U$ and decays rapidly away from $U$.

\begin{proof}
    Recall that we have the partition
    \begin{equation}\label{partunit}
        q(s)=\bigsqcup_{\si} q[\si](s), 
    \end{equation} 
    where $q[\si](s)=q[\le \si](s)\big\backslash q[\le2\si](s)$ when $\si<1$; $q[1](s)=q[\le1](s)$ is a $R^{-1}$-cube. Let $\phi_{q[\le\si](s)}$ be a bump function adapted to $q[\le\si](s)$. Let $\phi_{q[\si](s)}=\phi_{q[\le\si](s)}-\phi_{q[\le2\si](s)}$, which is a bump function adapted to $q[\si](s)$. $\{\phi_{q[\si](s)}\}_\si$ is a partition of unity subordinate to \eqref{partunit}. We have 
    \[ \wh g_s=\sum_\si \wh g_s \phi_{q[\si](s)}. \]
    Using Plancherel, the LHS of \eqref{weeineq} is
    \[ \lesssim (\log R)^{O(1)}\sum_\si \int_{\R^{n+1}}\Big|\sum_{s\in\I_\de}\wh g_s\phi_{q[\si](s)} \Big|^2. \]
    Using Lemma \ref{finiteoverlaplem}, it is bounded by
    \begin{align}
        &\lesssim (\log R)^{O(1)}\sum_\si \sum_{s'\in\I_\si}\int_{\R^{n+1}}\Big| \sum_{s\in\I_\de:|s-s'|\le \si}\wh g_s\phi_{q[\si](s)} \Big|^2\\
        \label{bound2}&\le(\log R)^{O(1)}\sum_\si \sum_{s'\in\I_\si}\int_{\R^{n+1}}\Big|\sum_{s\in\I_\de:|s-s'|\le \si}|g_s|*|\phi_{q[\si](s)}^\vee| \Big|^2.
    \end{align} 
It is easy to see $q[\le2\si](s)$ and $q[\le \si](s)$ are comparable.
By Lemma \ref{structurelem2}, $q[\le\si](s)$ is comparable to $q[\le\si](s')$ if $|s-s'|\ll \si$. It is harmless to assume $q[\le\si](s)$ is comparable to $q[\le\si](s')$ if $|s-s'|\le \si$ by losing a constant factor. We see that
\[ |\phi_{q[\si](s)}^\vee|\lesssim w_{q[\le\si](s')^*}. \]
Here, $w_{q[\le\si](s')^*}$ $\sim \frac{1}{|q[\le\si](s')^*|}$ on $q[\le \si](s')^*$ and decays rapidly away from $q[\le\si](s')^*$.

Therefore, 
\begin{align}
    \eqref{bound2}&\lesssim  (\log R)^{O(1)}\sum_\si \sum_{s'\in\I_\si}\int_{\R^{n+1}}(\sum_{s\in\I_\de:|s-s'|\le \si}|g_s|*w_{q[\le\si](s')^*})^2 \\
    &\lesssim  (\log R)^{O(1)}\sum_\si \sum_{s'\in\I_\si}\sum_{U\parallel q[\le\si](s')^*} |U|  (\int \sum_{s\in\I_\de:|s-s'|\le \si}|g_s|w_U)^2.
\end{align} 
Since $\delta \lesssim R^{-\frac1n}$,
    this completes the proof. 
\end{proof} 

We will use the following form of the Kakeya-type estimates. The first estimate is related to a nondegenerate curve in $\mathbb{R}^n$, and the second estimate is related to a cone generated by a nondegenerate curve. Recall the notation from Section \ref{0701.sec2}. We postpone the proofs of the propositions to Section \ref{0702.53}, where we prove some standard interpolation lemmas.  

\begin{proposition}\label{0616.thm33}
     For $2 \leq q < \infty$, $R \geq 1$, and $\e>0$, we have
    \begin{equation*}
        \int_{B_R} (\sum_{s\in\I_\de}|f_{s}|^q*w_{p(s)^*})^{2}\lesssim R^{\e} \sum_{\substack{ \de\le \si\le 1: \\ dyadic } } \sum_{s'\in\I_\si} \sum_{\substack{ U\parallel (q[ \le \si](s'))^*: \\ U \subset B_R }   } |U| \Big(\int \sum_{s:|s-s'|\le \si}|f_{s}|^qw_U \Big)^2
    \end{equation*}
    for any functions $\{f_s\}_{s\in\I_{\de}}$ satisfying $\supp \wh f_s\subset p(s)$. 
\end{proposition}

\begin{proposition}\label{0501.prop22} For any $2 \leq q < \infty$, $R \geq 1$, and $\epsilon>0$, we have
\begin{equation*}
    \int_{B_R} (\sum_{\theta \in \Xi_n(R) }|f_{\theta}|^{q} *w_{\theta^*} )^2 \lesssim R^{\e} \sum_{ \substack{ R^{-\frac{1}{n-1}} \leq s \leq 1: \\ dyadic } } \sum_{\tau \in \Xi_n(s^{-(n-1)}) } \sum_{ \substack{ U \| U_{\tau,R} \\ U \subset B_R  }  }|U|
    \Big( \int \sum_{\theta \subset \tau} |f_{\theta}|^q w_U \Big)^2. 
\end{equation*}    
\end{proposition}

Let us summarize the role of this section. Starting from the intermediate-scale planks arising in the reduction to Theorem \ref{05.29.thm28}, we decomposed them into smaller pieces and formulated the Kakeya-type estimates that control their overlap. These estimates provide the geometric input needed to pass from localized frequency pieces to wave-envelope type quantities. In the next sections, we develop the complementary analytic input: Section \ref{section4} reduces the wave envelope estimates to weighted $\ell^{p/2}$-function estimates, and Sections \ref{0627.sec6}–\ref{sec7} prove these estimates by the high/low method. These ingredients will then be combined in Section \ref{0624.sec8} to prove Theorem \ref{05.29.thm28}.

\section{\texorpdfstring{$\ell^{p/2}$}{}-function estimates}\label{section4}

In Section 3, we reduced the proof of the local smoothing estimate to Theorem \ref{05.29.thm28}. After establishing the Kakeya-type estimates for intermediate-scale planks in Section 4, we now prepare the other main analytic input for the proof. In this section, we reduce the wave envelope estimates for curves and cones (Theorem \ref{0501.thm12} and Theorem \ref{05.03.thm13}) to suitable $\ell^{p/2}$-function estimates (Theorem \ref{25.01.30.thm51} and Theorem \ref{25.04.26.thm13}), and then to their weighted versions. These weighted estimates will be proved in Sections \ref{0627.sec6}–\ref{sec7} by the high/low method and will be used in Section \ref{0624.sec8} to prove Theorem \ref{05.29.thm28}.
Before we begin the discussion, let us first give a precise definition of weight functions.

\subsection{Weight functions}\label{sec41}

Let $\phi:\R^{n}\to[0,\infty)$ be a radial, smooth bump function supported in $|\xi|\le 1/4$ and satisfying $|\widecheck{\phi}(x)|>0$ for all $|x|\le 1$. For any $d\in\mathbb{N}$, define $W^{n,d}:\R^{n}\to[0,\infty)$ by 
\[ W^{n,d}(x):=\sum_{j=0}^\infty \frac{1}{2^{100n^2jd}}|\widecheck{\phi}|^2(2^{-j}x). \]

Let $B_0$ be the unit ball centered at the origin in $\R^n$. For any set $U=T(B_0)$ where $T$ is an affine transformation $T:\R^n\to\R^n$, define
\[ W_{U}^{n,d}(x):=W^{n,d}(T^{-1}(x)). \]
If $U\subset\R^n$ is a convex set, then we write $W_U^{n,d}$ to mean $W_{\tilde{U}}^{n,d}$, where $\tilde{U}=T(B_0)$ for some affine transformation $T$ and $\tilde{U}$ is comparable to $U$. This defines $W_U^{n,d}$ up to a bounded constant, which is sufficient for our arguments.

\begin{definition} \label{M3ballweight} Let $d\in \mathbb{N}$ and let $W^{n,d}$ be the weight function.  Let $B_0\subset\R^n$ denote the unit ball centered at the origin. For any set $U=T(B_0)$ where $T$ is an affine transformation $T:\R^n\to\R^n$, define
\[ w_{U,d}(x):=|U|^{-1}W^{n,d}(T^{-1}(x)). \]
\end{definition}

When it is clear in the context, we denote  $w_{U,d}$ by $w_U$.
The function $w_U$ is essentially supported on  $U$ and its Fourier transform is supported in a constant multiple of $U^*$.

\subsection{\texorpdfstring{$\ell^{p/2}$}{}-function estimates without weights} Let us state the $\ell^{p/2}$-function estimate for the nondegenerate curve $\gamma_n$.

\begin{theorem}[$\ell^{p/2}$-estimate for the nondegenerate curve]\label{25.01.30.thm51} Let $n \geq 2$, $R \geq 1$, and $\epsilon>0$. For any $4 \leq p \leq  n^2+n-2$, we have
       \begin{equation*}
        \|f\|_{L^p(\R^n)} \leq C_{\gamma_n,\e}R^{\e}  R^{\frac1n(\frac12-\frac2p)}  \Big\|(\sum_{\theta \in \Theta^n(R) }|f_{\theta}|^{\frac{p}{2}})^{\frac2p} \Big\|_{L^p(\R^n)}
    \end{equation*}
    for any Schwartz functions $f: \R^n \rightarrow \mathbb{C}$ whose Fourier supports are in $\mcM^n(\gamma_n;R)$. The constant $C_{\gamma_n,\e}$ depends only on  $\e$ and the bounds of the determinant of Wronskian matrix of the curve $\gamma_n$.
\end{theorem}

Theorem \ref{0501.thm12} follows immediately from Theorem \ref{25.01.30.thm51} and Proposition \ref{0616.thm33}. Namely, since the Fourier support of $f_{\theta}$ is contained in $(10n)\theta$, we have  
\begin{equation}\label{0623.41}
    |f_{\theta}|^{\frac{p}{2}} \lesssim \big||f_{\theta}|^2 * w_{\theta^*} \big|^{\frac{p}{4}} \lesssim |f_{\theta}|^{\frac{p}{2}}*w_{\theta^*}.
\end{equation}
So we have
\begin{equation*}
    \begin{split}
        \|f\|_{L^p(\R^n)} &\lesssim_{\e} R^{\e}  R^{\frac1n(\frac12-\frac2p)}  \Big\|(\sum_{\theta \in \Theta^n(R) }|f_{\theta}|^{\frac{p}{2}})^{\frac2p} \Big\|_{L^p(\R^n)}
        \\& \lesssim_{\e} 
        R^{\e}  R^{\frac1n(\frac12-\frac2p)}  \Big\|(\sum_{\theta \in \Theta^n(R) }||f_{\theta}|^2* w_{\theta^*}|^{\frac{p}{4}})^{\frac2p} \Big\|_{L^p(\R^n)}
        \\&
        \lesssim_{\e}
        R^{\e}  R^{\frac1n(\frac12-\frac2p)}  \Big\|(\sum_{\theta \in \Theta^n(R) }|f_{\theta}|^{\frac{p}{2}}*w_{\theta^*} )^{\frac2p} \Big\|_{L^p(\R^n)}
        \\& \lesssim_{\e} 
        R^{2\e}   R^{\frac1n(\frac12-\frac2p)}  \Big(\sum_{ \substack{ R^{-\frac1n} \leq s \leq 1: \\ dyadic } } \sum_{\tau \in \Theta^n(s^{-n}) } \sum_{U \| U_{\tau,R}  } |U| \|S_{U,{\frac{p}{2}}}^{curve}f\|_{L^{\frac{p}{2}}(\R^n)}^p \Big)^{\frac1p}.
    \end{split}
\end{equation*}
So to prove Theorem \ref{0501.thm12}, it suffices to prove Theorem \ref{25.01.30.thm51} and Proposition \ref{0616.thm33}. 
 We next state the $\ell^{p/2}$-function estimate for a cone generated by a nondegenerate curve.

\begin{theorem}[$\ell^{p/2}$-estimate for the cone]\label{25.04.26.thm13} Let $n \geq 3$, $R \geq 1$, and $\epsilon>0$.
    For any $4 \leq p \leq n^2-n-2$,
       \begin{equation*}
        \|f\|_{L^p(\R^n)} \leq C_{\gamma_n,\e}R^{\e}R^{\frac{1}{n-1}(\frac12-\frac2p)}  \Big\|(\sum_{\theta \in \Xi_n(R) }|f_{\theta}|^{{\frac{p}{2}}})^{\frac2p} \Big\|_{L^p(\R^n)}
    \end{equation*}
    for any Schwartz functions $f: \R^n \rightarrow \mathbb{C}$ whose Fourier supports are in $\Gamma_n(\gamma_n; R)$.
\end{theorem}

Theorem \ref{05.03.thm13} follows immediately from Theorem \ref{25.04.26.thm13} and Proposition \ref{0122.prop32}. 

\begin{proposition}\label{0122.prop32}
    For any $4 \leq p < \infty$, $R \geq 1$, and $\e>0$, we have
    \begin{equation*}
        \Big\|(\sum_{\theta \in \Xi_n(R) }|f_{\theta}|^{{\frac{p}{2}}} * w_{\theta^*} )^{\frac2p} \Big\|_{L^p}^p \leq C_{\gamma_n,\e}R^{\e}   \sum_{ \substack{ R^{-\frac{1}{n-1}} \leq s \leq 1: \\ dyadic } } \sum_{\tau \in \Xi_n(s^{-(n-1)}) } \sum_{U \| U_{\tau,R}  } |U| \|S_{U,\frac{p}{2}}^{cone}f\|_{L^{\frac{p}{2}}(\R^n)}^p. 
    \end{equation*}
\end{proposition}

Note that
Proposition \ref{0122.prop32} follows from Proposition \ref{0501.prop22}.

\subsection{\texorpdfstring{$\ell^{p/2}$}{}-function estimates with weights}\label{sec31}

We prove that Theorem \ref{25.01.30.thm51} follows from the $\ell^{p/2}$-function estimate with a weight. Let us first state the $\ell^{p/2}$-function estimate with weights (Proposition \ref{05.03.prop32}).

\begin{proposition}\label{05.03.prop32} For any $4 \leq p \leq n^2+n-2$, we have
       \begin{equation*}
        \|f\|_{L^p(\R^n)} \leq C_{\e,\gamma_n}R^{\e}  R^{\frac1n(\frac12-\frac2p)}  \Big\|(\sum_{\theta \in \Theta^n(R) }|f_{\theta}|^{{\frac{p}{2}}}*w_{\theta^*})^{\frac2p} \Big\|_{L^p(\R^n)}
    \end{equation*}
    for any Schwartz functions $f: \R^n \rightarrow \mathbb{C}$ whose Fourier supports are in $\mcM^n(\gamma_n;R)$.
\end{proposition}

Let us prove that Proposition \ref{05.03.prop32} implies Theorem \ref{25.01.30.thm51}. It suffices to prove
\begin{equation}
    \Big\|(\sum_{\theta \in \Theta^n(R) }|f_{\theta}|^{{\frac{p}{2}}}*w_{\theta^*})^{\frac2p} \Big\|_{L^p(\R^n)} \lesssim R^{\epsilon} \Big\|(\sum_{\theta \in \Theta^n(R) }|f_{\theta}|^{{\frac{p}{2}}})^{\frac2p} \Big\|_{L^p(\R^n)}.
\end{equation}
We apply Proposition \ref{0616.thm33} to the left hand side. Then it suffices to prove
\begin{equation*}
    \sum_{ \substack{ R^{-\frac1n} \leq s \leq 1: \\ dyadic } } \sum_{\tau \in \Theta^n(s^{-n}) } \sum_{U \| U_{\tau,R}  } |U| \|S_{U,\frac{p}{2}}^{curve}f\|_{L^{\frac{p}{2}}(\R^n)}^p \lesssim R^{\epsilon} \Big\|(\sum_{\theta \in \Theta^n(R) }|f_{\theta}|^{{\frac{p}{2}}})^{\frac2p} \Big\|_{L^p(\R^n)}^p. 
\end{equation*}    
To prove this, since there are only $O(\log R)$ many choices of $s$, it suffices to  check 
\begin{equation}
    \sum_{\tau} \sum_U |U| \Big(\int |S_{U,\frac{p}{2}}^{curve}f|^{\frac{p}{2}} \Big)^2 \lesssim R^{\epsilon} \int (\sum_{\theta } |f_{\theta}|^{\frac{p}{2}} )^2. 
\end{equation}
By definition, the left hand side is equal to
\begin{equation}
    \sum_{\tau}\sum_U |U| \Big(\int \sum_{\theta \subset \tau}|f_{\theta}(x)|^{\frac{p}{2}}w_U(x) \Big)^2.
\end{equation}
By H\"{o}lder's inequality, this is bounded by
\begin{equation}
\begin{split}
    &\lesssim R^{\epsilon} \sum_{\tau}\sum_U  \|w_U\|_1 |U| \int \big( \sum_{\theta \subset \tau} |f_{\theta}(x)|^{\frac{p}{2}}  \big)^2 w_U(x)\,dx \\&\lesssim R^{\epsilon}  \sum_{\tau}\int (\sum_{\theta \subset \tau} |f_{\theta}|^{\frac{p}{2}} )^2
     \lesssim R^{\epsilon} \int (\sum_{\theta} |f_{\theta}|^{\frac{p}{2}} )^2.
\end{split}
\end{equation}
This completes the proof.
Therefore, to prove Theorem \ref{0501.thm12}, it suffices to prove Proposition \ref{05.03.prop32}.
\\

We can prove an analogous result for the cone. By following the same argument, one can show that Proposition \ref{05.04.prop33} implies Theorem \ref{25.04.26.thm13}. We omit the proof here.

\begin{proposition}\label{05.04.prop33}
    For any $4 \leq p \leq n^2-n-2$,
       \begin{equation*}
        \|f\|_{L^p(\R^n)} \leq C_{\e,\gamma_n}R^{\e} R^{\frac{1}{n-1}(\frac12-\frac2p)} \Big\|(\sum_{\theta \in \Xi_n(R) }|f_{\theta}|^{{\frac{p}{2}}}*w_{\theta^*} )^{\frac2p} \Big\|_{L^p(\R^n)}
    \end{equation*}
    for any Schwartz functions $f: \R^n \rightarrow \mathbb{C}$ whose Fourier supports are in $\Gamma_n(\gamma_n;R)$.
\end{proposition}

 We have reduced the wave envelope estimates for curves and cones, Theorems \ref{0501.thm12} and \ref{05.03.thm13}, to the weighted $\ell^{p/2}$ estimates, Propositions \ref{05.03.prop32} and \ref{05.04.prop33}. Thus the remaining task is to prove these weighted estimates. In the next section, we set up the high/low framework needed for this purpose. We then prove the weighted estimate for curves in Section \ref{0703.sec6} by induction on the dimension, using the corresponding cone estimate as an auxiliary input; the cone estimate is established in Section \ref{sec7} by a bootstrapping argument. Together with the Kakeya-type estimates from Section \ref{sec3}, these results will be used in Section \ref{0624.sec8} to prove the main estimate, Theorem \ref{05.29.thm28}.

\section{Set-up for the high/low method}\label{0627.sec6}

We introduce some definitions and lemmas for the high/low method. Most of the lemmas below are already contained in \cite{guth23}. In their paper, they considered the model curve $\gamma(s)=(s,s^2,\ldots,s^n)$. In our paper, we consider a nondegenerate curve $\gamma_n$, but the same proof works for the lemmas we are going to state below. 

\subsection{Reduction to polynomial curves}\label{51red}

For given $0 < A < \infty$, we define a collection $\mathcal{C}_n(A,R,\e)$ of polynomial curves
\begin{equation}\label{0623.51def}
    \gamma_n(R;s)=\big(\frac{s^n}{n!},\frac{s^{n-1}}{(n-1)!},\ldots,\frac{s}{1!} \big) + R^{-\e}s^{n+1}E_{\e}(s)
\end{equation}
where $E_{\e}(s)$ is an $n$-tuple of polynomials of degree at most $\e^{-1}$ and the coefficients of monomials of $E_{\e}(s)$ are all bounded by $A$.

\begin{proposition}\label{0621.prop51}
Let $n \geq 2$.
Fix $\e, A>0$. For any $\gamma_n \in \mathcal{C}_n(A,R,\e)$,
for any $4 \leq p \leq  n^2+n-2$, we have
       \begin{equation*}
        \|f\|_{L^p(\R^n)} \leq C_{A,\e}R^{200^n\e^{\frac12}}  R^{\frac1n(\frac12-\frac2p)}  \Big\|(\sum_{\theta \in \Theta^n(R) }|f_{\theta}|^{\frac{p}{2}} * w_{\theta^*} )^{\frac2p} \Big\|_{L^p(\R^n)}
    \end{equation*}
    for any Schwartz functions $f: \R^n \rightarrow \mathbb{C}$ whose Fourier supports are in $\mcM^n(\gamma_n;R)$. Here the constant $C_{A,\e}$ depends only on the parameter $A$ and $\e$.
\end{proposition}

\begin{proposition}\label{0621.prop52}
Let $n \geq 3$.
Fix $\e, A>0$. For any $\gamma_n \in \mathcal{C}_n(A,R,\e)$,
for any $4 \leq p \leq  n^2-n-2$, we have
       \begin{equation*}
        \|f\|_{L^p(\R^n)} \leq C_{A,\e}R^{200^n\e^{\frac12}}  R^{\frac{1}{n-1}(\frac12-\frac2p)}  \Big\|(\sum_{\theta \in \Xi_n(R) }|f_{\theta}|^{\frac{p}{2}}* w_{\theta^*} )^{\frac2p} \Big\|_{L^p(\R^n)}
    \end{equation*}
    for any Schwartz functions $f: \R^n \rightarrow \mathbb{C}$ whose Fourier supports are in $\Gamma_n(\gamma_n;R)$. Here the constant $C_{A,\e}$ depends only on the parameter $A$ and $\e$.
\end{proposition}

To prove Proposition \ref{05.03.prop32}, it suffices to consider polynomial curves in the collection $\mathcal{C}_n(A,R,\e)$. The reduction to a polynomial curve is fairly standard, for example, see Section 2 of \cite{MR4711363} or Section 7.2 of \cite{MR4047925}.

\begin{proposition}\label{0621.prop53}
      Proposition \ref{0621.prop51} implies 
    Proposition \ref{05.03.prop32}. 
    
    Proposition \ref{0621.prop52} implies Proposition \ref{05.04.prop33}.
\end{proposition}

\begin{proof}[Proof of Proposition \ref{0621.prop53}]
    The proofs of the two statements are the same. We prove only the first statement. Let us fix a smooth nondegenerate curve $\gamma_n:[0,1] \rightarrow \mathbb{R}^n$. Divide the interval $[0,1]$ into smaller subintervals of  length $R^{-\e}$. Since we are allowed to lose $R^{O(\e)}$ in the estimate, after translation, it suffices to consider $\gamma_n(s)$ on the domain $s \in [0,R^{-\e}]$. We write
    \begin{equation}
        \gamma_n(s)=\gamma_{n,W}(s)+\Delta_{W}(s)
    \end{equation}
    where $\gamma_{n,W}$ is the Taylor expansion of $\gamma_n$ up to the order $W$. Take $W=10/\e$. Then it is easy to see that
    \begin{equation}
        |\partial_s^{j} \Delta_W(s)| \leq C R^{-2}
    \end{equation}
for all $0 \leq j \leq n$ and $s \in [0,R^{-\e}]$. Here, $C$ depends on the choice of the curve $\gamma_n$. Since the size of $\Delta_W$ is small, to prove Proposition \ref{05.03.prop32}, it suffices to prove the estimate for $\gamma_{n,W}(s)$ for $s \in [0,R^{-\e}]$. By the nondegeneracy assumption \eqref{12}, if $R$ is sufficiently large, by a linear transformation with the determinant comparable to one, we may write
\begin{equation}
    \gamma_{n,W}(s)=\big(\frac{s^n}{n!},\frac{s^{n-1}}{(n-1)!},\ldots,\frac{s}{1!} \big) + s^{n+1}E_{\e/10}(s)
\end{equation}
where $E_{\e/10}(s)$ is a polynomial curve of degree at most $10\e^{-1}$ defined on $s \in [0,R^{-\e}]$.  Define the rescaling map 
\begin{equation}
    L(\xi)=(R^{n\e}\xi_1,\ldots,R^{\e}\xi_n).
\end{equation}
Define the curve
\begin{equation}
    \widetilde{\gamma}_n(s):= L \big(\gamma_{n,W}(R^{-\e}s) \big).
\end{equation}
Then $\widetilde{\gamma}_n$ is a curve in $\mathbb{R}^n$ defined on the domain $[0,1]$, and we have the expression 
\begin{equation}
 \widetilde{\gamma}_n(s)=\big(\frac{s^n}{n!},\frac{s^{n-1}}{(n-1)!},\ldots,\frac{s}{1!} \big) +R^{-\e}s^{n+1}E_{R,\e/10}(s)   
\end{equation}
where $E_{R,\e/10}(s)=R^{-n\e}  L \big(E_{\e/10}(R^{-\e}s) \big)$. Note that $E_{R,\e/10}(s)$ is a polynomial of degree at most $10\e^{-1}$, and every coefficient of a monomial of $E_{R,\e/10}$ is bounded if $R$ is sufficiently large. Therefore,  by applying scaling to $\gamma_{n,W}$, it is straightforward to see that Proposition \ref{05.03.prop32} for $\widetilde{\gamma}_n$ implies the desired result. It suffices now to note that $\widetilde{\gamma}_n \in \mathcal{C}_n(A,R,\e)$ for some $A$. 
\end{proof}

\subsection{Wave packet decomposition}\label{sec52}
Let $\e$ be a small number such that $\e^{-1}$ is an integer.
Introduce scales 
\begin{equation}
    1 \leq R^{\e} \leq R^{2\e} \leq R^{3\e} \leq \cdots \leq R^{\e^{-1}\e} = R.
\end{equation}
For simplicity, we use the notation $R_k=R^{k \epsilon}$ and $R_N=R$ with $N=\e^{-1}$. Whenever we use the notation $\theta$, it always means that $\theta \in \Theta^n(R)$. When we use the notation $\tau_k$, it usually means that $\tau_k$ is an element of $\Theta^n(R_k)$.
For parameters $\alpha,\beta>0$, define 
\begin{equation}\label{0622.59}
\begin{split}
    & U_{\alpha}:= \{ x \in \R^n: |f(x)| \geq \alpha \},
    \\&U_{\alpha,\beta}:=\{ x \in \R^n: |f(x)| \geq \alpha, \; \frac{\beta}{2} \leq \sum_{\theta}|f_{\theta}|^2 * w_{\theta^*}(x) \leq \beta \}.
\end{split}
\end{equation}

For each $\tau_k \in\Theta^n(R_k)$, for $k=1,2,\ldots,N$, let $\T_{\tau_k}$ be a tiling of $\R^n$ by  translates $T_{\tau_k}$ of $\tau_k^*$. Recall that the function $\phi$ is defined in Subsection \ref{sec41}. For each $m\in\mathbb{Z}^n$, let 
\[ \psi_m(x)=c\int_{[-\frac{1}{2},\frac{1}{2}]^n}|\widecheck{\phi}|^2(x-y-m)dy, \]
where $c$ is chosen so that $\sum_{m\in\mathbb{Z}^n}\psi_m(x)=c\int_{\R^n}|\widecheck{\phi}|^2=1$. Since $|\widecheck{\phi}|$ is a rapidly decaying function, for any $k\in\mathbb{N}$, there exists $C_k>0$ such that
\[ \psi_m(x)\le c\int_{[-\frac{1}{2},\frac{1}{2}]^n}\frac{C_k}{(1+|x-y-m|^2)^{100kn}}dy \le \frac{\tilde{C}_k}{(1+|x-m|^2)^k}. \]
Define the partition of unity $\{\psi_{T_{\tau_k}}\}$ associated to ${\tau_k}$ to be $\psi_{T_{\tau_k}}(x)=\psi_m\circ A_{\tau_k}$, where $A_{\tau_k}$ is a linear transformation taking $\tau_k^*$ to $[-\frac{1}{2},\frac{1}{2}]^n$ and $A_{\tau_k}(T_{\tau_k})=m+[-\frac{1}{2},\frac{1}{2}]^n$. The important properties of $\psi_{T_{\tau_k}}$ are 
\begin{enumerate}
    \item rapid decay off of $T_{\tau_k}$

    \item Fourier support contained in a translate of $\tau_k$ centered at  the origin.
\end{enumerate}
We next do a pigeonholing on the wave packets at the last scale $R$.

\begin{proposition}[Proposition 6.6 of \cite{guth23}] \label{wpd} Fix $\e>0$. Let ${\alpha}>C_\e R^{-100n}\max_\theta\|f_\theta\|_{L^\infty(\R^n)}$. There exist subsets $\tilde{\T}_\theta\subset\T_\theta$ and  a constant $\mathcal{X}>0$ depending on $\alpha$ with the following properties:
\begin{align} 
|U_{\a}|\lesssim (\log R)|\{x\in U_{\a}:\,\,{{\a}}&\lesssim |\sum_{\theta\in{\Theta^n(R)}}\sum_{T\in\tilde{\T}_\theta}\s_T(x)f_\theta (x)|\,\,\}|, \\
R^\e T\cap U_{\a}\neq \emptyset\qquad&\text{for all}\quad\theta\in \Theta^n(R),\quad T\in\tilde{\T}_\theta\\
\mathcal{X}\lesssim \|\sum_{T\in\tilde{\T}_\theta}\s_Tf_\theta\|_{L^\infty(\R^n)}&\lesssim R^{3\e} \mathcal{X} \qquad\text{for all}\quad  \theta\in \Theta^n(R)\label{propM}\\
\|\s_Tf_\theta\|_{L^\infty(\R^n)}&\sim  \mathcal{X} \qquad\text{for all}\quad  \theta\in 
 \Theta^n(R),\quad T\in\tilde{\T}_\theta. \label{prop'M}
\end{align}

\end{proposition}

We define the function
\begin{equation}\label{06.01.57}
    f^N(x):= \sum_{\tau_N} f_{\tau_N}^N(x) :=\sum_{\theta\in 
\Theta^n(R) }\sum_{T\in\tilde{\T}_\theta}\s_T(x)f_\theta (x).
\end{equation}
By replacing $\alpha$ by $\alpha/\mathcal{X}$, we may assume that $\mathcal{X}=1$. In the language of wave packets, this says that all the wave packets at the last scale have the height either comparable to one or zero. Throughout the paper, $f^N$ will play the role of $f$ frequently. We modify the definitions of $U_{\alpha,\beta}$ in \eqref{0622.59}, and use
\begin{equation}\label{0706.515}
    U_{\alpha,\beta}:=\{ x \in U_{\alpha}:  \frac{\beta}{2} \leq \sum_{\tau_N}|f_{\tau_N}^N|^2 * w_{\tau_N^*}(x) \leq \beta \}.
\end{equation}
This is the definition we use for the rest of the paper.
\medskip

\subsection{Interpolation}\label{0702.53}

The advantage of Proposition \ref{wpd} is to enable us to interpolate the $\ell^{q/2}$-function estimates. 
Here, we prove some elementary interpolation lemmas. The interpolation argument is somewhat standard and appears in several places, for example, Lemma 3.4 of \cite{MR3374964} or the proof of Proposition 4.4 of \cite{MR4794594}. Recall that Theorem \ref{0501.thm12} was reduced to Proposition \ref{0621.prop51}. 

\begin{proposition}\label{06.01.prop52}
   Proposition \ref{0621.prop51} for $4 \leq p \leq n^2+n-2$ follows from the endpoint cases $p=4$ and $n^2+n-2$.
\end{proposition}

Here is one more technical interpolation lemma, which is used in the proof of Proposition \ref{05.01.thm61} for the case $n=3$. 

\begin{proposition}\label{05.26.prop82} 
Let $n=3$.
For $8 \leq p \leq 10$,
\begin{equation}
    \sup_{\tau_N}\|f_{\tau_N}^{N}\|_{\infty}^{p-8} \int (\sum_{\tau_N}|f_{\tau_N}^N|^4 *w_{\tau_N^*} )^{2}  \lesssim \int(\sum_{\tau_N}|f_{\tau_N}^N|^{\frac{p}{2}} *w_{\tau_N^*} )^2.
\end{equation}
\end{proposition}

Lastly, we prove the Kakeya-type estimate (Proposition \ref{0616.thm33}) here. Let us recall the statement.

\begin{proposition}\label{0720.thm57}
    Let $2\le q<\infty$. Suppose  $\{f_s\}_{s\in\I_\de}$ satisfies $\supp \wh f_s\subset p(s)$. Then
    \begin{equation*}
        \int_{B_R} (\sum_{s\in\I_\de}|f_{s}|^q*w_{p(s)^* })^{2}\lessapprox \sum_{\de\le \si\le 1} \sum_{s'\in\I_\si} \sum_{\substack{ U\parallel (q[ \vec\de;\le \si](s'))^*: \\ U \subset B_R }   }|U|(\int \sum_{s:|s-s'|\le \si}|f_{s}|^qw_U )^2.
    \end{equation*}
\end{proposition}

Proposition \ref{0501.prop22} can be proved identically by following the proof of Proposition \ref{0720.thm57}. So we omit the proof here.

\begin{proof}[Proof of Proposition \ref{05.26.prop82}]
    Recall that  $\|f_{\tau_N}^{N}\|_{\infty} \sim 1$. We apply the Kakeya-type estimate (Proposition \ref{0616.thm33}) to the left hand side. It suffices to prove that
    \begin{equation*}
        \sum_s  \sum_{\tau \in \Theta^n(s^{-n}) } \sum_{U \| U_{\tau,R}  }|U| \|S_{U,4}f^N\|_{L^{4}(\R^n)}^8 \lesssim \sum_s \sum_{\tau \in \Theta^n(s^{-n}) } \sum_{U \| U_{\tau,R}  } |U| \|S_{U,\frac{p}{2}}f^N\|_{L^{\frac{p}{2}}(\R^n)}^p.
    \end{equation*}
    This is because
    \begin{equation}
        \sum_s \sum_{\tau \in \Theta^n(s^{-n}) } \sum_{U \| U_{\tau,R}  } |U| \|S_{U,\frac{p}{2}}f^N\|_{L^{\frac{p}{2}}(\R^n)}^p \lesssim \int(\sum_{\tau_N}|f_{\tau_N}^N|^{\frac{p}{2}} *w_{\tau_N^*} )^2.
    \end{equation}
    Once we have this, the desired bound follows from H\"{o}lder's inequality and embedding. The above inequality follows from
    \begin{equation}
        \|S_{U,4}f^N\|_4^8 \lesssim \|S_{U,\frac{p}{2}}f^N\|_{\frac{p}{2}}^p.
    \end{equation}
    By definition, this is equivalent to 
    \begin{equation}
        \int \sum_{\tau_N \subset \tau}|f_{\tau_N}^N(x)|^4 w_U(x)\,dx 
        \lesssim \int \sum_{\tau_N \subset \tau}|f_{\tau_N}^N(x)|^{\frac{p}{2}} w_U(x)\,dx.
    \end{equation}
    This follows from
    \begin{equation}\label{0621.514}
        \int |f_{\tau_N}^N(x)|^4 w_U(x)\,dx 
        \sim \int |f_{\tau_N}^N(x)|^{\frac{p}{2}} w_U(x)\,dx
    \end{equation}
    for all $\tau_N$. This is true since  both sides are essentially $\sim |T|\cdot \#\{T\in \tilde{ \mathbb{T}}_\theta:T\subset 2U\}$ as we have assumed the height $\mathcal{X}=1$. This is a standard application of a wave packet decomposition, so we omit the details.
\end{proof}

\begin{proof}[Proof of Proposition \ref{06.01.prop52}]  We need to prove 
 \begin{equation*}
        \|f\|_{L^p(\R^n)} \leq C_{\gamma_n,\e}R^{\e}  R^{\frac1n(\frac12-\frac2p)}  \Big\|(\sum_{\theta \in \Theta^n(R) }|f_{\theta}|^{{\frac{p}{2}}}*w_{\theta^*})^{\frac2p} \Big\|_{L^p(\R^n)}
    \end{equation*}
    for $4 \leq p \leq n^2+n-2$. By pigeonholing, it suffices to prove 
    \begin{equation}\label{0621.515}
        \big(\alpha^{p}|U_{\alpha}|\big)^{\frac1p} \leq C_{\gamma_n,\e}R^{\e}  R^{\frac1n(\frac12-\frac2p)}  \Big\|(\sum_{\theta \in \Theta^n(R) }|f_{\theta}|^{{\frac{p}{2}}}*w_{\theta^*})^{\frac2p} \Big\|_{L^p(\R^n)}
    \end{equation}
    for all $\alpha>0$. Let us fix $p$ and $\alpha$. Now,
    by Proposition \ref{wpd}, we may replace $f$ by $\sum_{\theta\in \Theta^n(R) }\sum_{T\in\tilde{\T}_\theta}\s_T(x)f_\theta (x)$. Abusing  notation, let us use the same notation $f$ for the new function. Define $\vartheta$ to be the number satisfying
\begin{equation}
    \frac1p=\frac{\vartheta}{4}+\frac{1-\vartheta}{n^2+n-2}.
\end{equation}
    By H\"{o}lder's inequality and the assumption that Proposition \ref{05.03.prop32} is true for $p=4$ and $n^2+n-2$, we have 
\begin{equation}
\begin{split}
    (\alpha^p |U_{\alpha}|)^{\frac1p}\lesssim  \|f\|_4^{\vartheta}\|f\|_{n^2+n-2}^{1-\vartheta}
     \lesssim R^{\frac1n(\frac12-\frac24)\vartheta}R^{\frac1n(\frac12-\frac{2}{n^2+n-2})(1-\vartheta)} X^{\vartheta}Y^{1-\vartheta}
\end{split}
\end{equation}
where
\begin{equation}
    \begin{split}
        &X:= \Big( \sum_{ \substack{ R^{-\frac1n} \leq s \leq 1: \\ dyadic } } \sum_{\tau \in \Theta^n(s^{-n}) } \sum_{U \| U_{\tau,R}  }|U| \|S_{U,2}f\|_{L^{2}(\R^n)}^4 \Big)^{\frac14}
        \\& Y:= \Big(\sum_{ \substack{ R^{-\frac1n} \leq s \leq 1: \\ dyadic } } \sum_{\tau \in \Theta^n(s^{-n}) } \sum_{U \| U_{\tau,R}  }|U| \|S_{U,\frac{n^2+n-2}{2}}f\|_{L^{\frac{n^2+n-2}{2}}(\R^n)}^{n^2+n-2}\Big)^{\frac{1}{n^2+n-2}}
    \end{split}
\end{equation}
Since $R^{\frac1n(\frac12-\frac24)\vartheta}R^{\frac1n(\frac12-\frac{2}{n^2+n-2})(1-\vartheta)} =R^{\frac1n(\frac12-\frac2p)}$, it suffices to check
\begin{equation}
    X^{p\vartheta}Y^{p(1-\vartheta)} \lesssim  \sum_s \sum_{\tau \in \Theta^n(s^{-n}) } \sum_{U \| U_{\tau,R}  }|U| \|S_{U,\frac{p}{2}}f\|_{L^{\frac{p}{2}}(\R^n)}^p.
\end{equation}
Once we have this, \eqref{0621.515} follows by embedding.
As we already saw in the proof of Proposition \ref{05.26.prop82} (see \eqref{0621.514}), we have
\begin{equation}
    \|S_{U,2}f\|_2^2 \sim \|S_{U,\frac{p}{2}}f\|_{\frac{p}{2}}^{\frac{p}{2}} \sim \|S_{U,\frac{n^2+n-2}{2}}f\|_{\frac{n^2+n-2}{2}}^{\frac{n^2+n-2}{2}}.
\end{equation}
Therefore, $X^4 \sim Y^{n^2+n-2}$, and furthermore, 
\begin{equation}
     X^{p\vartheta}Y^{p(1-\vartheta)} \sim \sum_s\sum_{\tau \in \Theta^n(s^{-n}) } \sum_{U \| U_{\tau,R}  } |U| \|S_{U,\frac{p}{2}}f\|_{L^{\frac{p}{2}}(\R^n)}^p.
\end{equation}
This completes the proof.
\end{proof}

\begin{proof}[Proof of Proposition \ref{0720.thm57}]
Use an  argument similar to the one used in \eqref{0621.514}, and  assume that each wave packet has height either one or zero. So we have
\begin{equation}
  \int_{B_R} (\sum_{s\in\I_\de}|f_{s}|^q*w_{p(s)^*})^{2}
  \lesssim  \int_{B_R} (\sum_{s\in\I_\de}|f_{s}|^2*w_{p(s)^*})^{2}.
\end{equation}
Take $g_s:=|f_s|^2 * w_{p(s)^*}$. By Proposition \ref{wee}, this is bounded by
\begin{equation}
\lesssim
     \delta^{-\e} \sum_{\substack{ \de\le \si\le 1: \\ dyadic }  } \sum_{s'\in\I_\si} \sum_{U\parallel (q[\le \si](s'))^*} |U| \Big(\int  \sum_{s\in\I_\de:|s-s'|\le \si} \big( |f_s|^2 * w_{p(s)^*} \big) w_U \Big)^2.
\end{equation}
Since $w_{p(s)^*}* w_U \lesssim w_U$, this is bounded by
\begin{equation}
\lesssim
     \delta^{-\e} \sum_{\substack{ \de\le \si\le 1: \\ dyadic }  } \sum_{s'\in\I_\si} \sum_{U\parallel (q[\le \si](s'))^*} |U| \Big(\int  \sum_{s\in\I_\de:|s-s'|\le \si}  |f_s|^2  w_U \Big)^2.
\end{equation}
As we have shown, since each wave packet has height either one or zero, by \eqref{0621.514}, this is bounded by
\begin{equation}
\lesssim
     \delta^{-\e} \sum_{\substack{ \de\le \si\le 1: \\ dyadic }  } \sum_{s'\in\I_\si} \sum_{U\parallel (q[\le \si](s'))^*} |U| \Big(\int  \sum_{s\in\I_\de:|s-s'|\le \si}  |f_s|^q  w_U \Big)^2.
\end{equation}
This completes the proof.
\end{proof}

\subsection{Pruning and high/low lemma}\label{sec64}

In this subsection, we state several definitions and lemmas for the high/low method. This is the setup in \cite{MR4721026, MR4794594}.
Recall that we have defined $f_{\tau_N}^{N}$ (see \eqref{06.01.57}) and introduced $\tau_k$ (see the beginning of Subsection \ref{sec52}). We next define $f_{\tau_{k}}^{k+1}$ in an inductive way. Define $f_{\tau_{N-1}}^{N}:= \sum_{\tau_{N} \subset \tau_{N-1}} f_{\tau_N}^{N}$. Let us fix $\alpha, \beta>0$. Let $C,K$ be sufficiently large constants, which will be determined later.

\begin{definition}[pruning with respect to $\tau_k$]
    For $1 \leq k \leq N-1$, define
    \begin{equation}\begin{split}
    &\mathbb{T}_{\tau_k}^g:= \Big\{ T_{\tau_k} \in \mathbb{T}_{\tau_k}: \; \|\psi_{T_{\tau_k}}^{\frac12}f_{\tau_k}^{k+1} \|_{\infty} \leq K^3 C^{N-k+1}\frac{\beta}{\alpha} \Big\}
    \\& f_{\tau_k}^k:= \sum_{ T_{\tau_k} \in \mathbb{T}_{\tau_k}^g } \psi_{T_{\tau_k}} f_{\tau_k}^{k+1}, \;\;\;\;\; f_{\tau_{k-1}}^k:= \sum_{\tau_k \subset \tau_{k-1}} f_{\tau_k}^k.
    \end{split}
    \end{equation}
\end{definition}

Recall that by pigeonholing argument (see \eqref{06.01.57}), each wave packet has height either comparable to one or zero. So we have the following lemma, which is Lemma 4.1 of \cite{guth23}.

\begin{lemma}[Lemma 4.1 of \cite{guth23}]\label{0622.lem58}
The following properties hold.
\begin{enumerate}
    \item $|f_{\tau_k}^k(x)| \leq |f_{\tau_{k}}^{k+1}(x)| \lesssim \#\{ \theta \in \Theta^n(R) : \theta\subset \tau_k \}$

    \item $\|f_{\tau_k}^k\|_{\infty} \leq K^3 C^{N-k+1}\frac{\beta}{\alpha}$.

    \item For $R$ sufficiently large depending on $\epsilon$, $\mathrm{supp}(\widehat{f_{\tau_k}^k}) \subset 3\tau_k$.
    
\end{enumerate}
\end{lemma}

Abusing notation, define $f_{\tau_k}:=\sum_{\tau_N\subset \tau_k} f^N_{\tau_N}
$. We have discarded the wave packets whose height is not comparable to one

\begin{lemma}[Locally constant property, Lemma 4.2 of \cite{guth23}]\label{0622.lem59} For each $\tau_k \in \Theta^n(R_k)$ and $T_{\tau_k}\in\T_{\tau_k}$, 
\begin{align*} 
\|f_{\tau_k}\|_{L^\infty(T_{\tau_k})}^2\lesssim |f_{\tau_k}|^2*w_{\tau_k^*}(x)\qquad\text{for any}\quad x\in T_{\tau_k} .\end{align*}
Also, for any $R_k^{1/n}$-ball $B_{R_k^{1/n}}$, 
\begin{align*} 
\Big\|\sum_{\tau_k \in \Theta^n(R_k) }|f_{\tau_k}|^2 \Big\|_{L^\infty(B_{R_k^{1/n}})}\lesssim \sum_{\tau_k \in \Theta^n(R_k)}|f_{\tau_k}|^2*w_{B_{R_k^{1/n}}}(x)\qquad\text{for any}\quad x\in B_{R_k^{1/n}} .\end{align*}
\end{lemma}

\begin{definition} Let $\eta:\R^n\to[0,\infty)$ be a radial, smooth bump function satisfying $\eta(x)=1$ on $B_{1/2}$ and $\supp\eta\subset B_1$. Then for each $s>0$, let 
\[ \eta_{\le s}(\xi) :=\eta(s^{-1}\xi) .\]
We will sometimes abuse notation by denoting $h*\widecheck{\eta}_{>s}:=h-h*\widecheck{\eta}_{\le s}$, where $h$ is some Schwartz function. Also define $\eta_{s}:=\eta_{\le s}-\eta_{\le s/2}$. 
\end{definition}

Let $N_0$ be a large integer such that $N_0 \leq N-1$, which will be determined later. Let us briefly explain why we introduce this number $N_0$. When we apply the high/low method, we make a dichotomy according to whether $k$ is smaller than $N_0$ or not. When $k$ is smaller than $N_0$, it is the ``low case'' and when $k$ is larger than $N_0$, it is the ``high case''. We refer to Subsection \ref{sec61} for this dichotomy.

\begin{definition} For $N_0\le k\le N-1$, let 
\[ g_k:=\sum_{\tau_k \in \Theta^n(R_k) }|f_{\tau_k}^{k+1}|^2*w_{\tau_k^*}, \;\;\; g_k^{\ell}:=g_k*\widecheck{\eta}_{\le R_{k+1}^{-1/n}}, \;\;\;\text{and}\;\;\; g_k^h:=g_k-g_k^{\ell}. \]
\end{definition}

In the following definition, $C \gg 1$ is the same constant that goes into the pruning definition of $f^k$. 
\begin{definition} \label{impsets}Define the high set by 
\[ \Omega_{N-1}:= \{x\in U_{\alpha,\beta}: C \beta \leq g_{N-1}(x)\}. \]
For each $k=N_0,\ldots,N-2$, let
\[ \Omega_k=\{x\in U_{\alpha,\beta}\setminus \cup_{l=k+1}^{N-1}\Omega_{l}: C^{N-k}\beta\le g_k(x) \}. \] 
Define the low set to be
\[ L:=U_{\alpha,\beta}\setminus[\cup_{k=N_0}^{N-1}\Omega_k]. \]
\end{definition}

\begin{lemma}[High-dominance on $\Omega_k$, Corollary 4.5 of \cite{guth23}]\label{highdom} For $R$ large enough depending on $\e$, $g_k(x)\le 2|g_k^h(x)|$ for all $x\in\Omega_k$. 
\end{lemma}

\begin{lemma}[Pruning lemma, Lemma 4.6 of \cite{guth23}]\label{ftofk} Let $0 < s \leq 1$.\footnote{In Lemma 4.6 of \cite{guth23}, they stated the range of $s$: $s>R^{-\e/n}$. But by the same argument, one can prove the same statement without a restriction on the range of $s$.} For any $\tau \in  \Theta^n(s^{-n})$, for any $N_0 \leq k \leq N-1$,
\begin{align*} 
\Big|\sum_{\tau_k\subset\tau}f_{\tau_k}(x)-\sum_{\tau_k\subset\tau}f_{\tau_k}^{k+1}(x) \Big|&\leq \frac{\alpha}{C^{1/2}K^3} \qquad\text{for all $x\in \Omega_k$}\\ \Big|\sum_{\tau_{N_0}\subset\tau}f_{\tau_{N_0} }(x)-\sum_{\tau_{N_0}\subset\tau}f_{\tau_{N_0} }^{{N_0}}(x) \Big|&\leq \frac{\alpha}{C^{1/2}K^3}\qquad \text{ for all $x\in L$}. 
\end{align*}
\end{lemma}
This finishes the high/low setup. 
\medskip

\section{Proof of Proposition \ref{0621.prop51}}\label{0703.sec6}

Recall from Sections \ref{section4} and \ref{0627.sec6} that Theorem
\ref{0501.thm12} has been reduced to Proposition \ref{0621.prop51}, and
that the corresponding cone estimate, Theorem \ref{05.03.thm13}, has
been reduced to Proposition \ref{0621.prop52}. In this section, we prove
Proposition \ref{0621.prop51} by induction on the dimension, assuming the
auxiliary cone estimate stated below as Proposition \ref{1020.prop313}.
This cone estimate is proved in Section \ref{sec7} and yields Proposition
\ref{0621.prop52}. Thus, in view of the reductions in Sections
\ref{section4} and \ref{0627.sec6}, Sections \ref{0703.sec6} and
\ref{sec7} complete the proofs of Theorems \ref{0501.thm12} and
\ref{05.03.thm13}.

  To prove Proposition \ref{0621.prop51},
we use an induction on the dimension $n$. The base case is $n=2$, and this case is already proved by \cite{MR320624} (see also \cite{MR688026}). Throughout the section, let us fix $A$ and $\e$ (see the definition of $\mathcal{C}_n(A,R,\e)$ in Subsection \ref{51red}), and assume the following induction hypothesis.

\subsection*{Induction hypothesis} Assume that Proposition \ref{0621.prop51} holds in all dimensions up to \(n-1\).

\begin{definition}\label{1003.def61}
Let $S_p(1,R)$  be the smallest constant such that 
 \begin{equation}\label{0621.61}
        \big( \alpha^p |U_{\alpha}| \big)^{\frac1p} \leq S_{p}(1,R) \Big\|(\sum_{\tau_N \in \Theta^n(R) }|f_{\tau_N}^N|^{\frac{p}{2} }*w_{\tau_N^*})^{\frac2p} \Big\|_{L^p(\R^n)}
    \end{equation}
    for all curves $\gamma_n \in \mathcal{C}_n(A,R,\e)$, all $\alpha>0$, and all  Schwartz functions $f$ whose Fourier supports are in $\mathcal{M}^n(\gamma_n;R)$.
\end{definition}

Here, recall that $\{\tau_N\}=\{\theta\}$ and \eqref{06.01.57} for the definition of $f_{\tau_N}^N$.

Let us give one remark on the definition. Since $|f_{\tau_N}^{N}| \leq |f_{\theta}|$, one may see that
\begin{equation}
    \|f\|_{L^p} \leq C_{\e}R^{\e} \big(1+ S_p(1,R) \big)
    \Big\|(\sum_{\theta \in \Theta^n(R) }|f_{\theta}|^{\frac{p}{2} }*w_{\theta^*})^{\frac2p} \Big\|_{L^p(\R^n)}.
\end{equation}
This inequality is more directly related to Proposition \ref{0621.prop51}. However,
by some technical reasons, we define the constant $S_p(1,R)$ by using the inequality \eqref{0621.61}.

$S_p(1,R)$ is the constant we lose when we ``decouple" the function from scale $1$ to scale $R$.
Recall that we have introduced intermediate scales $R_k$ (see the beginning of Subsection \ref{sec52}). Our proof relies on a multiscale analysis. In the following, we also define the constant from the scale $R_k$ to the last scale $R$. 

\begin{definition}\label{0621.def62}
For $1  \leq k \leq N-1$, let $S_p(R_k,R)$  be the smallest constant such that 
\begin{equation}\label{0621.63}
        \Big\|(\sum_{\tau_k \in \Theta^n(R_k) }|f_{\tau_k}^{k+1}|^{\frac{p}{2}} *w_{\tau_k^*} )^{\frac2p} 
 \Big\|_{L^p(\R^n)} \leq S_{p}(R_k,R) \Big\|(\sum_{\tau_N \in \Theta^n(R) }|f_{\tau_N}^{N}|^{\frac{p}{2}} *w_{\tau_N^*} )^{\frac2p} \Big\|_{L^p(\R^n)}
    \end{equation}
for all curves $\gamma_n \in \mathcal{C}_n(A,R,\e)$, and all Schwartz functions $f$ whose Fourier supports are in $\mathcal{M}^n(\gamma_n;R)$.
\end{definition}

By definition, Proposition \ref{0621.prop51} follows from
\begin{equation}
    S_{p}(1,R) \leq C_{\e}R^{200^n\e^{\frac12}}R^{\frac1n(\frac12-\frac2p)}
\end{equation}
for $4 \leq p \leq n^2+n-2$. To prove the estimate, we use a broad-narrow analysis originated from \cite{MR2860188}, where they applied the broad-narrow analysis to a hypersurface. Later, \cite{MR3161099} applied the broad-narrow analysis to a nondegenerate curve in $\mathbb{R}^n$. 
\medskip

We say $\tau^1,\ldots,\tau^{n} \in \Theta^n(K)$ are $K$-transverse if $d(\tau^{j_1},\tau^{j_2}) > K^{-\frac1n}$ for any $j_1 \neq j_2$. Recall \eqref{0706.515}.
For a $K$-transverse tuple $\{\tau^j\}_{j=1}^{n} \subset \Theta^n(K)$, define the broad set
\begin{equation}\label{0702.75}
    \BR_{\alpha,\beta}^{K}:= \big\{ x \in U_{\alpha,\beta}: \alpha \leq K \avprod_{j=1}^n |f_{\tau^j}(x)|, \;\; \max_{\tau^j}|f_{\tau^j}(x)| \leq \alpha \big\}.
\end{equation}
Let $S_{\BR,p,K}(1,R)$ be the smallest constant such that
\begin{equation}
    \alpha^p |\BR_{\alpha,\beta}^K| \leq S_{\BR,p,K}(1,R)^p \Big\|(\sum_{\tau_N \in \Theta^n(R) }|f_{\tau_N}^N|^{\frac{p}{2}}*w_{\tau_N^*})^{\frac2p} \Big\|_{L^p(\R^n)}^p
\end{equation}
for all $\alpha,\beta$, any choice of a $K$-transverse tuple, any functions $f$ whose Fourier supports are in $\mathcal{M}^n(\gamma_n;R)$ and any curves $\gamma_n \in \mathcal{C}_n(A,R,\e)$.

By a standard broad-narrow analysis, one can show that bounding $S_p(1,R)$ is reduced to bounding $S_{\mathrm{Br},p,K}(1,R)$.

\begin{proposition}\label{05.25.prop84} For any $\e>0$, there exists $K>0$ depending on $\e$ so that 
\begin{equation}
    S_p(1,R) \leq C_{\e}R^{\e} \sup_{1 \leq X \leq R} S_{\BR,p,K}(1,X)+ C_{\e}R^{\frac1n(\frac12-\frac2p)+ \e}
\end{equation}
for all $R \geq 1$.
\end{proposition}

To bound $S_{\mathrm{Br},p,K}(1,R)$, we use a multiscale argument. Here is the key estimate. For technical reasons (see \eqref{0621.624} and the discussion following it), we need to distinguish $n=3$ from $n \geq 4$. 

\begin{proposition}\label{05.01.thm61} Assume that Proposition \ref{0621.prop51} is true for $n-1$ for any $A,\epsilon>0$.

Let $n=3$. For $8 \leq p \leq 10$, we have
\begin{equation}
\begin{split}
    S_{\BR,p,K}(1,R) &\lesssim R^{10^n N_0\e}R^{\frac13(\frac12-\frac2p)}
    \\&
    + R^{1000 \e}R^{\frac13(\frac12-\frac2p)}  \sum_{j=N_0}^{N} \Big(\frac{(R_j)^{\frac13(\frac12-\frac28) }}{R^{\frac13(\frac12-\frac28)}} \Big)^{\frac{8}{p}} S_8(R_j, R )^{\frac{8}{p}}.
\end{split}
\end{equation}
Let $n \geq 4$. For $ p = n^2+n-2$, we have
    \begin{equation}
        S_{\BR,p,K}(1,R) \lesssim R^{10^n N_0\e}R^{\frac1n(\frac12-\frac2p)}+ R^{100^n\e}\sum_{j=N_0}^{N} (R_j)^{\frac1n(\frac12-\frac2p) }S_p(R_j, R).
    \end{equation}
    For $p=4$, we have
    \begin{equation}
        S_p(1,R) \lesssim R^{10^n N_0 \e}+ R^{100^n\e}.
    \end{equation}
\end{proposition}

Here \(N_0\) is the cutoff index introduced in Section \ref{sec64}; it will be chosen as
\[
N_0=\lfloor \epsilon^{-1/2}\rfloor
\]
at the end of the bootstrap argument.
By the above proposition, bounding $S_{\mathrm{Br},p,K}(1,R)$ is reduced to bounding $S_p(R_j,R)$. We next relate $S_p(R_j,R)$ to $S_p(1,X)$. 

\begin{proposition}\label{05.24.prop53} Let \(4\le p\le n^2+n-2\) and \(N_0\le k\le N-1\). Then
    \begin{equation}
        S_p(R_k,R) \lesssim R^{\e} \Big( \sup_{1 \leq X \leq R^{\e}}S_p(1,X) \big(\frac{R^{\e}}{X} \big)^{\frac1n(\frac12-\frac2p)}  \Big)^{N-k}.
    \end{equation}
\end{proposition}

The left hand side of \eqref{0621.63} contains the $L^p$ norm of the $l^{p/2}$ expression, whereas the left hand side of \eqref{0621.61} contains the $L^p$ norm of $f$. To relate these two, we use Proposition \ref{0616.thm33}. After applying the proposition, in the high case, we use rescaling together with the definition of $S_p(1,R)$. In the low case, we use the decoupling theorem for a nondegenerate curve in $\mathbb{R}^{n-1}$ by \cite{MR3548534}. This gives the proof of Proposition \ref{05.24.prop53}.

Let us state the proposition for a cone. Using a bootstrapping argument originating in \cite{MR2288738}, one can prove the following proposition.

\begin{proposition}\label{1020.prop313}
Assume that Proposition \ref{0621.prop51} is true for $n-1$ for any $A,\epsilon>0$.
    Then for $4 \leq p \leq n^2-n-2$ and $\e'>0$,
       \begin{equation*}
        \|g\|_{L^p(\R^n)} \leq C_{p,\e'} R^{\frac{1}{n-1}(\frac{1}{2}-\frac2p)+\e' }  \Big\|(\sum_{\theta \in \Xi_n(R) }|g_{\theta}|^{\frac{p}{2}} * w_{\theta^*} )^{\frac2p}\Big\|_{L^p(\R^n)}
    \end{equation*}
       for all curves $\gamma_n \in \mathcal{C}_n(A,R,\e)$, and all Schwartz functions $g$ whose Fourier supports are in $\Gamma_n(\gamma_n;R)$.
\end{proposition}
For the rest of this section, we prove Propositions \ref{05.25.prop84},
\ref{05.01.thm61}, and \ref{05.24.prop53}. Assuming also Proposition
\ref{1020.prop313}, we then deduce Proposition \ref{0621.prop51}.
Proposition \ref{1020.prop313} will be proved in Section \ref{sec7}.

\subsection{Proof of Proposition 
\ref{05.01.thm61}}\label{sec61}

Let us begin with the multilinear restriction estimate. The proof of the estimate is well-known, for example, see the proof of Theorem 5.2 of \cite{MR4794594} or Lemma 2.5 of \cite{MR3161099}.

\begin{proposition}[Multilinear restriction estimate]
    Let $\{\tau^j\}_{j=1}^{n} \subset \Theta^n(K)$ be a transverse tuple. Then for any $M \geq K$, we have
    \begin{equation}
        \int_{B_M} \avprod_{j=1}^{n} |f_{\tau^j}(x)|^{2n}\,dx \leq C_{A,K} \avprod_{j=1}^{n} \Big( \int_{\R^n} |f_{\tau^j}(x)|^2 W_{B_M}(x)\,dx \Big)^{n}
    \end{equation}
     for all Schwartz functions $f$ whose Fourier supports are contained in the $M^{-1}$-neighborhood of a nondegenerate curve $\gamma_n \in \mathcal{C}_n(A,R,\e)$.
\end{proposition}

We write
\begin{equation}\label{0621.612}
    \alpha^p |\mathrm{Br}_{\alpha,\beta}^K| \lesssim   \alpha^p |\mathrm{Br}_{\alpha,\beta}^K \cap L|+\sum_{k=N_0}^{N-1} \alpha^p| \mathrm{Br}_{\alpha,\beta}^K \cap \Omega_k|.
\end{equation}
 Let us first consider the term $\alpha^p |\mathrm{Br}_{\alpha,\beta}^K \cap L|$. By Lemma \ref{ftofk},  the definition of $L$, and H\"{o}lder's inequality, we have 
\begin{equation}
\begin{split}
    \alpha^p|\BR_{\alpha,\beta}^K \cap L| \lesssim \int_{L}| \sum_{\tau_{N_0}} f_{\tau_{N_0}}^{N_0}(x)|^p &\lesssim R^{10^n N_0\e} \int_{L} \Big(\sum_{\tau_N}|f_{\tau_N}^N|^2 * w_{\tau_N^*} \Big)^{\frac{p}{2}}
    \\& \lesssim R^{10^n N_0\e }R^{\frac1n(\frac12-\frac2p)}  \int_{\R^n} \Big(\sum_{\tau_N}|f_{\tau_N}^N|^{\frac{p}{2}}*w_{\tau_N^*} \Big)^2.
\end{split}
\end{equation}
Hence, it suffices to consider the second term. Let $N_0 \leq k \leq N-1$. For simplicity, let us use the notation $q:=p/2$. We cover $\mathbb{R}^n$ by finitely overlapping balls $B_{R_k^{1/n}}$, and obtain
\begin{equation}
    |\BR_{\alpha,\beta}^K \cap \Omega_k|\lesssim \sum_{B_{R_k^{1/n}}} |\BR_{\alpha,\beta}^K \cap \Omega_k \cap B_{R_k^{1/n}}|. 
\end{equation}
Fix $B_{R_k^{1/n}}$.
By Lemma \ref{ftofk}, the multilinear restriction estimate and H\"{o}lder's inequality, we have 
\begin{equation}\label{25.03.11.28}
\begin{split}
    \alpha^{2n}|\BR_{\alpha,\beta}^K \cap \Omega_k \cap B_{R_k^{1/n}}|\lesssim 
    \big(\int_{\R^n}(\sum_{\tau_k}|f^{k+1}_{\tau_k}(y)|^2)W_{B_{R_k}^{1/n}(y)}\mathrm{d}y\big)^n  \lesssim |B_{R_k^{1/n}}||g_{k}(x)|^n
\end{split}
\end{equation}
for any $x \in \BR_{\alpha,\beta}^K \cap \Omega_k \cap B_{R_k^{1/n}}$. Take $(\tilde{p},\tilde{q}):=(n^2-n-2,\frac{n^2-n-2}{2})$ so that 
\begin{equation}
    \frac{ \tilde{p}}{ \tilde{q} } = \frac{p}{q} =2. 
\end{equation}
By the inequalities $|g_{k}(x)| \gtrsim \beta$ on $x \in \Omega_k$ and the high-dominance on $\Omega_k$ (Lemma \ref{highdom}) and the locally constant property (Lemma \ref{0622.lem59}), the RHS of \eqref{25.03.11.28} is bounded by  
\begin{equation}
\lesssim
    \beta^{-\tilde{p}+n} |B_{R_k^{1/n}}| |g_k^h(x)|^{\tilde{p}} \lesssim \beta^{-\tilde{p}+n} \int_{\R^n}|g_k^h(y)|^{\tilde{p}} W_{B_{R_k^{1/n}}}(y)\,dy.
\end{equation}
By summing over $B_{R_k^{1/n}}$, we obtain
\begin{equation}
    \alpha^{2n}|\BR_{\alpha,\beta}^K \cap \Omega_k| \lesssim \beta^{-\tilde{p}+n} \int_{\R^n}|g_k^h(y)|^{\tilde{p}}\,dy.
\end{equation}
We next analyze the Fourier support of $g_k^h$. Write   
\begin{equation}
\widehat{g_k^h}= (1- \eta_{ \leq R_{k+1}^{-1/n}}) \sum_{\tau_k} \Big(\widehat{f_{\tau_k}^{k+1}} *\widehat{\overline{f_{\tau_k}^{k+1}}}  \Big) \widehat{w_{\tau_k^*}} =: \sum_{\tau_k} \widehat{G_{\tau_k}}.    
\end{equation}
 Write
\begin{equation}
    \widehat{g_k^h} = \sum_{ R^{-\e/n} \lesssim s' \lesssim 1:\, \mathrm{dyadic} } \sum_{\tau_k}\widehat{G_{\tau_k}}  \eta_{s' R_k^{-1/n}}.
\end{equation}
Since there are $O(\log R_k)$ many $s'$', we have
\begin{equation}\label{06.01.617}
    \alpha^{2n}|\BR_{\alpha,\beta}^K \cap \Omega_k| \lesssim \beta^{-\tilde{p}+n} R_k^{O(\e)} \int_{\R^n}|\sum_{\tau_k} G_{\tau_k} * (\eta_{s' R_k^{-1/n}})^{\vee}|^{\tilde{p}}
\end{equation}
for some $s'$.
The Fourier support of $G_{\tau_k} * (\eta_{s' R_k^{-1/n}})^{\vee}$ is contained in  
\begin{equation*}
    \Big\{  \sum_{j=1}^{n} \lambda_j \gamma_n^{(j)}(s): \; s \in [s_{\tau_k}, s_{\tau_k}+R_k^{-\frac1n}], \; |\lambda_i| \leq 2R_k^{-\frac{i}{n}}\; \mathrm{for \; all}\; 2 \leq i \leq n, \; |\lambda_1| \sim s'R_k^{-\frac1n}   \Big\}.
\end{equation*}
Let us do an isotropic rescaling so that the length of $\lambda_1$ becomes one.
Define $H_{\tau_k}$ to be a smooth function so that $\widehat{H_{\tau_k}}(\xi)= \widehat{G}_{\tau_k}(2s'R_k^{-\frac1n} \xi) \eta_{s' R_k^{-1/n}}(2s'R_k^{-\frac1n}\xi)$.
Then the Fourier support of $H_{\tau_k}$ is contained in
\begin{equation*}
    \Big\{  \sum_{j=1}^{n} \lambda_j \gamma_n^{(j)}(s): s \in [s_{k},s_{k}+ R_k^{-\frac1n} ], \;\; |\lambda_i| \leq (s')^{-1}R_k^{-\frac{i-1}{n}}\; \mathrm{for \; all}\; 2 \leq i \leq n, \; |\lambda_1| \sim 1   \Big\}.
\end{equation*}
Since we have assumed Proposition \ref{0621.prop51} holds for $n-1$ as induction hypothesis, by
applying  Proposition \ref{1020.prop313}, noting that $(s')^{-1}=O(R^{\e})$, and applying H\"{o}lder's inequality, we obtain  
\begin{equation}
    \Big\| \sum_{\tau_k} H_{\tau_k} \Big\|_{L^{\tilde{p}}}^{\tilde{p}} \lesssim R^{9^n \e} R_{k}^{\frac{\tilde{p}}{n}(\frac12-\frac{2}{\tilde{p}})} \Big\| \big( \sum_{\tau_k}|H_{\tau_k}|^{\frac{\tilde{p}}{2}} * w_{(\tau_k')^*} \big)^{\frac{2}{\tilde{p}}} \Big\|_{L^{\tilde{p}}}^{\tilde{p}}
\end{equation} 
where $\tau_k'$ is defined by $(2s')^{-1}R_k^{\frac1n}$-dilation of the Fourier support of  $G_{\tau_k} * (\eta_{s' R_k^{-1/n}})^{\vee}$.
After rescaling back, the right hand side of \eqref{06.01.617} is bounded by
\begin{equation}\label{0201.19}
\begin{split}
&\lesssim R^{10^n \e} \big(R_k^{\frac1n(\frac{1}{2}-\frac{2}{\tilde{p}}) }\big)^{\tilde{p}} 
    \beta^{-\tilde{p}+n} \int \Big( \sum_{\tau_k}|G_{\tau_k} * (\eta_{s' R_k^{-1/n}})^{\vee}|^{\frac{\tilde{p}}{2} } *w_{\tau_k^*} \Big)^{2}
\\&
\lesssim R^{10^n \e} \big(R_k^{\frac1n(\frac{1}{2}-\frac{2}{\tilde{p}}) }\big)^{\tilde{p}} 
    \beta^{-\tilde{p}+n} \int \Big( \sum_{\tau_k}|G_{\tau_k}|^{ \frac{\tilde{p}}{2} } *w_{\tau_k^*} \Big)^{2 }
    \\&
\lesssim R^{10^n \e} \big(R_k^{\frac1n(\frac{1}{2}-\frac{2}{\tilde{p}}) }\big)^{\tilde{p}} 
    \beta^{-\tilde{p}+n} \int \Big( \sum_{\tau_k}|f_{\tau_k}^{k+1}|^{\tilde{p}}*w_{\tau_k^*} 
  \Big)^{2 }.
\end{split}
\end{equation}
In summary, we have proved
\begin{equation}\label{05.22.515}
     \alpha^{2n}|\BR_{\alpha,\beta}^K \cap \Omega_k|  \lesssim R^{10^n \e} \big(R_k^{\frac1n(\frac{\tilde{p}}{2}-2) }\big) 
    \beta^{-\tilde{p}+n} \int \Big( \sum_{\tau_k}|f_{\tau_k}^{k+1}|^{{\tilde{p}}}*w_{\tau_k^*} \Big)^{2 }.
\end{equation}
By routine calculations, we have
\begin{equation}\label{0621.624}
   \frac{p}{2}=q \leq \tilde{p}=n^2-n-2 \;\;\; \mathrm{if \; and \; only \; if \; } \;\;  p \leq 2n^2-2n-4.
\end{equation}
Recall that we are considering the range $4 \leq p \leq n^2+n-2$. So we have $q \leq \tilde{p}$ for $n \geq 4$. On the other hand, it may happen that $q \geq \tilde{p}$ for $n=3$. So we distinguish two ranges of $n$.

\subsubsection{The case  $n=3$}
Let us state $\eqref{05.22.515}$. Recall  that $(\tilde{p},\tilde{q})=(4,2)$.
\begin{equation}
      \alpha^{6}|\BR_{\alpha,\beta}^K \cap \Omega_k| \lesssim  R^{10^n \e} 
    \beta^{-1} \int \Big( \sum_{\tau_k}|f_{\tau_k}^{k+1}|^{4} * w_{\tau_k^*} \Big)^{2 }.
\end{equation}
To prove the proposition,
 we need to show that
\begin{equation}\label{0622.630}
\begin{split}
    \alpha^{p-6} 
     \int \Big( \sum_{\tau_k}|f_{\tau_k}^{k+1}|^{4}*w_{\tau_k^*} \Big)^{2} 
    \lesssim  \beta R^{\frac{p}{3}(\frac12-\frac2p)+\e} \Big(\frac{R_k}{R}\Big)^{{\frac83(\frac12-\frac28)}} S_8(R_k,R)^{8} \int (\sum_{\tau_N}|f_{\tau_N}^N|^q *w_{\tau_N^*} )^{2}.
\end{split}
\end{equation}
By Lemma \ref{ftofk} and the definition of $\Omega_k$, for $x \in \Omega_k$, we have
\begin{equation}\label{0622.6301}
    \alpha^2 \lesssim |\sum_{\tau_k} f_{\tau_k}^{k+1}(x)|^2 \lesssim R_{k}^{\frac13} \sum_{\tau_k}|f_{\tau_k}^{k+1}(x)|^2 \lesssim R_k^{\frac13} g_k(x) \lesssim  R_k^{\frac13} R^{\e} g_{k+1}(x) \lesssim R_k^{\frac13}R^{\e} \beta.
\end{equation}
By the inequality, \eqref{0622.630} follows from
\begin{equation}\label{0622.6310}
\begin{split}
    \alpha^{p-8} 
    R_k^{\frac13} \int \Big( \sum_{\tau_k}|f_{\tau_k}^{k+1}|^{4}*w_{\tau_k^*} \Big)^{2} 
    \lesssim R^{\frac{p}{3}(\frac12-\frac2p)} \Big(\frac{R_k}{R}\Big)^{\frac23} S_8(R_k,R)^{8} \int (\sum_{\tau_N}|f_{\tau_N}^{N}|^q *w_{\tau_N^*} )^{2}.
\end{split}
\end{equation}
By  \eqref{0622.59} and \eqref{0622.6301}, we have $\alpha^2 \lesssim R_{k}^{\frac13}R^{\e} R^{\frac13} \sup_{\tau_N}\|f_{\tau_N}^N\|^2_{\infty}$. So, \eqref{0622.6310} follows from 
\begin{equation*}
\sup_{\tau_N}\|f_{\tau_N}^N\|_{\infty}^{p-8}
    \int \Big( \sum_{\tau_k}|f_{\tau_k}^{k+1}|^{4}*w_{\tau_k^*} \Big)^{2} \lesssim  (R_k)^{-\frac{p}{6}+\frac53}  S_8(R_k,R)^{8}
    \int (\sum_{\tau_N}|f_{\tau_N}^N
    |^q *w_{\tau_N^*} )^{2}.
\end{equation*}
Note that the power of $R_k$ is non-negative (since $8\le p\le 10$), we just need to prove
\begin{equation*}
\sup_{\tau_N}\|f_{\tau_N}^N\|_{\infty}^{p-8}
    \int \Big( \sum_{\tau_k}|f_{\tau_k}^{k+1}|^{4}*w_{\tau_k^*} \Big)^{2} \lesssim   S_8(R_k,R)^{8}
    \int (\sum_{\tau_N}|f_{\tau_N}^N
    |^q *w_{\tau_N^*} )^{2}.
\end{equation*}
By Definition \ref{0621.def62}, we have
\begin{equation}
    \int \Big( \sum_{\tau_k}|f_{\tau_k}^{k+1}|^{4}*w_{\tau_k^*} \Big)^{2}  \leq S_8(R_k,R)^{8}\int (\sum_{\tau_N}|f_{\tau_N}^N|^{4} *w_{\tau_N^*} )^{2}.
\end{equation}
Now, the interpolation lemma (Proposition \ref{05.26.prop82})  gives the desired result. This completes the proof. \\ 

\subsubsection{The case  $n \geq 4$}

Our goal is to prove 
\begin{equation} 
    \alpha^{p}|\BR_{\alpha,\beta}^K \cap \Omega_k|  \lesssim R^{100^n \e} (R_k)^{\frac1n(\frac12-\frac2p)p} S_p(R_k,R)^p \int (\sum_{\tau_N}|f_{\tau_N}^N|^{\frac{p}{2}} *w_{\tau_N^*} )^{2}. 
\end{equation}
Let us continue from \eqref{05.22.515}. 
Since $q \leq \tilde{p}$, we have 
\begin{equation*}
\begin{split}
      \alpha^{2n}|\BR_{\alpha,\beta}^K \cap \Omega_k| \lesssim  R^{10^n \e}R_k^{\frac1n(\frac{\tilde{p} }{2}-2)}  
    \beta^{-\tilde{p}+n}  \sup_{\tau_k} \|f_{\tau_k}^{k+1}\|_{\infty}^{(\tilde{p}-q)2 }  \int \Big( \sum_{\tau_k}|f_{\tau_k}^{k+1}|^{q} *w_{\tau_k^*}\Big)^{2 }.
\end{split}
\end{equation*}
Hence, it suffices to prove 
\begin{equation}
\begin{split}
    &\alpha^{p-2n} R_k^{\frac1n(\frac{\tilde{p} }{2}-2) } 
    \beta^{-\tilde{p}+n}  \sup_{\tau_k} \|f_{\tau_k}^{k+1}\|_{\infty}^{(\tilde{p}-q)2 }  \int \Big( \sum_{\tau_k}|f_{\tau_k}^{k+1}|^{q} *w_{\tau_k^*}\Big)^{2} \\&
    \lesssim R^{10^n \e} (R_k)^{\frac1n(\frac12-\frac2p)p} S_p(R_k,R)^p \int (\sum_{\tau_N}|f_{\tau_N}^N|^{\frac{p}{2}} *w_{\tau_N^*} )^{2}.
    \end{split}
\end{equation}
By Definition \ref{0621.def62} and $p/q=2$,
it suffices to verify 
\begin{equation}
    \alpha^{p-2n} R_k^{\frac1n({\frac{ \tilde{p}}{2} }-2) } 
    \beta^{-\tilde{p}+n}  \sup_{\tau_k} \|f_{\tau_k}^{k+1}\|_{\infty}^{2\tilde{p}-p } \lesssim R^{10^n \e}R_k^{\frac1n(\frac{p}{2}-2)}.
\end{equation}
By the inequality $\|f_{\tau_k}^{k+1}\|_{\infty} \lesssim \beta \alpha^{-1}$ (see Lemma \ref{0622.lem58}), it becomes
\begin{equation}
    \alpha^{p-2n} R_k^{\frac1n({\frac{ \tilde{p}}{2} }-2) } 
    \beta^{-\tilde{p}+n}  (\beta \alpha^{-1})^{2\tilde{p}-p } \lesssim  R^{10^n \e} R_k^{\frac1n(\frac{p}{2}-2)}.
\end{equation}
After calculations, this is equivalent to
\begin{equation}
    \alpha^{2p-2n-2\tilde{p}} \beta^{\tilde{p}+n-p} \lesssim R^{10^n \e} (R_k)^{\frac1n(\frac{p-\tilde{p}}{2})}.
\end{equation}
By following the same argument as in the proof of \eqref{0622.6301},
we have $\alpha \lesssim R^{\e} R_k^{\frac{1}{2n}} \beta^{\frac12}$. Using this bound, it remains to check that $p-\tilde{p} \leq 2n$. This is true because $4 \leq p \leq n^2+n-2$.
This completes the proof.

\subsection{Proof of Proposition \ref{05.24.prop53}}
 To simplify the notation, we introduce 
\begin{equation}
    S_p^*(1,R):=\sup_{1 \leq X \leq R}S_p(1,X) \big(\frac{R}{X} \big)^{\frac1n(\frac12-\frac2p)}.
\end{equation}
To prove the proposition, we need to show 
\begin{equation}
\begin{split}
     \Big\|(\sum_{\tau_k \in \Theta^n(R_k) }|f_{\tau_k}^{k+1}|^\frac{p}{2}*w_{\tau_k^*} )^{\frac2p} \Big\|_{L^p(\R^n)} 
     \lesssim \big( R^{2\e^3} S_p^*(1,R^{\e}) \big)^{N-k} \Big\|(\sum_{\tau_N \in \Theta^n(R) }|f_{\tau_N}^N|^{\frac{p}{2}} *w_{\tau_N^*} )^{\frac2p} \Big\|_{L^p(\R^n)}.
    \end{split}
\end{equation}
Recall that $R/R_k=R^{\e(N-k)}$. To prove the above inequality, it suffices to prove the inequality below for all $k \leq j \leq N-1$. 
\begin{equation}\label{0702.739}
\begin{split}
     \Big\|(\sum_{\tau_j \in \Theta^n(R_{j}) }|f_{\tau_j}^{j+1}|^{\frac{p}{2}} *w_{\tau_j^*} )^{\frac2p} \Big\|_{L^p(\R^n)} 
     \lesssim   \big( R^{2\e^3} S_p^*(1,{R^{\e}} )  \big) \Big\|(\sum_{\tau_{j+1} \in \Theta^n(R_{j+1}) }|f_{\tau_{j+1}}^{j+2}|^\frac{p}{2} *w_{\tau_{j+1}^*} )^{\frac2p} \Big\|_{L^p(\R^n)}.
     \end{split}
\end{equation}
In the case $j=N-1$, we interpret $f^{N+1}_{\tau_N}$ as $f^N_{\tau_N}$.
Let us fix $j$. Recall that $|f_{\tau_{j+1}}^{j+1}| \lesssim |f_{\tau_{j+1}}^{j+2}|$. Since the proof uses an induction on $r$, we prove the following  general inequality; for every $1 \leq r \leq R^{\e}$,
\begin{equation}\label{0622.643}
\begin{split}
     \Big\|(\sum_{\tau_j \in \Theta^n(rR_{j}) }|f_{\tau_j}^{j+1}|^{\frac{p}{2}} *w_{\tau_j^*} )^{\frac2p} \Big\|_{L^p(\R^n)} 
     \lesssim    \big( \frac{R^{\e}}{r} \big)^{\e^2} R^{2\e^3} S_p^*(1,\frac{R^{\e}}{r} )   \Big\|(\sum_{\tau_{j+1} \in \Theta^n(R_{j+1}) }|f_{\tau_{j+1}}^{j+1}|^\frac{p}{2} *w_{\tau_{j+1}^*} )^{\frac2p} \Big\|_{L^p(\R^n)}.
\end{split}
\end{equation}
 Taking $r=1$ in \eqref{0622.643}, and then using $|f_{\tau_{j+1}}^{j+1}| \lesssim |f_{\tau_{j+1}}^{j+2}|$, gives \eqref{0702.739}. We prove \eqref{0622.643} by descending induction on $r$. The factor $(R^{\e}/r)$ is artificially added to close the backward induction. \medskip

The estimate is trivial for $r \in [R^{\e-\e^3}, R^{\e}]$.  Suppose that \eqref{0622.643} is true for $r \in [R^{\e-(X-1) \e^5}, R^{\e}]$ (for some $X\in\mathbb N$). We will show the estimate for $ r \in [ R^{\e-X\e^5},R^{\e}]$. Apply Proposition \ref{0616.thm33} to the left hand side of \eqref{0622.643}. Then it is bounded by
\begin{equation}\label{0623.644}
    \sum_{(rR_j)^{-1} \leq \sigma \leq 1 }\sum_{ \tau \in \Theta^n(\sigma^{-1})} \sum_{U\parallel U_{\tau}}
    |U| \Big( \int_{\R^{n}} \sum_{\tau_j: \tau_j \subset \tau } |f_{\tau_j}^{j+1}(x)|^{\frac{p}{2}}   w_U(x)\,dx \Big)^2
\end{equation}
Here $U_{\tau}$ is a union of the dual boxes of $\tau_j$, centered at the origin, for all $\tau_j \subset \tau$. 
We write \eqref{0623.644} as follows. 
\begin{equation}\label{0623.6450}
\begin{split}
    &\sum_{(rR_j)^{-1} \leq \sigma \leq R^{(n-1)\e^5}(r R_j)^{-1} }\sum_{ \tau \in \Theta^n(\sigma^{-1})} \sum_{U\parallel U_{\tau}}
    |U| \Big( \int_{\R^{n}} \sum_{\tau_j: \tau_j \subset \tau } |f_{\tau_j}^{j+1}|^{\frac{p}{2}}   w_U \Big)^2
    \\&
    +
    \sum_{R^{(n-1)\e^5} (rR_j)^{-1} \leq \sigma \le 1 }\sum_{ \tau \in \Theta^n(\sigma^{-1})} \sum_{U\parallel U_{\tau}}
    |U| \Big( \int_{\R^{n}} \sum_{\tau_j: \tau_j \subset \tau } |f_{\tau_j}^{j+1}|^{\frac{p}{2}}   w_U \Big)^2.
\end{split}
\end{equation}
We need to bound these two terms.
\\

Let us first consider the first term. By H\"{o}lder's inequality, it is bounded by 
\begin{equation}
    R^{C\e^5} \sum_{\tau_j} \int |f_{\tau_j}^{j+1}|^p.
\end{equation}
Fix $\tau_j$. We claim that
\begin{equation}\label{0623.6480}
\begin{split}
    \int |f_{\tau_j}^{j+1}|^p \lesssim  S_p^* \big(1,\frac{R^{\e}}{r} \big)^p \int   \big( \sum_{\tau_{j+1} \subset \tau_j } |f_{\tau_{j+1}}^{j+1}|^{\frac{p}{2}}* w_{\tau_{j+1}^*} \big)^2.
\end{split}
\end{equation}
To prove the claim, we rescale and apply Definition \ref{1003.def61}. Since we are going to use the same rescaling argument several times, let us give the proof in detail. By \eqref{caps1} and \eqref{0623.16}, we can write $\tau_j$ as follows.
\begin{equation}\label{0623.648}
     \Big\{ \gamma_n(s) + \sum_{j=1}^n \lambda_j \gamma_n^{(j)}(s): s \in [s_{\tau_j},s_{\tau_j}+(r R_j)^{-\frac1n} ], \;\;\; |\lambda_j| \leq (rR_j)^{-\frac{j}{n}}  \Big\}.
\end{equation}
Recall that $\gamma_n$ belongs to $\mathcal{C}_n(A,R,\e)$ (see \eqref{0623.51def}).  Take a linear transformation $L$ and $\vec{v} \in \mathbb{R}^{n}$ such that $L \big( \gamma_n(s+s_{\tau_j}) -\vec{v} \big)$ belongs to $\mathcal{C}_n(10nA,R,\e)$, the norm of $\vec{v}$ is smaller than ten, and the determinant of $L$ is comparable to one. By dividing the domain of the curve into smaller intervals, and abusing the notation, we may assume that $L \big( \gamma_n(s+s_{\tau_j}) -\vec{v} \big)$ belongs to $\mathcal{C}_n(A,R,\e)$. Define the curve $\widetilde{\gamma}_n(s):=L \big( \gamma_n(s+s_{\tau_j}) -\vec{v} \big)\in \mathcal{C}_n(A,R,\e)$. Note that $\widetilde{\gamma}_n^{(j)}(s) =L \big( \gamma_n^{(j)}(s+s_{\tau_j}) \big)$. So we have $L(\eqref{0623.648} - \vec{v})$ is equal to
\begin{equation}\label{0623.6490}
     \Big\{ \widetilde{\gamma}_n(s) + \sum_{j=1}^n \lambda_j  \widetilde{\gamma}_n^{(j)}(s): s \in [0,(r R_j)^{-\frac1n} ], \;\;\; |\lambda_j| \leq (rR_j)^{-\frac{j}{n}}  \Big\}.
\end{equation}
Note that the range of $\lambda_j$ does not change.
Hence, to prove \eqref{0623.6480}, by applying the linear transformation to both sides of \eqref{0623.6480}, we may assume that $\tau_j$ contains the origin. Let us continue to use the notation $\widetilde{\gamma_n}$. Next, we rescale. Define
\begin{equation}\label{0623.5300}
    \widetilde{\widetilde{ \gamma } }_n(s):= L_1 \Big( \widetilde{\gamma}_n \big( (r R_j)^{-\frac1n} s \big) \Big)
\end{equation}
where 
\begin{equation}\label{0625.6510}
L_1(\xi):=\big( (rR_j)^{\frac{n}{n}} \xi_1,(rR_j)^{\frac{n-1}{n}} \xi_2 ,\ldots, (rR_j)^{\frac1n} \xi_n \big).    
\end{equation}
 Then this curve is defined on $[0,1]$, and it belongs to $\mathcal{C}_n(A,R,\e)$. Note that the image of \eqref{0623.6490} under the map $L_1$ is equal to
\begin{equation}
    \Big\{ \widetilde{\widetilde{\gamma}}_n(s) + \sum_{j=1}^n \lambda_j  \widetilde{\widetilde{\gamma}}_n^{(j)}(s): s \in [0,1 ], \;\;\; |\lambda_j| \leq 1  \Big\}.
\end{equation}
Moreover, the image of $\tau_{j+1}$ under the map $L_1$ becomes an element of $\Theta^n( r^{-1}R_{j}^{-1}  R_{j+1})$ associated to the curve $\widetilde{\widetilde{ \gamma_n }}$. So we can apply the definition of $S_p^*(1,r^{-1}R^{\e})$ (see Definition \ref{1003.def61} and below). After rescaling back, we obtain the claim. So the first term of \eqref{0623.6450} is bounded by the right hand side of \eqref{0622.643}. \\

Let us next bound the second term of \eqref{0623.6450}. We claim that 
\begin{equation}\label{0623.652}
    \int_{\R^{n}} |f_{\tau_j}^{j+1}|^{\frac{p}{2}}w_U \lesssim R^{\e^9}( R^{\frac{\e^5}{n}} )^{(\frac12-\frac2p)\frac{p}{2}} \sum_{\tau_j' \in \Theta^n(r R_j R^{\e^5 } ): \tau_j' \subset \tau_j} \int_{\R^{n}} |(f_{\tau_j}^{j+1})_{\tau_j'}|^{\frac{p}{2}} w_U
\end{equation}
for every $\tau_j \in \Theta^n(rR_j)$.

Let us prove the claim.
By using the same linear transformation (see the discussion below $\eqref{0623.648}$), define the curve $\widetilde{\gamma}_n(s):=L \big(\gamma_n(s+s_{\tau_j})-\vec{v} \big)$. Then, as before, by applying the linear transformation to both sides of \eqref{0623.652}, we may assume that $\widetilde{\gamma}_n \in \mathcal{C}_n(A,R,\e)$. It suffices to prove \eqref{0623.652} for the case that $\tau_j$ contains the origin. 
Note that $U \| U_{\tau}$, and $\tau$ also contains the origin. By the definition in Section \ref{sec3}, $U$ is comparable to 
\begin{equation}\label{0625.654}
    10r R_j\Big([- 1, 1] \times [-\sigma^{\frac1n}, \sigma^{\frac1n} ] \times \cdots \times [-\sigma^{\frac{n-1}{n}}, \sigma^{\frac{n-1}{n}} ] \Big).
\end{equation} 
The next step is to rescale.
We use the rescaling map $L_1$ in \eqref{0625.6510}, and define the curve $\widetilde{\widetilde{\gamma}}_n(t)$ by \eqref{0623.5300}. Note that after applying the linear transformations, $\tau_j'$ becomes an element of $\Theta^n(R^{\e^5})$. By definition, the dual box of an element of $\Theta^n(R^{\e^5})$ has dimensions
\begin{equation}\label{06256550}
    R^{\frac{\e^5}{n}} \times R^{\frac{2\e^5}{n}} \times \cdots \times R^{\frac{n\e^5}{n}}.
\end{equation}
On the other hand, the image of \eqref{0625.654} under the map $(L_1)^{-1}$ is as follows. 
\begin{equation}
    10 \Big( [-1,1] \times [-( \sigma rR_j)^{\frac1n},(\sigma rR_j)^{\frac1n}] \times \cdots \times [-(\sigma r R_j )^{\frac{n-1}{n}}, (\sigma r R_j)^{\frac{n-1}{n}} ] \Big).
\end{equation}
Since $\sigma \geq R^{(n-1)\e^5}(rR_j)^{-1}$, the above set contains
\begin{equation}
    10 \Big( [-1,1] \times [-R^{\frac{(n-1)\e^5}{n}} ,R^{\frac{(n-1)\e^5}{n}}]^{n-1} \Big).
\end{equation}
To apply the decoupling for a nondegenerate curve in $\R^{n-1}$ by \cite{MR3548534}, let us check the range. Since $4 \leq p \leq n^2+n-2$, we have $p/2 \leq (n-1)^2+(n-1)$. Hence, after freezing the first variable, the decoupling theorem gives the following inequality: for $4 \leq p \leq n^2+n-2$,
\begin{equation}
\begin{split}
    \int_{\R^{n}}|g|^{\frac{p}{2}} w_{L_1^{-1}(U)} &\lesssim  R^{\e^9}\Big(\sum_{\tilde{\tau} \in \Theta^{n-1}( R^{{(n-1)\e^5}/{n}} )} \big( \int_{\R^{n}} |g_{\tilde{\tau}}|^{\frac{p}{2}} w_{L_1^{-1}(U)} \big)^{\frac{2}{p} \cdot 2 } \Big)^{\frac12 \cdot \frac{p}{2}} 
    \\&
    \lesssim R^{\e^9} R^{\frac{\e^5}{n}(\frac12-\frac2p)\frac{p}{2} } \sum_{\tilde{\tau} \in \Theta^{n-1}( R^{{(n-1)\e^5}/{n}} )}  \int_{\R^{n}} |g_{\tilde{\tau}}|^{\frac{p}{2}} w_{L_1^{-1}(U)}  
\end{split}
\end{equation}
for all Schwartz functions $g$ whose Fourier supports are in $\mathcal{M}^n(\gamma_n;R^{(n-1)\e^5/n})$.
Here $g_{\tilde{\tau}}$ is a Fourier restriction in the last $(n-1)$ variables to $\tilde{\tau}$. The dual box $\tilde{\tau}^*$ has dimensions
\begin{equation}
    R^{\frac{\e^5}{n}} \times R^{\frac{2\e^5}{n}} \times \cdots \times R^{\frac{(n-1)\e^5}{n}}.
\end{equation}
Recall \eqref{06256550}.
 Hence, by applying the decoupling and rescaling back, we prove the claim \eqref{0623.652}. \\ 
 
 Let us now finish the proof. By \eqref{0623.652}, the second term of $\eqref{0623.6450}$ is bounded by
\begin{equation*}
R^{\e^9}
    R^{\frac{\e^5}{n}(\frac12-\frac2p){p}}
    \sum_{R^{(n-1)\e^5} (rR_j)^{-1} \leq \sigma \le 1 }\sum_{ \tau \in \Theta^n(\sigma^{-1})} \sum_{U\parallel U_{\tau}}
    |U| \Big( \int_{\R^{n}} \sum_{\tau_j': \tau_j' \subset \tau } |f_{\tau_j'}^{j+1}|^{\frac{p}{2}}   w_U \Big)^2.
\end{equation*}
By H\"{o}lder's inequality and embedding, this is bounded by
\begin{equation}
      R^{2\e^{9} } R^{\frac{\e^5}{n}(\frac12-\frac2p){p}}
     \Big\| \big( \sum_{\tau_j' \in \Theta^n(rR_j R^{\e^5})} |f_{\tau_j'}^{j+1}|^{\frac{p}{2}} \big)^{\frac2p}     \Big\|_{L^p(\R^n)}^p.
\end{equation}
By the uncertainty principle (see \eqref{0623.41}) and the induction hypothesis on $r$, this is bounded by
\begin{equation*} R^{2\e^{9}}
\big( \frac{R^{\e}}{r R^{\e^5}} \big)^{p\e^2} R^{2p \e^3} R^{\frac{\e^5}{n}(\frac12-\frac2p){p}} S_p^*(1,\frac{R^{\e}}{r R^{\e^5}} )^p   \Big\|(\sum_{\tau_{j+1} \in \Theta^n(R_{j+1}) }|f_{\tau_{j+1}}^{j+1}|^\frac{p}{2} *w_{\tau_{j+1}^*} )^{\frac2p} \Big\|_{L^p}^p.
\end{equation*}
By the definition of $S_p^*$, we have
\begin{equation}
    S_p^*(1,\frac{R^{\e}}{r R^{\e^5}} )  \lesssim  R^{-\frac{\e^5}{n}(\frac12-\frac2p)} S_p^*(1, \frac{R^{\e}}{r}),
\end{equation}
and hence,
\begin{equation}
     R^{2\e^9} \big( \frac{R^{\e}}{r R^{\e^5}} \big)^{\e^2} R^{2\e^3} R^{\e'} R^{\frac{\e^5}{n}(\frac12-\frac2p)} S_p^*(1,\frac{R^{\e}}{r R^{\e^5}} ) \lesssim \big( \frac{R^{\e}}{r} \big)^{\e^2} R^{2 \e^3} S_p^*(1, \frac{R^{\e}}{r}).
\end{equation}
This closes the induction on $r$ and completes the proof.

\subsection{Proof of Proposition \ref{05.25.prop84}}

Let us fix $\alpha>0$.
To do a broad-narrow analysis, 
fix $x \in U_{\alpha}$. Let us introduce sufficiently large $K$, which will be determined later. Define a collection of significant caps.
\begin{equation}
    \mathcal{C}:=\{ \tau \in \Theta^n(K) :  |(f^N)_\tau(x)| \geq K^{-100n}|f^N(x)|   \}.
\end{equation} 
We say $x$ is broad if $|\mathcal{C}|> K^{\e}$. Otherwise, we say $x$ is narrow. If $x$ is broad, then we can find $n$ many $\tau_j \in \mathcal{C}$ such that any two of them are not adjacent. Denote by $\mathrm{Nar}$ a collection of narrow points in $U_{\alpha}$. We are in the narrow case if $|U_{\alpha}| \lesssim |\mathrm{Nar}|$. Otherwise, we are in the broad case.

If we are in the narrow case, then for some $\widetilde{\alpha}$, we have  

\begin{equation}
    \alpha^p |U_{\alpha}| \lesssim \sum_{\tau \in \Theta^n(K) } \widetilde{\alpha}^p | \{x : |\sum_{\tau_N} (f^N_{\tau_N})_{\tau}| > \widetilde{\alpha} \} |.
\end{equation}
We now fix $\tau$ and repeat the rescaling argument in the proof of \eqref{0622.643}. By using a linear transformation (see the discussion below \eqref{0623.648}), we may assume that $\tau$ contains the origin. Next, by using scaling (see \eqref{0623.5300} and the discussion there) and the definition of $S_p(1,R/K)$, we have 
\begin{equation*}
    \widetilde{\alpha}^p \big|\{x : |(f^N)_{\tau}(x)|>\widetilde{\alpha} \}\big| \lesssim  S_p(1,R/K)^p  \Big\|(\sum_{\tau_N \in \Theta^n(R): \tau_N \subset \tau }|f_{\tau_N}^N|^{\frac{p}{2}}*w_{\tau_N^*})^{\frac2p} \Big\|_{L^p(\R^n)}^p.
\end{equation*}
Hence, by embedding, we have
\begin{equation*}
\begin{split}
    \alpha^p |U_{\alpha}| &\lesssim S_p(1,R/K)^p \Big\|(\sum_{\tau_N \in \Theta^n(R) }|f_{\tau_N}^{N}|^{\frac{p}{2}}*w_{\tau_N^*})^{\frac2p} \Big\|_{L^p(\R^n)}^p.
\end{split}
\end{equation*}

Suppose that we are in a broad case.
By pigeonholing, we can find $\beta>0$ so that
\begin{equation}
\begin{split}
    \alpha^p|U_{\alpha}| &\lesssim \alpha^p|U_{\alpha} \setminus \mathrm{Nar}|
    \\&
    \lesssim (\log R) \alpha^p |U_{\alpha,\beta} \setminus \mathrm{Nar}|+ R^{\frac1n(\frac12-\frac2p)p} \Big\|(\sum_{\tau_N \in \Theta^n(R) }|f_{\tau_N}^N|^{\frac{p}{2}}*w_{\tau_N^*})^{\frac2p} \Big\|_{L^p(\R^n)}^p.
\end{split}
\end{equation}
  We have
\begin{equation}
    |U_{\alpha,\beta} \setminus \mathrm{Nar}| \lesssim K^{100n} |\BR_{\alpha',\beta}^K  |,
\end{equation}
where the broad set is defined in \eqref{0702.75} and $\alpha'=K^{O(n)}\alpha$. For simplicity, abusing notation, let us still use $\alpha$ for $\alpha'$.
We have
\begin{equation}
\begin{split}
     (\log R)\alpha^p K^{100n} |\BR_{\alpha,\beta}^K| 
    \lesssim (\log R) K^{100n} S_{\BR,p,K}(1,R)^p \Big\|(\sum_{\tau_N \in \Theta^n(R) }|f_{\tau_N}^N|^{\frac{p}{2}}*w_{\tau_N^*})^{\frac2p} \Big\|_{L^p(\R^n)}^p.
\end{split}
\end{equation}

In summary, we have
\begin{equation*}
    \big( \alpha^p |U_{\alpha}| \big)^{\frac1p} \lesssim \big( S_p(1,\frac{R}{K}) +(\log R)^{\frac1p} K^{100n} S_{Br,p,K}(1,R) + R^{\frac1n(\frac12-\frac2p)+\epsilon} \big) \Big\|(\sum_{\tau_N \in \Theta^n(R) }|f_{\tau_N}^N|^{\frac{p}{2}}*w_{\tau_N^*})^{\frac2p} \Big\|_{L^p(\R^n)}.
\end{equation*}
This implies
\begin{equation}
    S_p(1,R) \leq C_p S_p(1,R/K) + C_p (\log R)^{\frac1p} K^{100n} S_{Br,p,K}(1,R)+R^{\frac1n(\frac12-\frac2p)+\epsilon}
\end{equation}
for any $R \geq 1$.
After applying the inequality repeatedly, we obtain
\begin{equation}
    S_{p}(1,R) \leq C_{\e} R^{\e}
    \sup_{1 \leq X \leq R} S_{Br,p,K}(1,X) + R^{\frac1n(\frac12-\frac2p)+\epsilon}
\end{equation}
provided that $K$ is sufficiently large.
This gives the desired result.

\subsection{Proof of Proposition \ref{0621.prop51} assuming Proposition \ref{1020.prop313}}

We now complete the proof of Proposition \ref{0621.prop51}, assuming Proposition \ref{1020.prop313}. By the interpolation result from Subsection \ref{0702.53}, it suffices to prove
the endpoint estimates. The intermediate range of \(p\) then follows by
interpolation. By replacing $S_p(1,R)$ by $\sup_{1 \leq X \leq R}S_p(1,X)$, we may assume that $S_p(1,R)$ is non-decreasing. For convenience, we introduce the notations
\begin{equation}
\begin{split}
    &\widetilde{S}_p(1,R) := R^{-\frac1n(\frac12-\frac2p)} S_p(1,R)
    \\& \widetilde{S_p}(R_k,R):=(\frac{R}{R_k})^{-\frac1n(\frac12-\frac2p)} S_p(R_k,R).
\end{split}
\end{equation}
We need to show that
\begin{equation}
    \begin{split}
\widetilde{S}_p(1,R) \lesssim R^{\e^{1/2}}.
    \end{split}
\end{equation}
As in Proposition \ref{05.01.thm61}, we will consider two cases: $n=3$ and $n \geq 4$.
Let us first consider $n \geq 4$.
With the new notations, Proposition \ref{05.25.prop84} and  \ref{05.01.thm61} give the following inequality.
\begin{equation}
    \widetilde{S_p}(1,R) \lesssim R^{10^n N_0\e} + R^{100^n \e} \sum_{k=N_0}^{N-1} \widetilde{S_p}(R_k,R).
\end{equation}
Recall that $R^{\e (N-k)} = R R_k^{-1}$.
Proposition \ref{05.24.prop53} becomes
\begin{equation}
\begin{split}
    \widetilde{S_p}(R_k,R) &= (\frac{R}{R_k})^{-\frac1n(\frac12-\frac2p)} S_p(R_k,R)
    \\& \lesssim R^{\e} (\frac{R}{R_k})^{-\frac1n(\frac12-\frac2p)} \big(\frac{R}{R_k}\big)^{\frac1n(\frac12-\frac2p)} \sup_{1 \leq X \leq R^{\e}}
 \widetilde{S_p}(1,X)^{N-k}
 \\& \lesssim R^{\e} \sup_{1 \leq X \leq R^{\e}}
 \widetilde{S_p}(1,X)^{N-k}.
\end{split}
\end{equation}
Combining these two, we obtain
\begin{equation}
    \widetilde{S_p}(1,R) \lesssim R^{10^n N_0\e} + R^{101^n \e}\sup_{1 \leq X \leq R^{\e}}
 \widetilde{S_p}(1,X)^{N-N_0}.
\end{equation}
Take $\eta$ so that
\begin{equation}
    \limsup_{R \rightarrow \infty} \frac{\widetilde{S_p}(1,R)}{R^{\eta}}>0, \;\;\; \limsup_{R \rightarrow \infty} \frac{\widetilde{S_p}(1,R)}{R^{\eta+\e'}}=0
\end{equation}
for every $\e'>0$.
Then we have
\begin{equation}
    \eta \leq \max \big( 10^n N_0 \e, 101^n\e + \eta \e(N-N_0) \big).
\end{equation}
We take $N_0=\e^{-1/2}$.
If the first term dominates, then the desired bound follows immediately. Suppose that the second term dominates.
Since $N\e = 1$, this implies
\begin{equation}
    \eta \e N_0 \leq 101^n \e.
\end{equation}
This gives $ \eta \leq 101^n/N_0$. It suffices to recall that $N_0=\e^{-1/2}$.
\medskip

Let us next consider $n=3$. With the new notations, Proposition \ref{05.25.prop84} and \ref{05.01.thm61} give the following inequality. 
\begin{equation}
    \widetilde{S_p}(1,R) \lesssim R^{100 N_0\e} + R^{100^3 \e} \sum_{k=N_0}^{N-1} \widetilde{S_8}(R_k,R)
\end{equation}
for any $4 \leq p \leq 10$. By Proposition \ref{05.24.prop53}, we have
\begin{equation}
\begin{split}
    \widetilde{S_8}(R_k,R)  \lesssim R^{\e}\sup_{1 \leq X \leq R^{\e}}
 \widetilde{S_8}(1,X)^{N-k}.
\end{split}
\end{equation}
Combining these two, we obtain
\begin{equation}\label{06.01.659}
    \widetilde{S_p}(1,R) \lesssim R^{100 N_0\e} + R^{101^3\e}\sup_{1 \leq X \leq R^{\e}}
 \widetilde{S_8}(1,X)^{N-N_0}.
\end{equation}
By the above inequality, it suffices to prove
\begin{equation}
    \widetilde{S}_8(1,R) \lesssim R^{100\e^{1/2}}.
\end{equation}
The rest of the proof is identical to that for the case $n \geq 4$. We omit the details.

\subsection*{Summary of the section}

In this section, we proved Proposition~\ref{0621.prop51} modulo  Proposition~\ref{1020.prop313}. The proof used the broad-narrow reduction, Proposition~\ref{05.25.prop84}, the key multiscale estimate, Proposition~\ref{05.01.thm61}, and the rescaling estimate, Proposition~\ref{05.24.prop53}. Combining these estimates closes the bootstrap after choosing $(N_0=\lfloor \epsilon^{-1/2}\rfloor)$. It remains to prove Proposition~\ref{1020.prop313}, which is the goal of the next section.

\section{Proof of Proposition \ref{1020.prop313}}\label{sec7}
We prove a generalization of Proposition \ref{1020.prop313}, which will
also be used in Section \ref{0624.sec8}.  Let $\gamma_{n+k} $ be an element of $\mathcal{C}_{n+k}(A,R,\e)$. For $k \geq 1$, define the $k$-flat cone $\Gamma_{n,k}(\gamma_{n+k};R) \subset \mathbb{R}^{n+k}$ to be 
\begin{equation}
  \Big\{  \sum_{j=1}^{n+k} \lambda_j \gamma_{n+k}^{(j)}(s): s \in [0,1], \; \lambda_k \geq \frac12,\; |\lambda_i| \leq  \min(1, R^{\frac{k-i}{n}}), \; 1 \leq i \leq n+k \Big\}.
\end{equation}
 Write $\Gamma_{n,k}(\gamma_{n+k};R) \subset \bigcup_{\theta \in \Xi_{n,k}(R)} \theta$ where $\Xi_{n,k}(R)$ is the collection of  $\theta$ of the form
\begin{equation}
     \Big\{ \sum_{j=1}^{n+k} \lambda_j \gamma_{n+k}^{(j)}(s):  s \in [s_{\theta},s_{\theta}+R^{-\frac1n} ], \, \lambda_k \geq \frac12,\, |\lambda_i| \leq  \min(1, R^{\frac{k-i}{n}}), \, 1 \leq i \leq n+k \Big\}
\end{equation}
where $s_{\theta} \in R^{-\frac1n}\mathbb{Z}$. Note that Proposition \ref{1020.prop313} is a special case of Proposition \ref{0618.prop71}.  Here the parameter $n$ in Proposition \ref{0618.prop71} is not the ambient dimension in Proposition \ref{1020.prop313}. If \(d\) is the ambient dimension in Proposition \ref{1020.prop313}, then
a one-sheet piece of \(\Gamma_d(\gamma_d;R)\) is
$\Gamma_{d-1,1}(\gamma_d;R).$
 Applying Proposition
\ref{0618.prop71} with \(n=d-1\), \(k=1\), and \(R\),
which gives Proposition \ref{1020.prop313}. Recall that $p_{n}=n^2+n-2$.

\begin{proposition}\label{0618.prop71}
Let $ n \geq 2$ and $k \geq 1$.
 Assume that Proposition \ref{0621.prop51} is true for $n$.
Then for $4 \leq p \leq p_{n}$ and $\e'>0$, we have
    \begin{equation*}
        \Big\|\sum_{\theta \in \Xi_{n,k}(R)  }f_{\theta}\Big\|_{L^p(\R^{n+k})} \leq C_{A,p,\e,\e'}(R^{\frac1n})^{(\frac12-\frac2p)+\e'} \Big\|(\sum_{\theta \in \Xi_{n,k}(R) }|f_{\theta}|^{\frac{p}{2}}*w_{\theta^*})^{\frac2p} \Big\|_{L^p(\R^{n+k})}
    \end{equation*}
    for all curves $\gamma_{n+k} \in \mathcal{C}_{n+k}(A,R,\e)$, and all Schwartz functions $f$ whose Fourier supports are in $\Gamma_{n,k}(\gamma_{n+k};R)$.
\end{proposition}

Our proof relies on a bootstrapping argument by \cite{MR2288738}. To do that, let us first define a truncated $k$-flat cone.
For $K \geq 100^{n+k}$, and $\vec{h}=(h_1,\ldots,h_k)$ with $|h_i| \leq 1$ and $|h_k| \geq \frac12$, define the truncated $k$-flat cone $\Gamma_{n,k}(\gamma_{n+k},K,\vec{h},R) \subset \mathbb{R}^{n+k}$ to be
\begin{equation}\label{0624.73}
\begin{split}
  \Big\{  \sum_{j=1}^{n+k} \lambda_j \gamma_{n+k}^{(j)}(s): \; s \in [0,1], \;\; h_j \leq \lambda_j \leq h_j+\frac1K,\; 1 \leq j \leq k,
  |\lambda_j| \leq   R^{\frac{k-j}{n}}, \; k+1 \leq j \leq n+k \Big\}.
  \end{split}
\end{equation}
 Write $\Gamma_{n,k}(\gamma_{n+k},K,\vec{h};R) \subset \bigcup_{\theta \in \Xi_{n,k,K,\vec{h}}(R)} \theta$ where $\Xi_{n,k,K,\vec{h}}(R)$ is the collection of  $\theta$ of the form
 \begin{equation*}
\begin{split}
  \Big\{  \sum_{j=1}^{n+k} \lambda_j \gamma_{n+k}^{(j)}(s): \; s \in [s_{\theta},s_{\theta}+R^{-\frac1n}], \;\; h_j \leq \lambda_j \leq h_j+\frac1K,\; 1 \leq j \leq k,
  |\lambda_j| \leq   R^{\frac{k-j}{n}}, \; k+1 \leq j \leq n+k \Big\}.
  \end{split}
\end{equation*}
where $s_{\theta} \in R^{-\frac1n}\mathbb{Z}$.
Define $D_{p,K}(R)$ to be the smallest constant such that
 \begin{equation*}
        \Big\|\sum_{\theta \in \Xi_{n,k,K,\vec{h}} (R) }f_{\theta  }\Big\|_{L^p(\R^{n+k})} \leq D_{p,K}(R) \Big\|(\sum_{\theta \in \Xi_{n,k,K,\vec{h}} (R) }|f_{\theta}|^{\frac{p}{2}}*w_{\theta^*})^{\frac2p} \Big\|_{L^p(\R^{n+k})}
    \end{equation*}
    for all curves $\gamma_{n+k} \in \mathcal{C}_{n+k}(A,R,\e)$, for all $\vec{h}$ with $|h_i| \leq 1$ and $|h_k| \geq \frac12$, and all Schwartz functions $f$ whose Fourier supports are in $\Gamma_{n,k}(\gamma_{n+k},K,\vec{h};R)$. To prove Proposition \ref{0618.prop71}, it suffices to prove that for $4 \leq p \leq p_n$
\begin{equation}\label{0624.75}
    D_{p,K}(R) \lesssim (R^\frac1n)^{\frac12-\frac2p+\e'}.
\end{equation}
We use an iterative argument. For brevity, write
\[
    \mathfrak X(r):=\Xi_{n,k,K,\vec h}(r).
\]
Let $D_{p,K}(R_1,R)$ be the smallest constant such that
\begin{equation*}
\begin{aligned}
&\Bigg\|\Bigg(
    \sum_{\tau\in\mathfrak X(R_1)}
    \Bigg|\sum_{\substack{\theta\in\mathfrak X(R)\\ \theta\subset\tau}}
        f_{\theta}\Bigg|^{\frac p2}*w_{\tau^*}
    \Bigg)^{\frac2p}\Bigg\|_{L^p(\R^{n+k})}
\\
&\qquad\leq D_{p,K}(R_1,R)
    \Bigg\|\Bigg(
    \sum_{\theta\in\mathfrak X(R)}
        |f_{\theta}|^{\frac p2}*w_{\theta^*}
    \Bigg)^{\frac2p}\Bigg\|_{L^p(\R^{n+k})}.
\end{aligned}
\end{equation*}
for all curves $\gamma_{n+k} \in \mathcal{C}_{n+k}(A,R,\e)$, for all $\vec{h}$ with $|h_i| \leq 1$ and $|h_k| \geq \frac12$, and all Schwartz functions $f$ whose Fourier supports are in $\Gamma_{n,k}(\gamma_{n+k},K,\vec{h};R)$. Note that 
\begin{equation}\label{0624.76}
    D_{p,K}(R) \lesssim K^{O(1)}D_{p,K}(K,R), \;\;\; D_{p,K}(K^{-\e}R,R) \lesssim K^{O(1)}.
\end{equation}
\begin{proposition}[Iterative formula]\label{0611.prop73}  Let $4 \leq p \leq p_n$. For every $R \geq K$ and $R \geq R_1K^{\frac1n}$, we have
    \begin{equation}
    \begin{split}
        D_{p,K}(R_1,R) \lesssim  (K^{\frac{\e}{n}})^{\frac12-\frac2p+2\epsilon^5} D_{p,K}(   K^{{\e } } R_1,R)
        +(K^{\frac1n})^{\frac12-\frac2p+\e } D_{p,K}(   K R_1,R).
        \end{split}
    \end{equation}
\end{proposition}
The desired bound \eqref{0624.75} follows by applying \eqref{0624.76} and Proposition \ref{0611.prop73} repeatedly. So it suffices to prove the above proposition. The proof of Proposition \ref{0611.prop73} is analogous to that of Proposition \ref{05.24.prop53}. To do that, we need the following lemma.

\begin{lemma}\label{0616lem73} Assume that Proposition \ref{0621.prop51} is true for $n$. For every $R \geq K^{\frac{100}{\e}}$, $4 \leq p \leq p_{n}$, and $\epsilon'>0$, we have
    \begin{equation*}
        \Big\| \sum_{\theta \in \Xi_{n,k,K,\vec{h} }(R) }f_{\theta} \Big\|_{L^p(\R^{n+k})} \leq C_{p,\e'} (K^{\frac1n})^{\frac12-\frac2p+\e'} \Big\| \big(\sum_{\tau \in \Xi_{n,k,K,\vec{h} }(K) } \big| \sum_{\theta \subset \tau}f_{\theta} \big|^{\frac{p}{2}}*w_{\tau^*} \big)^{\frac2p} \Big\|_{L^p(\R^{n+k})}. 
    \end{equation*}
    for all curves $\gamma_{n+k} \in \mathcal{C}_{n+k}(A,R,\e)$, for all $\vec{h}$ with $|h_i| \leq 1$ and $|h_k| \geq \frac12$, and all Schwartz functions $f$ whose Fourier supports are in $\Gamma_{n,k}(\gamma_{n+k},K,\vec{h};R)$.
\end{lemma}

Lemma \ref{0616lem73} can be viewed as the base case in the Pramanik-Seeger's bootstrapping argument. Since we have truncated the $k$-flat directions to length $K^{-1}$, these directions become harmless. The problem is reduced to that for a nondegenerate curve $\ga_n$ in $\R^n$.

In this section, we use rescaling frequently. So let us introduce a rescaling map here.  For any $1 \leq J \leq n+k$ and $X  >0$, we define the linear transformation $L_{J,X}:\mathbb{R}^{n+k} \rightarrow \mathbb{R}^{n+k}$ by 
\begin{equation}\label{0620.71}
    L_{J,X}(\xi):=\Big(X^{n+k-J}\xi_1,  \ldots, X\xi_{n+k-J}, \xi_{n+k+1-J},\ldots, X^{-(J-1)}\xi_{n+k} \Big).
\end{equation}
Let us begin the proof of Lemma \ref{0616lem73}. 

\begin{proof}[Proof of Lemma \ref{0616lem73}] The first step is to decompose the frequency by the triangle inequality.
\begin{equation}\label{0625.79}
    \Big\| \sum_{\theta \in \Xi_{n,k,K,\vec{h} }(R) }f_{\theta} \Big\|_{L^p(\R^{n+k})} \lesssim
    \Big(\sum_{\tau_0 \in \Xi_{n,k,K,\vec{h}}(100^{n(n+k)}) }\Big\|  \sum_{\theta \subset \tau_0 }f_{\theta}  \Big\|_{L^p(\R^{n+k})}^p \Big)^{\frac1p}.
\end{equation}
Let us fix $\tau_0$. By translation (see the discussion around \ref{0623.648}), we may assume that $\tau_0$ contains the origin.
By definition (see \eqref{0624.73}), the intersection of the $k$-flat cone with $\tau_0$ is as follows. 
\begin{equation*}
\begin{split}
  \Big\{  \sum_{j=1}^{n+k} \lambda_j \gamma_{n+k}^{(j)}(s): \; s \in [0,100^{-(n+k)}], \;\; h_j \leq \lambda_j \leq h_j+\frac1K,\; 1 \leq j \leq k,
  |\lambda_j| \leq   R^{\frac{k-j}{n}}, \; k+1 \leq j \leq n+k \Big\}.
  \end{split}
\end{equation*}
This set is equal to
\begin{equation*}
\begin{split}
  \Big\{  \sum_{j=1}^k h_j \gamma_{n+k}^{(j)}(s)+ \sum_{j=1}^{n+k} \lambda_j \gamma_{n+k}^{(j)}(s):  s \in  [0, 100^{-(n+k)}], \; 0 \leq \lambda_j \leq \frac1K,\; 1 \leq j \leq k,
  |\lambda_j| \leq   R^{\frac{k-j}{n}}, \; k+1 \leq j \leq n+k \Big\}.
  \end{split}
\end{equation*}
Since $\gamma_{n+k} \in \mathcal{C}_{n+k}(A,R,\e)$ and $R \geq K^{\frac{100}{\e}}$, we can ignore the perturbation term in the definition of $\gamma_n$ (see \eqref{0623.51def}), and pretend that $\gamma_{n+k}(s)=(\frac{s^{n+k}}{(n+k)!},\frac{s^{n+k-1}}{(n+k-1)!},\ldots,\frac{s}{1!})$. Then the above set is contained in
\begin{equation}\label{0625.712}
\begin{split}
  \mathcal{N}_{2K^{-1}} \Big\{  \sum_{j=1}^k h_j \gamma_{n+k}^{(j)}(s): s \in [0,100^{-(n+k)}]   \Big\}.
  \end{split}
\end{equation}
Recall that $\frac12 \leq h_k \leq 1$.
Next, we rescale using the map $L_{k,100^{n+k} }(\xi)$ (see \eqref{0620.71}). Define the curve $\widetilde{\gamma}_{n+k}(s):= L_{k,100^{n+k}} \big( {\gamma}_{n+k} \big( s/100^{n+k} \big) \big)$ for $s \in [0,1]$.
 Note that
\begin{equation}
    \widetilde{\gamma}_{n+k}^{(j)}(s)= 100^{-(n+k)j} L_{k,100^{n+k} } \big( {\gamma}_{n+k}^{(j)} \big( s/100^{n+k} \big) \big).
\end{equation}
If we pretend that $\gamma_{n+k}(s)=(\frac{s^{n+k}}{(n+k)!},\frac{s^{n+k-1}}{(n+k-1)!},\ldots,\frac{s}{1!})$, then we have
\begin{equation}
    100^{(n+k)k}  \widetilde{\gamma}_{n+k}^{(k)}(s)= \big( \frac{s^n}{n!},\ldots,\frac{s}{1},1,0,\ldots,0 \big).
\end{equation}
Then the image of \eqref{0625.712} under the rescaling map $L_{k,100^{n+k} }$ is given by
\begin{equation}\label{0625.7132}
    \mathcal{N}_{2K^{-1}} \Big\{  \sum_{j=1}^k 100^{(n+k)j} h_j  \widetilde{\gamma}_{n+k}^{(j)}(s): s \in [0,1]   \Big\}.
\end{equation}
We now see the reason why we applied the inequality \eqref{0625.79}. Define $\widetilde{\widetilde{\gamma}}_{n}(s)$ to be the projection of $\sum_{j=1}^k 100^{(n+k)j} h_j \widetilde{\gamma}_{n+k}^{(j)}(s)$ to the  first $n$ coordinates. We write
\begin{equation}\label{0706.714}
    \sum_{j=1}^k 100^{(n+k)j} h_j \widetilde{\gamma}_{n+k}^{(j)}(s) = 100^{(n+k)k} h_k \widetilde{\gamma}_{n+k}^{(k)}(s)+ \sum_{j=1}^{k-1} 100^{(n+k)j} h_j \widetilde{\gamma}_{n+k}^{(j)}(s).
\end{equation}
Recall that $|h_j| \lesssim 1$ and $|h_k| \sim 1$.
The first term on the right hand side is ``essentially'' equal to
$
     (\frac{s^n}{n!}, \ldots, \frac{s}{1!},1,0,\ldots,0).$
The second term on the right hand side of \eqref{0706.714} can be viewed as a small perturbation of the first term. Hence, one may see that $\widetilde{\widetilde{\gamma_{n}}}$ is a nondegenerate curve in $\mathbb{R}^n$.
  After getting rid of  the last $k$ coordinates,  \eqref{0625.7132} is contained in 
\begin{equation}
     \mathcal{N}_{CK^{-1}}\Big\{ \widetilde{\widetilde{\gamma}}_{n}(s) : 0 \leq s \leq 1 \Big\} \times \mathbb{R}^k.
\end{equation} 
By following the same argument, the Fourier support of $\sum_{\theta \subset \tau}f_{\theta}$ is contained in 
\begin{equation}
    \mathcal{N}_{CK^{-1}}\Big\{ \widetilde{\widetilde{\gamma}}_{n}(s): s\in [s_{\tau},s_{\tau}+K^{-\frac1n}] \Big\} \times \mathbb{R}^k.
\end{equation}
Therefore, after freezing the last $k$ variables, we can apply the result  for a nondegenerate curve (Proposition \ref{0621.prop51}), and obtain the 
desired result. 
\end{proof}

Let us next prove Proposition \ref{0611.prop73}. The proof is very similar to that for Proposition \ref{05.24.prop53}. Compared to Proposition \ref{05.24.prop53}, the main difference is that we are dealing with the $k$-flat cone here. This makes the rescaling more involved.

\begin{proof}[Proof of Proposition \ref{0611.prop73}] 

Let us first state our goal. To simplify the notation, we use the notation $\Xi(R)$ for $\Xi_{n,k,K,\vec{h}}(R)$. By the definition of $D_{p,K}$, what we need to prove is as follows.
\begin{equation}\label{0627.7188}
    \begin{split}
        &\Big\|(  \sum_{\tau\in \Xi(R_1)} \big| \sum_{ \theta \in \Xi(R): \theta \subset \tau} f_{\theta} \big|^{\frac{p}{2}}*w_{\tau^*})^{\frac2p} \Big\|_{L^p(\R^{n+k})} 
        \\&
        \lesssim 
        \Big((K^{\frac{\e}{n}})^{\frac12-\frac2p+ 2\e^5} D_{p,K}(   K^{{\e} } R_1,R)
        \\ & \;\;\;\;\;\;\;\;\;\;\;+(K^{\frac1n})^{\frac12-\frac2p + \e} D_{p,K}(   K R_1,R) \Big) \Big\|(\sum_{\theta \in \Xi(R) }|f_{\theta}|^{\frac{p}{2}}*w_{\theta^*})^{\frac2p} \Big\|_{L^p(\R^{n+k})}
    \end{split}
\end{equation}
for all curves $\gamma_{n+k} \in \mathcal{C}_{n+k}(A,R,\e)$, for all $\vec{h}$ with $|h_i| \leq 1$ and $|h_k| \geq \frac12$, and all Schwartz functions $f$ whose Fourier supports are in $\Gamma_{n,k}(\gamma_{n+k},K,\vec{h};R)$.

To use Proposition \ref{0616.thm33}, let us introduce some notations. Let  $q=\frac{p}{2}$. Define 
\begin{equation}\label{0628.718}
    \vec{\delta}:=( \underbrace{1, \ldots, 1}_{k-1}, R_1^{-\frac{1}{n}}, \underbrace{1, \ldots, 1}_{n-2} ), \;\; \vec{\delta''}:=( \underbrace{1, \ldots, 1}_{k-1}, R^{-\frac{1}{n}}, \underbrace{1, \ldots, 1}_{n-2} ).
\end{equation}
Denote
\[ \vec\nu_{\pm}:=( \underbrace{0, \ldots, 0}_{k-1}, \pm 1, \underbrace{0, \ldots, 0}_{n-2} ). \]
Note that for given $\theta \in \Xi(R)$ there exists $s''$ so that $f_{\theta}=f_{p[\vec{\delta}'',\vec\nu_+ ](s'')}+f_{p[\vec{\delta}'',\vec\nu_- ](s'')}$.
By the triangle inequality, we just consider one branch $f_{\theta}=f_{p[\vec{\delta}'',\vec\nu_+](s'')}$. We may omit the $\vec\nu$-notation and write $f_{\theta}=f_{p[\vec{\delta}''](s'')}$.
So we may say that $\Xi(R)$ and $\{p[\vec{\delta}''](s'')\}_{s'' \in \I_{R^{-1/n}} }$ are the same.
For simplification, use the notation 
\begin{equation}
    g_{s}:=\sum_{s'' \in \I_{ R^{-1/n} }: |s-s''| \leq R_1^{-1/n}  } f_{p[\vec{\delta''}](s'') }.
\end{equation}
By Proposition \ref{0616.thm33} and dilation, we have 
 \begin{equation*}
 \begin{split}
        \int_{ \mathbb{R}^{n+k} } \Big(\sum_{s\in\I_{R_1^{-1/n}} } |g_s|^{\frac{p}{2}} \Big)^{2}
        \lessapprox \sum_{R_1^{-1/n} \le \si\le 1} \sum_{s'\in\I_\si} \sum_{U\parallel (q[\vec{\delta};\le\si](s'))^*} \frac{1}{|U|} \Big(\int_U \sum_{s:|s-s'|\le \si}|g_s|^{ \frac{p}{2} } \Big)^2.
    \end{split}
    \end{equation*}
    We consider two cases depending on the size of $\sigma$: $\sigma \leq K^{\e}R_1^{-1/n}$ or $\sigma \geq K^{\e}R_1^{-1/n}$.
    In other words,
    we write the right hand side as  two terms, and bound each term individually.
    \begin{equation}\label{0627.718}
        \begin{split}
            &\sum_{R_1^{-1/n} \le \si\le  K^{\e} R_1^{-1/n} } \sum_{s'\in\I_\si} \sum_{U\parallel (q[\vec{\delta};\le\si](s'))^*} \frac{1}{|U|} \Big(\int_U \sum_{s \in \I_{R_1^{-1/n}} :|s-s'|\le \si}|g_s|^{ \frac{p}{2} } \Big)^2
            \\&
            +\sum_{K^{\e} R_1^{-1/n} \leq \sigma \leq 1} \sum_{s'\in\I_\si} \sum_{U\parallel (q[\vec{\delta};\le\si](s'))^*} \frac{1}{|U|} \Big(\int_U \sum_{s \in \I_{R_1^{-1/n}} :|s-s'|\le \si}|g_s|^{ \frac{p}{2} } \Big)^2.
        \end{split}
    \end{equation}
    In the language of the high/low method, the former term corresponds to the high term and the latter term corresponds to the low term. As we mentioned before, we may say that $\{\theta \}_{\theta \in \Xi(R)}$ and $\{ p[\vec{\delta''}](s'')  \}_{s'' \in \I_{R^{-1/n}}}$ are the same. Hence, it suffices to bound \eqref{0627.718} by the right hand side of \eqref{0627.7188}. \\
    
    Let us bound the first term of \eqref{0627.718}.
    Suppose that $\sigma \leq K^{\e} R_1^{-1/n}$. For given $s' \in \I_{\sigma}$, the number of $s \in \I_{R_1^{-1/n}}$ such that $|s-s'| \leq \sigma$ is bounded by $K^{\e}$. So by H\"{o}lder's inequality, the first term of \eqref{0627.718} is bounded by
    \begin{equation}\label{0618.714}
    \begin{split}
     &\lesssim  K^{O(\e)} \sum_{ s \in \I_{R_1^{-1/n}}}    \int_{\R^{n+k}} |g_s|^p 
     \\&
     \lesssim K^{O(\e)} \sum_{ s \in \I_{R_1^{-1/n}}} \int_{\R^{n+k}} \Big| \sum_{s'' \in \I_{ R^{-1/n} }: |s-s''| \leq R_1^{-1/n}  } f_{p[\vec{\delta''}](s'' )} \Big|^p
     \\&
     \lesssim K^{O(\e)} \sum_{  \tau \in \Xi(R_1) } \int_{\R^{n+k}}  \Big| \sum_{ \theta \in \Xi(R): \theta \subset \tau} f_{\theta} \Big|^p.
     \end{split}
    \end{equation}
    We next do rescaling and apply Lemma \ref{0616lem73}. Let us first fix $\tau \in \Xi(R_1)$. By using the linear transformation in the discussion below \eqref{0623.648}, we may assume that $\tau$ contains the origin. Hence, by definition,
    \begin{equation}
\begin{split}
\tau=
       \Big\{ \sum_{j=1}^{n+k} a_j \gamma_{n+k}^{(j)}(0) :\;  h_j \leq  a_j \leq h_j+\frac{1}{K}, \; 1 \leq j \leq k, \;\;  
       0\le a_j \le (R_1
       )^{\frac{k-j}{n}}, \;\; k+1\le j\le n+k  \Big\}.
    \end{split}
    \end{equation}
    Also, $\theta \subset \tau$ is expressed by
    \begin{equation}
\begin{split}
       \Big\{ \sum_{j=1}^{n+k} a_j \gamma_{n+k}^{(j)}(s_{\theta}) :\;  h_j \leq  a_j \leq h_j+\frac{1}{K}, \; 1 \leq j \leq k, \;\;  
       0\le a_j \le (R
       )^{\frac{k-j}{n}}, \;\; k+1\le j\le n+k  \Big\}
    \end{split}
    \end{equation}
    where $|s_{\theta}| \leq R_1^{-1/n}$.
    We use the rescaling map $L_{k,R_1^{1/n}}$ (see \eqref{0620.71} for the definition). 
    Since $\gamma_{n+k} \in \mathcal{C}_{n+k}(A,R,\e)$, we have the expression
    \begin{equation}
        \gamma_{n+k}(s)= \big(\frac{s^{n+k}}{(n+k)!},\frac{s^{n+k-1}}{(n+k-1)!},\ldots,\frac{s}{1!} \big) + R^{-\e}s^{n+k+1}E_{\e}(s)
    \end{equation}
    where $E_{\e}(s)$ is a polynomial of degree at most $\e^{-1}$ and coefficients  are bounded by $A$.
    Let us see how $\theta \subset \tau$ changes under the rescaling. Define 
    \begin{equation}\label{0630.7250}
        \widetilde{\gamma}_{n+k}(s):=R_1^{k/n} L_{k,R_1^{1/n}} \big( \gamma_{n+k}(R_1^{-1/n}s) \big).
    \end{equation}
    Note that $\widetilde{\gamma}_{n+k} \in \mathcal{C}_{n+k}(A,R,\e)$. We have 
    \begin{equation}
        \widetilde{\gamma}_{n+k}^{(j)}(s)=R_1^{-j/n} R_1^{k/n} L_{k,R_1^{1/n}} \big( \gamma_{n+k}^{(j)}(R_1^{-1/n}s) \big).
    \end{equation}
    Let $\widetilde{s_{\theta}}=R_1^{1/n}s_{\theta}$ and $\tilde{h}_j=R_1^{-\frac{k-j}{n}} h_j$.
    The image of $\theta$ under the map $L_{k,R_1^{1/n}}$ is 
    \begin{equation}\label{0628.7270}
\begin{split}
       \Big\{ \sum_{j=1}^{n+k} a_j \widetilde{\gamma}_{n+k}^{(j)}(\widetilde{s}_{\theta} ) :\;   \widetilde{h}_j \leq  a_j \leq \widetilde{h}_j+\frac{ R_1^{-\frac{k-j}{n}} }{K}, \; 1 \leq j \leq k, \;\;  
       0\le a_j \le   (R_1^{-1}R)^{\frac{k-j }{n}}, \;\; k+1\le j\le n+k  \Big\}.
    \end{split}
    \end{equation}
So we can apply Lemma \ref{0616lem73}. After rescaling back, \eqref{0618.714} is bounded by 
\begin{equation}
\begin{split}
    & \lesssim K^{O(\e)} K^{\frac1n(\frac12-\frac2p)p} \sum_{\tau} \Big\| \big(\sum_{ \tau' \in \Xi(KR_1) : \tau' \subset \tau} \big| \sum_{\theta \subset \tau'}f_{\theta}|^{\frac{p}{2}}*w_{(\tau')^*} \big)^{\frac2p} \Big\|_{L^p(\R^{n+k})}^p 
    \\&
    \lesssim K^{O(\e)} K^{\frac1n(\frac12-\frac2p)p}  \Big\| \big(\sum_{\tau' \in \Xi(KR_1) } \big| \sum_{\theta \subset \tau'}f_{\theta}|^{\frac{p}{2}}*w_{(\tau')^*} \big)^{\frac2p} \Big\|_{L^p(\R^{n+k})}^p.
\end{split}
\end{equation}
After applying the definition of $D_{p,K}(K R_1,R)$ to the right hand side, this is further bounded by 
\begin{equation}
    K^{O(\e)} K^{\frac1n(\frac12-\frac2p)p} D_{p,K}(KR_1,R)^p  \Big\| \big(\sum_{\theta} \big| f_{\theta}|^{\frac{p}{2}}*w_{\theta^*} \big)^{\frac2p} \Big\|_{L^p(\R^{n+k})}^p.
\end{equation}
This is the second term on the right hand side of \eqref{0627.7188}. This takes care of the first term of \eqref{0627.718}.
\\

Let us bound the second term of \eqref{0627.718}. We claim that for every $\tau \in \Xi(R_1)$ we have
\begin{equation*}
    \int_{\R^{n+k}} \Big|\sum_{\theta \in \Xi(R): \theta \subset \tau}f_{\theta} \Big|^{\frac{p}{2}} w_U \lesssim  (K^{\frac{\e}{n} })^{(\frac12-\frac{2}{p})\frac{p}{2}+\e^5}  \sum_{ \tau' \in \Xi(K^{\e} R_1): \tau' \subset \tau} \int_{\R^{n+k}}  \Big| \sum_{\theta \subset \tau'} f_{\theta} \Big|^{\frac{p}{2}} w_U
\end{equation*}
for any $U\| (q[\vec{\delta};\leq \si](s'))^*$. Here, we recall that $\vec{\delta}$ is defined in $\eqref{0628.718}$, $K^{\e} R_1^{-1/n} \leq \sigma \leq 1$, and $s'$ corresponds to $\tau$.

Once we have the claim, by the claim, the second line of \eqref{0627.718} is bounded by 
\begin{equation}
    (K^{ \frac{\e}{n} })^{(\frac12-\frac{2}{p})p+2\e^5}  \int_{\R^{n+k}} \Big(\sum_{\tau'} \big| \sum_{\theta \subset \tau'} f_{\theta} \big|^{\frac{p}{2}}*w_{(\tau')^*} \Big)^2.
\end{equation}
Apply the definition of $D_{p,K}$, and the above inequality is bounded by
\begin{equation}
    (K^{\frac{\e}{n} })^{(\frac12-\frac{2}{p})p+2\e^5} D_{p,K}( K^{\e} R_1 ,R)^p 
    \int_{\R^{n+k}} \Big(\sum_{\theta} \big|  f_{\theta} \big|^{\frac{p}{2}}*w_{\theta^*} \Big)^2.
\end{equation}
This is the first term on the right hand side of \eqref{0627.7188}.

Let us prove the claim. Fix $\tau \in \Xi(R_1)$. As we did before, by using the linear transformation (see the discussion around \eqref{0623.648}), we may assume that $\tau$ contains the origin (in other words, $s'=0$). Let us compute the size of $U$. Recall that 
    \[ q[\vec\de;\le\si]=\{ (a_i): |a_i|\le \frac{1}{R_1} \min\{ \si^{i-n-k},R_1\rho_{i-1}\} (1\le i\le n+k)   \}. \]
Hence, by definition, we have 
\begin{equation}\label{0630.5293}
q[\vec\de;\le\si](0)=\{ (a_i): |a_i|\le \frac{1}{R_1} \min\{ \si^{1-i},R_1\rho_{n+k-i}\} (1\le i\le n+k)   \}.    
\end{equation}
The coordinate is reversed since $\bfe_i(0)=(\underbrace{0,\cdots,0}_{n+k-i},1,\underbrace{0,\cdots,0}_{i-1})$.

Note that
\begin{equation}
    \sigma^{1-i} \leq R_1 \rho_{n+k-i} = R_1 (R_1^{-\frac1n})^{n-i+1} \;\; \mathrm{for} \;\; 1 \leq i \leq n+1.
\end{equation}
We don't use any information on the last $k$ coordinates of $q[\vec\de;\le\si](0)$. We write $U$  as follows. 
\begin{equation}
   [-R_1,R_1] \times [-R_1 \sigma,  R_1 \sigma ] \times \cdots \times [-R_1 \sigma^{n-1}, R_1 \sigma^{n-1}] \times B
\end{equation}
where $B$ is a set in $\R^k$.
The next step is to apply the rescaling map from $\eqref{0620.71}$.
 After applying the inverse map of $L:=L_{k,R_1^{1/n}}$, we have
\begin{equation}\label{0628.7330}
L^{-1}(U)=
    [-1,  1] \times [-R_1^{\frac1n} \sigma,R_1^{\frac1n} \sigma  ] \times \cdots \times [-R_1^{\frac{n-1}{n}}\sigma^{n-1},R_1^{\frac{n-1}{n}}\sigma^{n-1}] \times \widetilde{B}
\end{equation}
for some rectangular box $\widetilde{B} \subset \mathbb{R}^{k}$. We don't use any information on $\widetilde{B}$. Since $\sigma \geq K^{\e}R_1^{-1/n}$, this set is large enough for our purpose. Note that we used the same rescaling map to deal with the first term of \eqref{0627.718}.
On the other hand, the image of $\theta \in \Xi(R)$ satisfying $\theta \subset \tau$ under the map $L$ becomes $\eqref{0628.7270}$. Hence, $\cup_{\theta \subset \tau} L(\theta)$ is contained in
\begin{equation*}
\begin{split}
       \Big\{ \sum_{j=1}^{n+k} a_j \widetilde{\gamma}_{n+k}^{(j)}(s) :\; 0 \leq s \leq 1, \; \widetilde{h}_j \leq  a_j \leq \widetilde{h}_j+\frac{ R_1^{-\frac{k-j}{n}} }{K}, \; 1 \leq j \leq k, \;\;  
       0\le a_j \le   (R_1^{-1}R)^{\frac{k-j }{n}}, \;\; k+1\le j\le n+k  \Big\}.
    \end{split}
    \end{equation*}
Here $\widetilde{\gamma}_{n+k} \in \mathcal{C}_{n+k}(A,R_1,\e)$, $|\widetilde{h}_j| \lesssim R^{-\frac{k-j}{n}}$ and $|\widetilde{h}_k| \sim 1$. So, the contribution from $a_1,\ldots,a_{k-1}$ is minor. Take $\gamma_n(s)$ to be the first $n$ coordinates of the curve $\widetilde{h}_k\widetilde{\gamma}_{n+k}^{(k)}(s)$. Then $\gamma_n$ becomes a nondegenerate curve in $\mathbb{R}^n$. Also, the set $L(\theta)$ is contained in
\begin{equation}
\begin{split}
     \mathcal{N}_{K^{-1}} \Big\{  \gamma_{n}(R_1^{-\frac1n} s_{\theta})+ \sum_{j=1}^{n}a_j \gamma_{n}^{(j)}(R_1^{-\frac1n} s_{\theta}) : \;  
    |a_j| \leq ( R_1^{-1}R )^{-\frac{j}{n}}, \;\; 1 \leq j \leq n    \Big\} \times \mathbb{R}^{k}.
\end{split}
\end{equation}
After this change of variables, the claim reduces to proving
\begin{equation}\label{0628.737o}
    \int \Big|\sum_{\theta \subset \tau}f_{L(\theta)} \Big|^{\frac{p}{2}} w_{L^{-1}(U)} \lesssim K^{\frac{\e}{n}(\frac12-\frac2p)\frac{p}{2}+\frac{\e^6}{n} } \sum_{\tau' \subset \tau }  \int \Big|     \sum_{\theta \subset \tau'}f_{L(\theta)} \Big|^{\frac{p}{2}} w_{L^{-1}(U)}.
\end{equation}
Since the first coordinate of $L^{-1}(U)$ is small (recall \eqref{0628.7330}), we use only the coordinates from the second to the $n$th. Note that $L^{-1}(U)$ contains
\begin{equation}
    [-1,  1] \times [-K^{\e},K^{\e}  ]^{n-1} \times \widetilde{B}.
\end{equation}
We freeze the first coordinate and the coordinates from the $(n+1)$th to the $(n+k)$th.
To apply  the decoupling for a nondegenerate curve in $\mathbb{R}^{n-1}$ by  \cite{MR3548534}, we need to check that  $\frac{p_{n}}{2} \leq (n-1)^2+(n-1)$. Since $p_{n}=n^2+n-2$, the inequality is true for any $n \geq 2$. Hence, we obtain \eqref{0628.737o}, and this finishes the proof.
\end{proof}

\section{Proof of Theorem \ref{05.29.thm28}}\label{0624.sec8}

In this section, we prove Theorem \ref{05.29.thm28}, and hence Theorem \ref{mainthm}. Let us recall the theorem.

\begin{theorem}\label{0629.618} Fix an admissible pair $(\vec\de,\vec\nu)$  at scale $R$.  Let $\de=\prod_{i=1}^{n-1}\de_i$. Let $\I_{\delta}$ be a $\delta$-separated subset of $[0,1]$. For $4 \leq p \leq 4(n-1)$ and $\e>0$, we have
    \begin{equation}\label{0630.new}
        \Big\|  \sum_{s \in \I_{\delta}}a_s \Big(\wh f(\xi)\phi_{p[\vec{\de},\vec{\nu}](s)} \Big)^\vee(\cdot-R\bfn_s)  \Big\|_{L^p(W_{B_R^{n+1}})}\lesssim \de^{-1} R^{-1} R^{\frac1p}R^{-\frac1n(\frac12+\frac2p)+\e}\|f\|_{L^p(W_{B_R^n})}
    \end{equation}
    for any function $f$ with $\mathrm{supp}\wh f\subset B_1^n\setminus B_{\frac12}^{n}$,  $|a_s|\lesssim 1$, and $\bfn_s \in \Ga(s)$.
\end{theorem}

Since the role of $\vec{\nu}$ is minor, and since the argument is uniform in the finitely many possible sign patterns of $\vec{\nu}$, we suppress  $\vec{\nu}$
 from the notation and write $$p[\vec\de](s)=p[\vec\de,\vec\nu](s).$$
To simplify  notation, define the operator
\begin{equation}
Tf(x,t):=\sum_{s \in \I_{\delta}}a_s \big(\wh f(\xi)\phi_{p[\vec\de](s)} \big)^\vee \big((x,t)-R\bfn_s \big).
\end{equation}
We also write
\begin{equation}
    T_{s}f(x,t):= a_s \big(\wh f(\xi)\phi_{p[\vec\de](s)} \big)^\vee \big((x,t)-R\bfn_s \big).
\end{equation}
  Recall that $p[\vec\de](s)$ is the $R^{-1}$-dilation of the set $P[\vec{\delta},\vec\nu](s)$, and this rectangular box is defined in \eqref{deltaplank}. So $p[\vec\de](s)$ has dimension
\begin{equation}
    1 \times \rho_1 \times \cdots \times \rho_{n-1} \times R^{-1},
\end{equation}
where $\rho_j=\delta_1^{j} \delta_2^{j-1} \cdots \delta_j$. With the new notations, our goal is to prove 
\begin{equation}
        \| Tf\|_{L^{p}(W_{B_R^{n+1}})}\lesssim \delta^{-1} R^{-1} R^{\frac{1}{p}}R^{-\frac1n (\frac12+\frac{2}{p})+\e}\|f\|_{L^{p}(W_{B_R^{n}}) }
    \end{equation}
for any admissible $(\vec\delta,\vec\nu)$ and $4 \leq p \leq 4(n-1)$. By a standard interpolation argument, it suffices to prove the estimates at the endpoints $p=4$ and $p=4(n-1)$.  We first state three propositions, and derive Theorem \ref{0629.618}.

\begin{proposition}\label{05.29} Let $p_c:=4(n-1)$. For any $\e>0$, we have
\begin{equation}
        \| Tf\|_{L^{p_c}(W_{B_R^{n+1}})}\lesssim \delta^{-1} R^{- \frac{1}{n}(\frac12+\frac{2}{p_c}) +\e} \sup_{ s \in \I_{\delta}}\|T_sf\|_{\infty}^{\frac{p_c-4}{ p_c }} \Big\| \big(\sum_{s \in \I_{\delta} } |T_{s}f|^2 \big)^{\frac12} \Big\|_{L^4(W_{B_R^{n+1}})}^{\frac{4}{p_c}}.
    \end{equation}
\end{proposition}

\begin{proposition}\label{prop8.2}
For any $\e>0$, we have
    \begin{equation}
        \|Tf\|_{L^4(W_{B_R^{n+1}})} \lesssim \delta^{-1} R^{-1} R^{\frac14-\frac1n +\e} \|f\|_{L^4(W_{B_R^{n}})}.
    \end{equation}
\end{proposition}

\begin{proposition}\label{05.30.prop.10.2}
For any $\e>0$, we have
    \begin{equation}
        \Big\| \big(\sum_{     s \in \I_{\delta}     } |T_{s}f|^2 \big)^{\frac12} \Big\|_{L^4(W_{B_R^{n+1}})}^{4} \lesssim R^{-3+\e} \|f\|_{L^4(W_{B_R^{n}})}^4.
    \end{equation}
\end{proposition}

Let us assume these propositions and finish the proof. By interpolation, it suffices to prove for $p=4$ and $p=4(n-1)$.
By the uncertainty principle, we may assume that $f$ is essentially constant on each unit cube. By losing a power of $\log R$ and pigeonholing argument, we may assume that $|f| \sim X$ on the support of $f$ for some constant $X$. By Proposition \ref{05.29},
\begin{equation}
        \| Tf\|_{L^{p_c}(W_{B_R^{n+1}})}^{p_c} \lesssim \delta^{-p_c} R^{- \frac{p_c}{n}(\frac12+\frac{2}{p_c}) +p_c\e} \sup_{s \in \I_{\delta} }\|T_sf\|_{\infty}^{{p_c-4} } \Big\| \big(\sum_{s \in \I_{\delta} } |T_{s}f|^2 \big)^{\frac12} \Big\|_{L^4(W_{B_R^{n+1}})}^{4}.
    \end{equation}
By Proposition \ref{05.30.prop.10.2} and  the inequality $\sup_{s \in \I_{\delta} }\|T_{s}f\|_{\infty} \lesssim R^{-1} \|f \|_{\infty}$,
we have
\begin{equation}
     \| Tf\|_{L^{p_{c}}(W_{B_R^{n+1}})}^{p_c} \lesssim \delta^{-p_c}R^{-p_c}   R^{- \frac{p_c}{n}(\frac12+\frac{2}{p_c}) +O(\e)}
     R \|f\|_{L^\infty}^{p_c-4} \|f\|_{L^4(W_{B_R^{n}})}^{4}.
\end{equation}
Since $|f| \sim X$ on the support of $f$, we have
\begin{equation}
    \|f\|_{L^\infty}^{p_c-4} \|f\|_{L^4(W_{B_R^{n}})}^{4} \lesssim \|f\|_{L^{p_c}(W_{B_R^{n}}) }^{p_c}.
\end{equation}
This completes the proof. It remains to prove Proposition \ref{05.29}, \ref{prop8.2}, and \ref{05.30.prop.10.2}.

\subsection{Proof of Proposition \ref{05.29} and \ref{prop8.2}}

The proofs of the two propositions follow the same inductive scheme. Let us first fix $\vec{\delta}=(\delta_1,\ldots,\delta_{n-1})$ and $\vec\nu$. For $1 \leq J \leq n-1$, define  
\begin{equation}
    \vec{\delta}_{(J)}:=(\delta_1,\ldots,\delta_J,1,\ldots,1).
\end{equation}
\[ \vec\nu_{(J)}:=(\nu_1,\ldots,\nu_J,0,\ldots,0). \]
Let $r_J:=\de_1^{-n}\cdots\de_J^{-(n+1-J)}$. Let $P[\vec\de_{(J)},\vec\nu_{(J)}]$ be the $(\vec\de_{(J)},\vec\nu_{(J)})$-plank at scale $r_J$.
Define
\begin{equation}\label{defps}
    p[\vec\de_{(J)}](s):=r_J^{-1} P[\vec\de_{(J)},\vec\nu_{(J)}](s).
\end{equation}
 We have the nested sequence:
\[ p[\vec\de](s)\subset p[\vec\de_{(n-2)}](s)\subset \dots\subset p[\vec\de_{(1)}](s). \]
For \(1\le J\le n-1\), write
\[
\delta_{(J)}:=\delta_1\cdots\delta_J .
\]
For \(0<\sigma\le 1\), let \(I_\sigma\) denote a fixed maximal
\(\sigma\)-separated subset of \([0,1]\), chosen compatibly as \(\sigma\)
ranges over the scales \(\delta_{(J)}\).
For \(s\in I_{\delta_{(J)}}\), define the grouped piece
\[
T_{p[\vec\delta(J)](s)}f
:=
\sum_{{s'\in I_\delta: |s-s'|\lesssim \delta_{(J)}}}
T_{s'}f .
\]
In particular, when \(J=n-1\), we also write
\[
T_{p[\vec\delta](s)}f:=T_{p[\vec\delta(n-1)](s)}f .
\]
We write $p_d=d^2+d-2$ for $d \geq 2$.
To prove the propositions, we use the following inequality.

\begin{proposition}\label{05.30.prop10.3} For every $1 \leq J \leq n-2$, $4 \leq p \leq p_{n-J}$ and any $\e>0$, we have
    \begin{equation}
    \begin{split}
        \Big\| \big(\sum_{s \in \I_{\delta_1 \cdots \delta_J } } \big| &\sum_{ s' \in \I_{\delta_1 \cdots \delta_{J+1}: |s-s'| \leq \delta_1 \cdots \delta_J } } F_{p[\vec{\delta}_{(J+1)}](s') }  \big|^{\frac{p}{2}}\big)^{\frac{2}{p}} \Big\|_{L^{p}(W_{B_R^{n+1}})} 
        \\&
        \lesssim (\delta_{J+1})^{-(\frac12-\frac{2}{p } +\e ) } \Big\| \big(\sum_{ s' \in \I_{\delta_1 \cdots \delta_{J+1} }   }|F_{p[\vec{\delta}_{(J+1)}](s') } |^{\frac{p}{2}}\big)^{\frac{2}{p}} \Big\|_{L^{p}(W_{B_R^{n+1}})}
    \end{split}
    \end{equation}
    for any functions $F_{p[\vec{\delta}_{(J+1)}](s')}: \mathbb{R}^{n+1} \rightarrow \mathbb{C}$ with $\wh F_{p[\vec{\delta}_{(J+1)}](s')}\subset p[\vec{\delta}_{(J+1)}](s')$.
\end{proposition}
Let us assume Proposition \ref{05.30.prop.10.2} and \ref{05.30.prop10.3}, and finish the proof of  Proposition \ref{05.29} and \ref{prop8.2}. Let us first prove Proposition \ref{prop8.2}, which is easier. By Theorem \ref{25.04.26.thm13} with a standard localization argument, we have 
\begin{equation}
    \|Tf\|_{L^4(W_{B_R^{n+1}})} \lesssim \delta_{1}^{-\e} \Big\| \Big( \sum_{s \in \I_{\delta_{1}}} \big| \sum_{s' \in \I_{\de}: |s-s'| \leq \delta_1 } T_{s'}f \big|^{2} \Big)^{\frac12} \Big\|_{L^4(W_{B_R^{n+1}} )}.
\end{equation}
 Applying Proposition \ref{05.30.prop10.3} recursively  for all $1 \leq J \leq n-2$ with $p=4$,  and applying Proposition \ref{05.30.prop.10.2} at the last step, we have 
\begin{equation}
    \|Tf\|_{L^4(W_{B_R^{n+1}})} \lesssim \delta^{-\e} R^{-1}R^{\frac14+\e } \|f\|_{L^4(W_{B_R^n}) }.
\end{equation}
To prove Proposition \ref{prop8.2},
it suffices to prove
\begin{equation}
    1 \lesssim \delta^{-1}R^{-\frac{1}{n}}.
\end{equation}
Recall that $R=\delta_1^{-n} \cdots \delta_{n-1}^{-2}$ and $\delta=\delta_1 \cdots \delta_{n-1}$. After rearranging the terms and raising the $n$th power, the above inequality is equivalent to
\begin{equation}
    (\delta_1 \cdots \delta_{n-1})^n \lesssim \delta_1^n \cdots \delta_{n-1}^2.
\end{equation}
This holds true, and proves  Proposition \ref{prop8.2}.
\\

Let us next prove Proposition \ref{05.29}. For simplicity, set $p_c=4(n-1)$. Recall that $p_n=n^2+n-2$. Let $m$ be the largest integer such that
\begin{equation}\label{0629.817}
    p_{n+1-m}=(n+1-m)^2+(n+1-m)-2 \ge p_c.
\end{equation}
By Proposition \ref{05.04.prop33} and Proposition \ref{05.30.prop10.3}, we have 
\begin{equation}\label{0629.818}
\begin{split}
    \| Tf\|_{p_c}^{p_c} \lesssim R^{O(\e)} ( \prod_{j=1}^{m} \delta_j^{-1})^{(\frac12-\frac{2}{p_c })p_c} \Big\|  (\sum_{s \in \I_{\delta_1 \cdots \delta_{m} } }|T_{ p[\vec{\delta}_{(m)} ](s)  }f|^{\frac{p_c}{2}})^{\frac{2}{p_c }}  \Big\|_{p_c }^{p_c }
    .
\end{split}
\end{equation}
By the maximality of $m$, we have $p_c>p_{n-m}$, so Proposition \ref{05.30.prop10.3} is no longer applicable at the next step. Hence, we take $L^{\infty}$ and bound \eqref{0629.818} by 
\begin{equation}
\begin{split}
    \lesssim R^{O(\e)}
    ( \prod_{j=1}^{m} \delta_j^{-1})^{(\frac12-\frac{2}{p_c })p_c} \sup_{ s \in \I_{\delta_1 \cdots \delta_m}  }\|T_{ p[\vec{\delta}_{(m)}](s)   }f\|_{\infty}^{p_c-p_{ n-m  }} 
    \Big\|  (\sum_{ s \in \I_{\delta_1 \cdots \delta_m}   }|T_{ p[\vec{\delta}_{(m)} ](s)  }f|^{\frac{p_{n-m} }{2}})^{\frac{2}{p_{n-m}  }}  \Big\|_{p_{n-m}}^{p_{n-m}}.
\end{split}
\end{equation}
We can apply Proposition \ref{05.30.prop10.3} and repeat this process. Then this is bounded by 
\begin{equation*}
\begin{split}
    &\lesssim
    \delta^{-O(\e)}( \prod_{j=1}^{m} \delta_j^{-1})^{(\frac12-\frac{2}{p_c })p_c} \sup_{ s \in \I_{\delta_1 \cdots \delta_m}  }\|T_{ p[\vec{\delta}_{(m)}](s)   }f\|_{\infty}^{p_c-p_{ n-m  }} 
    \\& \times
    \Big(   \prod_{j=m+1}^{n-2} 
     (\delta_{j}^{-1})^{(\frac12-\frac{2}{p_{n+1-j}})p_{n+1-j}}
       \sup_{s \in \I_{\delta_{1} \cdots \delta_j }}\|T_{p[\vec{\delta}_{(j)}](s)  }f\|_{\infty}^{p_{n+1-j}-p_{n-j}}  \Big)
    \Big\|  (\sum_{ s \in \I_{\delta}   }|T_{ p[\vec{\delta} ](s)  }f|^{2})^{\frac12}  \Big\|_{4}^{4}.
\end{split}
\end{equation*}
We apply Proposition \ref{05.30.prop.10.2} and rearrange the terms. This is further bounded by
\begin{equation}\label{0629.last}
\begin{split}
    &\lesssim R^{O(\e)}
    ( \prod_{j=1}^{m} \delta_j^{-1})^{(\frac12-\frac{2}{p_c })p_c} \sup_{ s \in \I_{\delta_1 \cdots \delta_m}  }\|T_{ p[\vec{\delta}_{(m)}](s)   }f\|_{\infty}^{p_c-p_{ n-m  }} 
    \\& \times
    \Big(   \prod_{j=m+1}^{n-2} 
     (\delta_{j}^{-1})^{(\frac12-\frac{2}{p_{n+1-j}})p_{n+1-j}}
       \sup_{s \in \I_{\delta_{1} \cdots \delta_j }}\|T_{p[\vec{\delta}_{(j)}](s)  }f\|_{\infty}^{p_{n+1-j}-p_{n-j}}  \Big) \Big\| \big(\sum_{s \in \I_{\delta} } |T_{s}f|^2 \big)^{\frac12} \Big\|_{L^4(W_{B_R^{n+1}})}^{4}.
\end{split}
\end{equation}
Lastly, we use the following trivial estimate.
\begin{equation}
    \sup_{ s \in \I_{\delta_1 \cdots \delta_j} }\|T_{ p[\vec{\delta}_{(j)}](s)  }f\|_{\infty} \lesssim \big( \prod_{k=j+1}^{n-1}\delta_k^{-1} \big) \sup_{s \in \I_{\delta_1 \cdots \delta_{n-1}} } \|T_{p[\vec{\delta}](s) }f\|_{\infty}.
\end{equation}
Then  \eqref{0629.last} is bounded by 
\begin{equation}
\begin{split}
    &\lesssim R^{O(\e)}
    ( \prod_{j=1}^{m} \delta_j^{-1})^{(\frac12-\frac{2}{p_c })p_c} \big( \prod_{k=m+1}^{n-1} \delta_k^{-1} \big)^{p_c-p_{ n-m  }} 
    \\& \times
    \Big(   \prod_{j=m+1}^{n-2} \big(
     (\delta_{j}^{-1})^{(\frac12-\frac{2}{p_{n+1-j}})p_{n+1-j}}
       \big( \prod_{k=j+1}^{n-1} \delta_k^{-1} \big)^{p_{n+1-j}-p_{n-j}} \big) \Big) 
       \\& \times   \Big(\sup_{s \in \I_{\delta_1 \cdots \delta_{n-1}} } \|T_{p[\vec{\delta}](s) }f\|_{L^\infty}^{p_c-4} \Big) \Big\| \big(\sum_{s \in \I_{\delta} } |T_{s}f|^2 \big)^{\frac12} \Big\|_{L^4(W_{B_R^{n+1}})}^{4}.
\end{split}
\end{equation}
Hence,  to prove Proposition \ref{05.29}, by recalling $R=\delta_1^{-n}\cdots \delta_{n-1}^{-2}$, it suffices to check
\begin{equation}\label{0629.8230}
    \begin{split}
         &\big( \prod_{j=1}^{m} \delta_j^{-1}  \big)^{(\frac12-\frac{2}{ p_c}) p_c }  \big(  \prod_{k=m+1}^{n-1}\delta_k^{-1} \big)^{ p_c-p_{n-m}}
         \\& \times 
         \Big(   \prod_{j=m+1}^{n-2} 
     \big((\delta_{j}^{-1})^{(\frac12-\frac{2}{p_{n+1-j}})p_{n+1-j}}
       (\prod_{k=j+1}^{n-1}\delta_k^{-1} )^{p_{n+1-j}-p_{n-j}} \big) \Big)
         \\&\lesssim \Big(\prod_{j=1}^{n-1}\delta_j^{-p_c} \Big) (\prod_{j=1}^{n-1}\delta_{j}^{n-j+1} )^{\frac1n(\frac12+\frac{2}{p_c }) p_c}.
    \end{split}
\end{equation}
Since $p_c=4(n-1)$, by direct computations, one may see that
\begin{equation}
\begin{split}
    \textup{RHS}= \prod_{j=1}^{n-1} \delta_j^{-2n-2j+6}.
\end{split}
\end{equation}
We need to compute the left hand side of \eqref{0629.8230}. By computations, the exponent of $\delta_j$ for $1 \leq j \leq m$ is equal to $-2n+4$. Let us next consider the case that $m+1 \leq j \leq n-1$. By computations, the exponent of $\delta_j$ is as follows.
\begin{equation}
    \begin{split}
        -(p_c-p_{n-m}) - \Big(\frac12-\frac{2}{p_{n+1-j} } \Big)p_{n+1-j}- ( p_{n-m}-p_{n+1-j} ).
    \end{split}
\end{equation}
This number is equal to
\begin{equation}
    -p_c+\frac{p_{n+1-j} }{2}+2.
\end{equation}
So, the left hand side of \eqref{0629.8230} is equal to
\begin{equation}
    \big(\prod_{j=1}^m \delta_j^{-2n+4}  \big) \big(\prod_{j=m+1}^{n-1} \delta^{-p_c +\frac{p_{n+1-j}}{2}+2 } \big).
\end{equation}
Comparing the exponents on both sides of \eqref{0629.8230}, what we need to check becomes as follows. For $m+1 \leq j \leq n-1$,
\begin{equation}
    -p_c + \frac{p_{n+1-j}}{2} +2 \geq -2n-2j+6.
\end{equation}
Recall that $p_c=4n-4$ and $p_{n+1-j}=(n+1-j)^2+(n+1-j)-2$.
By calculations, one can see that this is the case. This completes the proof.

\begin{remark}
{\rm The estimate $\eqref{0630.new}$ is not true for $p>4(n-1)$ by the counterexample of \cite{BH}.
    Let us explain why our proof does not give \eqref{0630.new} for $p>4(n-1)$. Suppose that $p>4(n-1)$. If we follow our argument, \eqref{0630.new} is reduced to  proving \eqref{0629.8230} where $p_c$ is now replaced by $p$. If \eqref{0629.8230} were true for $p$, then by comparing the exponent of $\delta_{n-1}$, we must have
    \begin{equation}
        -p+\frac{p_{2}}{2}+2 \geq -p + \frac{2}{n}(\frac12+\frac2p)p.
    \end{equation}
    Since $p_2=4$, a computation shows that this is equivalent to $p \leq 4(n-1)$. So our proof cannot give \eqref{0630.new} for $p > 4(n-1)$. }
\end{remark}

\subsection{Proof of Proposition \ref{05.30.prop10.3}}

Next, we prove Proposition \ref{05.30.prop10.3}. The proof is very similar to that of Proposition \ref{05.24.prop53}, except that the rescaling is more involved. 

\begin{proof}[Proof of Proposition \ref{05.30.prop10.3}] Let $4 \leq p \leq  p_{n-J}$. Fix $\e>0$. Because the argument is iterative, we introduce intermediate scales.
\begin{equation}
\begin{split}
    &\vec{\delta}_{(J)}':=(\delta_1,\ldots, \delta_J,\delta_{J+1}' ,\ldots,1),
    \\&
    \vec{\delta}_{(J)}'':=(\delta_1,\ldots, \delta_J,\delta_{J+1}'' ,\ldots,1).
\end{split}
\end{equation}
Here $\delta_{J+1} \leq \delta_{J+1}' \leq 1$ and $\delta_{J+1}'':= \delta_{J+1}' \delta_{J+1}^{\e}$.
We also introduce
\begin{equation}\label{0703.832}
    \delta_{(J+1)}':= \delta_1 \cdots \delta_J \delta_{J+1}', \;\;\; \delta_{(J+1)}'':= \delta_1 \cdots \delta_J \delta_{J+1}''.
\end{equation}
We claim that 
\begin{equation}\label{0629.831}
\begin{split}
        &\Big\| \big(\sum_{ s' \in \I_{\delta_{(J+1)}'}  } \big| \sum_{s \in \I_{\delta_{(J+1)} } :|s'-s| \leq \delta_{(J+1)}' }  F_{p[\vec{\delta}_{(J+1)} ](s) }  \big|^{\frac{p}{2}}\big)^{\frac{2}{p}} \Big\|_{L^{p}} \\&\lesssim (\delta_{J+1}^{\e})^{-(\frac12-\frac{2}{p }+\e) } \Big\| \big(\sum_{ s'' \in \I_{\delta_{(J+1)}''} } \big| \sum_{  s \in \I_{\delta_{(J+1)}}:|s-s''| \leq \delta_{(J+1)}''   } F_{  p[ \vec{\delta}_{(J+1)} ](s)  }  \big|^{\frac{p}{2}}\big)^{\frac{2}{p}} \Big\|_{L^{p}}
            \\& +
            (\delta_{J+1}'^{-1}  \delta_{J+1})^{-(\frac12-\frac{2}{p } +\e) } \Big\| \big(\sum_{ s \in \I_{\delta_{(J+1)}}  }|F_{  p[ \vec{\delta}_{(J+1)} ](s)} |^{\frac{p}{2}}\big)^{\frac{2}{p}} \Big\|_{L^{p}}.
    \end{split}
    \end{equation}
    Assuming the claim, repeated application of this inequality gives the desired result. So, let us prove the claim. 
Apply Proposition \ref{0616.thm33} to the left hand side of \eqref{0629.831}. Then the left hand side is bounded by 
\begin{equation}\label{0620.821}
         \sum_{\de_{(J+1)}' \le \si\le 1} \sum_{s''\in\I_\si} \sum_{U\parallel (q[ \le \si](s''))^*} |U| \Big(\int_U \sum_{s' :|s'-s''|\leq \sigma } \big| \sum_{s:|s-s'| \leq \delta'_{(J+1)} } F_{p[\vec{\delta}_{(J+1)}](s) } \big|^{\frac{p}{2} }  w_U \Big)^2.
\end{equation}
We consider two cases according to the size of $\sigma$: 
\begin{equation}\label{0619.822}
    \sigma \leq {\delta}_{(J+1)}' \delta_{J+1}^{-\e}  \;\;\; \mathrm{or} \;\;\; \sigma \geq {\delta}_{(J+1)}'\delta_{J+1}^{-\e}.
\end{equation}
So we write the right hand side of \eqref{0620.821} as follows.
\begin{equation}\label{0620.823}
    \begin{split}
        &\sum_{ \si\le  \delta_{(J+1)}' \delta_{J+1}^{-\e} } \sum_{s''\in\I_\si} \sum_{U} |U|\Big(\int \sum_{s' :|s'-s''|\leq \sigma } \big| \sum_{s:|s-s'| \leq \delta'_{(J+1)} } F_{p[\vec{\delta}_{(J+1)}](s) } \big|^{\frac{p}{2} }  w_U \Big)^2
        \\&+
        \sum_{ \sigma \geq \delta_{(J+1)}' \delta_{J+1}^{-\e} } \sum_{s''\in\I_\si} \sum_{U} |U| \Big(\int \sum_{s' :|s'-s''|\leq \sigma } \big| \sum_{s:|s-s'| \leq \delta'_{(J+1)} } F_{p[\vec{\delta}_{(J+1)}](s) } \big|^{\frac{p}{2} } w_U  \Big)^2.
    \end{split}
\end{equation}
To deal with the first term, we use Proposition \ref{0618.prop71}. To deal with the second term, we use a decoupling inequality for a nondegenerate curve in $\mathbb{R}^{n-J-1}$.

Let us first bound the first case. After applying H\"{o}lder's inequality, the first term is bounded by
\begin{equation}
\delta_{J+1}^{-O(\e)}
\sum_{s'\in\I_{{\delta}_{(J+1)}' }}
    \Big\|\sum_{s: |s-s'| \leq \delta_{(J+1)}' }  F_{p[\vec{\delta}_{(J+1)}](s)} \Big\|_{L^{p}}^{p}.
\end{equation}
Let us fix $s'$. By using the linear transformation argument (see the discussion around \eqref{0623.648}), we may assume that $s'=0$. We need to deal with
\begin{equation}
    \Big\|\sum_{s \in \I_{\delta_{(J+1)} }: |s| \leq \delta_{(J+1)}' }  F_{p[\vec{\delta}_{(J+1)}](s)} \Big\|_{L^{p}}^{p}.
\end{equation}
We next apply the rescaling. Define the rescaling map $\widetilde{L}_{J,X}: \mathbb{R}^{n+1} \rightarrow \mathbb{R}^{n+1}$ by
\begin{equation}
    \widetilde{L}_{J,X}(\xi):=\Big(X^{n+1-J}\xi_1,  \ldots, X\xi_{n+1-J}, \xi_{n+2-J},\ldots, X^{-(J-1)}\xi_{n+1} \Big).
\end{equation}
Here is the rescaling map we use.
\begin{equation}
L(\xi):=L_{{J+1},(\delta_{J+1}')^{-1} }\circ L_{J,{\delta_J^{-1}}}\circ \cdots \circ L_{1,{\delta_1^{-1}}}(\xi).
\end{equation} 
By definition, 
\begin{equation*}
\begin{split}
    L_{1,\delta_1^{-1}}(\xi)&:=(\delta_1^{-n}\xi_1 ,\delta_1^{-(n-1)}\xi_2,\ldots,\delta_1^{-1}\xi_n,\xi_{n+1} ),
    \\ L_{2,\delta_2^{-1}}(\xi)&:=(\delta_2^{-(n-1)}\xi_1,\ldots,\delta_2^{-1}\xi_{n-1},\xi_n,\delta_2 \xi_{n+1}),
    \\& \vdots
    \\ L_{J,\delta_J^{-1}}(\xi)&:=(\delta_{J}^{-(n+1-J)}\xi_1,  \ldots, \delta_J^{-1}\xi_{n+1-J}, \xi_{n+2-J},\delta_J \xi_{n+3-J},\ldots, \delta_J^{J-1}\xi_{n+1} ),
    \\ L_{J+1,(\delta_{J+1}')^{-1}}(\xi)&:=( (\delta_{J+1}')^{-(n-J)}\xi_1, \ldots, (\delta_{J+1}')^{-1} \xi_{n-J},\xi_{n+1-J},\ldots, (\delta_{J+1}')^{J}\xi_{n+1} ).
\end{split}
\end{equation*}
One may see that $p[\vec{\delta}_{(J+1)}](s)$ is contained in the following set. 
\begin{equation}
\begin{split}
       \Big\{ \sum_{j=1}^{n+1} a_j \gamma_{n+1}^{(j)}(s) :\;  & |a_{J+1}| \sim  \widetilde{\rho}_{J}, \;\; |a_j| \lesssim \widetilde{\rho}_{j-1}, \;\; j \neq J+1  \Big\}.
    \end{split}
    \end{equation}
Here,  the number $\widetilde{\rho}_j$ is defined by
\begin{equation}
\widetilde{\rho}_j=
\begin{cases}
1, \;\;\;\;\;\;\;\;\;\;\;\;\;\;\;\;\;\;\;\;\;\;\;\;\;\;\;\;\;\;\;\;\;\;\,\;\;\;\;\;\;\;\;\, j = 0
\\     \delta_1^j \cdots \delta_j, \;\;\;\;\;\;\;\;\;\;\;\;\;\;\;\;\;\;\;\;\;\;\;\;\;\;\;\;\;\;\;\;\; 1 \leq j \leq J+1
    \\
    \delta_1^j \cdots \delta_J^{j-J+1} (\delta_{J+1})^{j-J}, \;\;\; J+2 \leq j \leq n
\end{cases}
\end{equation}
Recall the definition of $\delta_{(J+1)}'$ in \eqref{0703.832}. Define
\begin{equation}
        \widetilde{\gamma}_{n+1}(s):= \big( \delta_1 \delta_2^2 \cdots \delta_J^J (\delta_{J+1}')^{J+1} \big)^{-1} L \big( \gamma_{n+1}( \delta'_{(J+1)} s) \big).
    \end{equation}
For $j \geq J+1$ we have
\begin{equation}
    \widetilde{\gamma}_{n+1}^{(j)}(s)= \widetilde{\rho}_{j-1} (\delta_{J+1} \delta_{J+1}'^{-1} )^{J-(j-1)} L \big( \gamma_{n+1}^{(j)}( \delta'_{(J+1)} s) \big).
\end{equation}
Then
the set $L(p(s))$ is contained in
\begin{equation}
\begin{split}
       \Big\{ \sum_{j=1}^{n+1} a_j \widetilde{\gamma}_{n+1}^{(j)}\big((\delta_{(J+1)}')^{-1}  s \big) :\;  |a_{J+1}| \sim 1,\;\; 
        |a_j| \lesssim ( (\delta_{J+1}')^{-1} \delta_{J+1})^{j-1-J}, \; J+2 \leq j \leq n+1  \Big\}.
    \end{split}
    \end{equation}
This is a piece of the form $\Xi_{n-J,J+1}( (\delta_{J+1}'^{-1}\delta_{J+1})^{-(n-J)} )$.
So we can apply the $l^{p/2}$ function estimate (Proposition \ref{0618.prop71}) with $(n,k)$ replaced by $(n-J,J+1)$. Hence, the first term of \eqref{0620.823} is bounded by
\begin{equation}
    \delta_{J+1}^{-O(\e)}  ( \delta_{J+1}'^{-1} \delta_{J+1})^{-(\frac12-\frac{2}{p } +\e) } \Big\| \big(\sum_{ s \in \I_{\delta_{(J+1)}}  }|F_{  p[ \vec{\delta}_{(J+1)} ](s)} |^{\frac{p}{2}}\big)^{\frac{2}{p}} \Big\|_{L^{p}}.
\end{equation}
This gives the desired result because this is the second term of \eqref{0629.831} on the right hand side.
\\

Let us next bound the second term of \eqref{0620.823}. We use $\widetilde{R}:=\delta_1^{-n} \cdots \delta_{J+1}^{-n+J}$.  We need to deal with
\begin{equation}
    \int_{\R^{n+1}} \Big| \sum_{s: |s-s'| \leq \delta_{(J+1)}' }F_{ p[\vec{\delta}_{(J+1)}](s) } \Big|^{\frac{p}{2}} w_U.
\end{equation}
The proof of this part is essentially parallel to that for dealing with the second term of \eqref{0627.718}.
As before, by using the linear transformation argument (see the discussion around \eqref{0623.648}), we may assume that $s'=0$, and we use the same rescaling map $L$. Since $U\| (q[ \le \si](0))^*$, the set $U$ has dimension \begin{equation}
    [-\widetilde{R},\widetilde{R}] \times [-\widetilde{R} \sigma, \widetilde{R} \sigma] \times \cdots \times [-\widetilde{R} \sigma^{n-J}, \widetilde{R} \sigma^{n-J}] \times B
\end{equation}
for some rectangular box $B \subset \R^{J}$ (see the discussion around \eqref{0630.5293}).
Apply $L$ to the frequency variables and $L^{-1}$ to the physical variables. The domain $U$  becomes $L^{-1}(U)$, which is comparable to 
\begin{equation*}
    [-1,1] \times [ -(\delta_{(J+1)}'^{-1} \sigma),(\delta_{(J+1)}'^{-1} \sigma)] \times \cdots \times  
    [-(\delta_{(J+1)}'^{-1} \sigma)^{n-J},(\delta_{(J+1)}'^{-1} \sigma)^{n-J} ]
    \times
    \widetilde{B}
\end{equation*}
for some rectangular box $\widetilde{B} \subset \mathbb{R}^{J+1}$ (compare this with \eqref{0628.7330}). By the assumption that $\sigma \geq \delta_{(J+1)}' \delta_{J+1}^{-\e}$, the set contains
\begin{equation}
    [-1,1] \times [-\delta_{J+1}^{-\e},\delta_{J+1}^{-\e} ]^{n-J-1} \times \tilde{B}.
\end{equation}
Freeze the first coordinate and the coordinates from the $(n-J+1)$th to the $(n+1)$th, and apply  the decoupling inequality for a nondegenerate curve in $\mathbb{R}^{n-J-1}$ by \cite{MR3548534}. To apply the theorem, we need to check that 
\begin{equation}\label{0630.848}
\frac{p}{2} \leq (n-J-1)(n-J).    
\end{equation}
Recall that we have the condition $p \leq p_{n-J}=(n-J)^2+(n-J)-2$. So \eqref{0630.848} is true. By the decoupling theorem, we have 
\begin{equation}
\begin{split}
    &\int_{\R^{n+1}} \Big| \sum_{s: |s-s'| \leq \delta_{(J+1)}' }F_{ p[\vec{\delta}_{(J+1)}](s) } \Big|^{\frac{p}{2}} w_U \\& 
    \lesssim (\delta_{J+1}^{-\e})^{(\frac12-\frac2p) \frac{p}{2}+\e} \sum_{s'' \in \I_{\delta''_{(J+1)}}: |s'-s''| \leq \delta_{(J+1)}' } \int_{\R^{n+1}}
    \Big| \sum_{s: |s-s''| \leq \delta_{(J+1)}'' }F_{ p[\vec{\delta}_{(J+1)}](s) } \Big|^{\frac{p}{2}} w_U.
\end{split}
\end{equation}
By embedding, the second term of \eqref{0620.823} is bounded by
\begin{equation}
    (\delta_{J+1}^{\e})^{-(\frac12-\frac{2}{p }+\e) } \Big\| \big(\sum_{ s'' \in \I_{\delta_{(J+1)}''} } \big| \sum_{  s \in \I_{\delta_{(J+1)}}:|s-s''| \leq \delta_{(J+1)}''   } F_{  p[ \vec{\delta}_{(J+1)} ](s)  }  \big|^{\frac{p}{2}}\big)^{\frac{2}{p}} \Big\|_{L^{p}}.
\end{equation}
This is the first term of \eqref{0629.831} on the right hand side.
\end{proof}

\subsection{Proof of Proposition \ref{05.30.prop.10.2}}
The proof is a simple application of Proposition \ref{0616.thm33}.    By Proposition \ref{0616.thm33} with $q=2$, we have 
\begin{equation}
\begin{split}\label{0617.821}
    \Big\| \big(\sum_{s } |T_sf|^2 \big)^{\frac12} \Big\|_{L^4(W_{B_R^{n+1}})}^{4} 
    \lesssim \delta^{-\e} \sum_{\de\le \si\le 1} \sum_{s'\in\I_\si} \sum_{U\parallel (q[ \le \si](s'))^*} |U| \Big(\int \sum_{s:|s-s'|\le \si}|T_{s}f|^2 w_U \Big)^2.
    \end{split}
\end{equation}
Recall that
\begin{equation}
    T_{s}f(x,t):= a_s \big(\wh f(\xi)\phi_{p[\vec{\delta}](s)} \big)^\vee \big((x,t)-R\bfn_s \big)
\end{equation}
where $|a_s| \lesssim 1$.
By the geometry of the tiles $U$, \eqref{0617.821} is bounded by
\begin{equation}\label{0701.853}
    \delta^{-\e} \sum_{\de\le \si\le 1} \sum_{s'\in\I_\si} \sum_{U\parallel (q[ \le \si](s'))^*} |U| \Big(\int \sum_{s:|s-s'|\le \si}|( \hat{f} \phi_{p[\vec{\delta}](s)} )^{\vee}(x,t) |^2 w_U(x,t) \Big)^2.
\end{equation}

Let $\pi:\mathbb R^{n+1}\to\mathbb R^n$ denote the projection
$\pi(\xi,\xi_{n+1})=\xi$.
We first note that the projected planks
$\{\pi(p[\vec\delta](s))\}_{s\in I_\delta}$ have bounded overlap.
Indeed,
\begin{equation}
\left|
\det\bigl(\pi(e_1(s)),\ldots,\pi(e_n(s))\bigr)
\right|
=
\left|(e_{n+1}(s))_{n+1}\right|
\sim 1.
\end{equation}
Moreover, the Taylor expansion of $\pi(e_i(s+h))$ in the frame
$\{\pi(e_j(s))\}_{j=1}^n$ has, for $i<j\le n$, the coefficient
\begin{equation}
\frac{h^{j-i}}{(j-i)!}
\kappa_i(s)\cdots\kappa_{j-1}(s)
+O\bigl(|h|^{j-i+1}\bigr)
\end{equation}
in the $\pi(e_j(s))$ direction.
The contribution of the last frequency coordinate to the
projected plank is $O(R^{-1})$ in this frame.
Consequently,
\begin{equation}
\sum_{s\in I_\delta}
\mathbf 1_{\pi(p[\vec\delta](s))}(\xi)\lesssim 1,
\end{equation}
with the same conclusion for the fixed enlargements used below.

For each fixed $\xi$, the support of
$\phi_{p[\vec\delta](s)}(\xi,\xi_{n+1})$
in $\xi_{n+1}$ has length $O(R^{-1})$. Hence
\begin{equation}
\sum_{s\in I_\delta}
\int_{\mathbb R}
\left|
\phi_{p[\vec\delta](s)}(\xi,\xi_{n+1})
\right|^2\,d\xi_{n+1}
\lesssim R^{-1}.
\end{equation}
By Plancherel, for every $h\in L^2(\mathbb R^n)$,
\begin{equation}
\sum_{s\in I_\delta}
\left\|
\bigl(\widehat h\,\phi_{p[\vec\delta](s)}\bigr)^\vee
\right\|_{L^2(\mathbb R^{n+1})}^2
\lesssim R^{-1}\|h\|_{L^2(\mathbb R^n)}^2.
\end{equation}

Fix $\delta\le \sigma\le 1$ and $s'\in I_\sigma$.
Choose $\vec\delta'$ as in Lemma~\ref{conv}, and let $\psi_{s'}$
be a smooth cutoff adapted to the spatial projection
$\pi(p[\vec\delta'](s'))$, equal to one on
\begin{equation}
\bigcup_{|s-s'|\le \sigma}
\pi\bigl(\operatorname{supp}\phi_{p[\vec\delta](s)}\bigr).
\end{equation}
Thus, whenever $|s-s'|\le \sigma$, we may replace
$\widehat f$ by $\widehat f\,\psi_{s'}$
in $\widehat f\,\phi_{p[\vec\delta](s)}$.
We use $\vee_n$ to denote inverse Fourier transform in the
$n$ spatial variables.

For a tile $U\parallel(q[\le \sigma](s'))^*$, let $B_U$
be the spatial box obtained after applying the shear
\begin{equation}
(x,t)\longmapsto (x-t\gamma(s'),t).
\end{equation}
More precisely, the image of $U$ is comparable to the product
of $B_U$ and an interval of length comparable to $R$.
In particular,
\begin{equation}
|U|\sim R|B_U|.
\end{equation}
The kernels of the original multipliers are adapted to boxes
contained in
$(q[\le \sigma](s'))^*$. Hence,
\begin{equation}
\begin{aligned}
&\sum_{|s-s'|\le \sigma}
\int_{\mathbb R^{n+1}}
\left|
\bigl(\widehat f\,\phi_{p[\vec\delta](s)}\bigr)^\vee
\right|^2w_U\\
&\qquad\lesssim
\frac{R^{-1}}{|U|}
\int_{\mathbb R^n}
\left|
\bigl(\widehat f\,\psi_{s'}\bigr)^{\vee_n}
\right|^2W_{B_U}\\
&\qquad\lesssim
R^{-2}
\int_{\mathbb R^n}
\left|
\bigl(\widehat f\,\psi_{s'}\bigr)^{\vee_n}
\right|^2w_{B_U}.
\end{aligned}
\end{equation}
We apply this estimate directly to \eqref{0701.853}.
For each fixed $\sigma$, Jensen's inequality yields
\begin{equation}
\begin{aligned}
&\sum_{s'\in I_\sigma}\sum_U
|U|
\left(
\sum_{|s-s'|\le \sigma}
\int_{\mathbb R^{n+1}}
\left|
\bigl(\widehat f\,\phi_{p[\vec\delta](s)}\bigr)^\vee
\right|^2w_U
\right)^2\\
&\qquad\lesssim
R^{-4}\sum_{s'\in I_\sigma}\sum_U
|U|
\left(
\int_{\mathbb R^n}
\left|
\bigl(\widehat f\,\psi_{s'}\bigr)^{\vee_n}
\right|^2w_{B_U}
\right)^2\\
&\qquad\lesssim
R^{-4}\sum_{s'\in I_\sigma}\sum_U
|U|
\int_{\mathbb R^n}
\left|
\bigl(\widehat f\,\psi_{s'}\bigr)^{\vee_n}
\right|^4w_{B_U}\\
&\qquad\lesssim
R^{-3}\sum_{s'\in I_\sigma}
\int_{\mathbb R^n}
\left|
\bigl(\widehat f\,\psi_{s'}\bigr)^{\vee_n}
\right|^4W_{B_R^n}.
\end{aligned}
\end{equation}
In the last step, we used $|U|\sim R|B_U|$ and the bounded
overlap of the boxes $B_U$ for tiles near $B_R$.

The projected planks
$\{\pi(p[\vec\delta'](s'))\}_{s'\in I_\sigma}$
also have bounded overlap.
Thus the spatial multipliers
\begin{equation}
f\longmapsto
\bigl(\widehat f\,\psi_{s'}\bigr)^{\vee_n}
\end{equation}
satisfy an $L^2\to\ell^2(L^2)$ bound.
Each multiplier also has a uniform $L^\infty$ bound,
since its kernel has bounded $L^1$ norm.
Interpolation and localization at scale $R$ give
\begin{equation}
\sum_{s'\in I_\sigma}
\int_{\mathbb R^n}
\left|
\bigl(\widehat f\,\psi_{s'}\bigr)^{\vee_n}
\right|^4W_{B_R^n}
\lesssim
\int_{\mathbb R^n}|f|^4W_{B_R^n}.
\end{equation}
Substituting these estimates into \eqref{0701.853} and summing over
the dyadic scales $\sigma$, we obtain
\begin{equation}
\left\|
\left(\sum_{s\in I_\delta}|T_sf|^2\right)^{1/2}
\right\|_{L^4(W_{B_R^{n+1}})}^4
\lesssim
R^{-3+O(\epsilon)}
\|f\|_{L^4(W_{B_R^n})}^4.
\end{equation}
This completes the proof.

\section{Proof of Theorem \ref{2025.04.25.thm11}}

In this section, we prove Theorem \ref{2025.04.25.thm11}.

 The theorem is obvious when $p=2$ and $p=\infty$. By interpolation, it suffices to prove the estimate for $p=4$ and $p=n^2+n-2$. For the rest of the section, let us assume that $p=4$ or $p=n^2+n-2$.
 By a standard localization argument,   we may assume that $f$ is supported on the ball $B_{CR}$ for some constant $C>0$. It suffices to prove 
\begin{equation}
    \|\mathrm{BR}_R(f)\|_{L^p(B_R)} \leq C_{\epsilon}R^{\epsilon} \big( 1+R^{\frac1n(\frac{1}{2}-\frac{2}{p}) }  +R^{\frac1n}R^{-\frac{n+1}{2p}-\frac{1}{np}} \big) \|f\|_{L^p}
\end{equation}
Here, $B_{CR}$ and $B_{R}$ have the same center. By the uncertainty principle, we may further assume that $f$ is constant on a unit ball.
By normalization, we may assume that $\|f\|_p=1$. Write
\begin{equation}
    f(x)= \sum_{j} f(x)\chi_{|f(x)| \sim 2^j}.
\end{equation}
Since $f$ is constant at scale $1$ and supported on $B_{CR}$, after discarding the part where \(|f|\le R^{-N}\), which contributes a rapidly
decaying error, there are $O(\log R)$ many $j$'s. Denote by $f_j$ the function $f(x)\chi_{ |f(x)| \sim 2^j}$. By the wave envelope estimate (Theorem \ref{0501.thm12}),
\begin{equation*}
        \|\mathrm{BR}_R(f_j)\|_{L^p(B_R)} \leq C_{\e}R^{\e} R^{\frac1n(\frac12-\frac2p)} \Big( \sum_{ \substack{ R^{-\frac1n} \leq s \leq 1: \\ dyadic } } \sum_{\tau \in \Theta^n(s^{-n}) } \sum_{U \| U_{\tau,R}  } |U| \|S^{curve}_{U,{\frac{p}{2}}}f_j\|_{L^{\frac{p}{2}}}^p \Big)^{\frac1p}. 
    \end{equation*}
    Recall that
    \begin{equation}
    S_{U,\frac{p}{2}}^{curve}f(x)= \Big(\sum_{\theta \subset \tau}|f_{\theta}(x)|^{\frac{p}{2}} w_U(x) \Big)^{\frac2p}.
\end{equation}
We need to prove
\begin{equation}\label{05.30.83}
    \sum_{ \substack{ R^{-\frac1n} \leq s \leq 1: \\ dyadic } } \sum_{\tau \in \Theta^n(s^{-n}) } \sum_{U \| U_{\tau,R}  } |U| \|S^{curve}_{U,{\frac{p}{2}}}f_j\|_{L^{\frac{p}{2}}(\R^n)}^p \lesssim \|f_j\|_{L^p(\R^n)}^p.
\end{equation}
Since there are $\log(R)$ many $s'$, it suffices to prove
\begin{equation}
    \sum_{\tau \in \Theta^n(s^{-n}) } \sum_{U \| U_{\tau,R}  } |U| \|S^{curve}_{U,{\frac{p}{2}}}f_j\|_{L^{\frac{p}{2}}(\R^n)}^p \lesssim \|f_j\|_{L^p(\R^n)}^p.
\end{equation}
By direct calculations, we have
\begin{equation}
\begin{split}
    |U|\|S^{curve}_{U,{\frac{p}{2}}}f_j\|_{L^{\frac{p}{2}}(\R^n)}^p &\lesssim 
     |U|\|S^{curve}_{U,2}f_j\|_{L^{2}(\R^n)}^4 \sup_{\theta} \|(f_{j})_{\theta}\|_{\infty}^{p-4}
     \\&
     \lesssim |U|\|S^{curve}_{U,2}f_j\|_{L^4}^4 \|f_j\|_{\infty}^{p-4}.
\end{split}
\end{equation}
Hence,
\begin{equation}
    \begin{split}
         \sum_{\tau \in \Theta^n(s^{-n}) } \sum_{U \| U_{\tau,R}  } |U| \|S^{curve}_{U,{\frac{p}{2}}}f_j\|_{L^{\frac{p}{2}}(\R^n)}^p
         \lesssim  \|f_{j}\|_{\infty}^{p-4} \int (\sum_{\theta} |(f_j)_{\theta}|^2)^2.
    \end{split}
\end{equation}
To proceed with the proof, we use the following inequality
    \begin{equation}
    \int (\sum_{\theta} |(f_j)_{\theta}|^2)^2 \lesssim \|f_j\|_4^4.
\end{equation}
This inequality can be proved by following the proof of Proposition \ref{05.30.prop.10.2}. Since the proof is identical, we do not reproduce it here.
Since $f_j(x)=f(x) \chi_{|f(x)| \sim 2^j}$, we have $\|f_j\|_{\infty}^{p-4}\|f_j\|_4^4 \sim \|f_j\|_p^p$. Combining all these inequalities gives
\begin{equation}\label{070199}
    \sum_{\tau \in \Theta^n(s^{-n}) } \sum_{U \| U_{\tau,R}  } |U| \|S^{curve}_{U,{\frac{p}{2}}}f_j\|_{L^{\frac{p}{2}}(\R^n)}^p
         \lesssim  \|f_j\|_p^p.
\end{equation}
So we proved \eqref{05.30.83}, and this completes the proof.

\section{Sharpness of Theorem \ref{2025.04.25.thm11}}\label{sec10}

In this section, we give a heuristic explanation for the sharpness of the Bochner-Riesz type estimate (Theorem \ref{2025.04.25.thm11}). The idea goes back to a counterexample for a ball multiplier operator by  \cite{MR296602}.

Let $C_p(R)$ be the smallest constant such that
\begin{equation}
    \|\mathrm{BR}_R(f)\|_{L^p} \leq C_p(R) \|f\|_{L^p}
\end{equation}
for all functions $f$. We need to show that
\begin{equation}\label{0702.102}
    C_p(R) \gtrsim 1+R^{\frac1n(\frac{1}{2}-\frac{2}{p}) }  +R^{\frac1n}R^{-\frac{n+1}{2p}-\frac{1}{np}}.
\end{equation}
We will construct two examples to prove the lower bound \eqref{0702.102}. \\

The first example can be constructed as follows. Let $f=\sum_{\theta}f_{T_{\theta}}$ where $f_{T_{\theta}}$ is a wave packet with height one, and $T_{\theta}$ has dimension $R^{1/n} \times R^{2/n} \times \cdots \times  R$. Here, $\theta$ determines the direction of $T_{\theta}$ in a canonical way. The volume of each plank is $R^{\frac{n+1}{2}}$. The number of planks is $R^{\frac1n}$. We take $T_{\theta}$ so that a translation of $T_{\theta}$ by $R$ units along the longest direction intersects the origin.
So we have
\begin{equation}
    \|f\|_{L^p} \lesssim R^{\frac{n+1}{2p}+\frac{1}{np}}.
\end{equation}
On the other hand, heuristically, $\mathrm{BR}_R(f_{T_{\theta}})$ is a wave packet with height comparable to one and essentially supported on $10T_{\theta}$.
Considering the outer and inner parts, we have
\begin{equation}
    \begin{split}
        \|\mathrm{BR}_Rf\|_{L^p} \gtrsim R^{\frac1n}+R^{\frac{n+1}{2p}+\frac{1}{np}}.
    \end{split}
\end{equation}
So we have
\begin{equation}
    C_p(R) \gtrsim 1+ R^{\frac1n}R^{-\frac{n+1}{2p}-\frac{1}{np}}.
\end{equation}
This gives the first and third terms on the right hand side of \eqref{0702.102}.
\\

The second example can be constructed as follows. Let $\{a_{\theta}\}$ be random signs.
Let  $ f = \sum_{\theta} a_{\theta}\sum_{T_{\theta}} f_{T_{\theta}}$ where $f_{T_{\theta}}$ is a wave packet with height one. We stack these planks so that for given $\theta$ the union of the planks $\bigcup T_{\theta}$ has  dimension $R^{(n-1)/n} \times R^{(n-1)/n} \times R^{(n-1)/n} \times \cdots \times R^{(n-1)/n}  \times  R$. We take the  planks so that a translation of $\cup T_{\theta}$ by $R$ units along the longest direction intersects the origin. By computations, for given $\theta$, 
\begin{equation}
    \Big| \bigcup T_{\theta} \Big| \sim   R^{\frac{(n-1)^2}{n}+1 } \sim R^{n-1+1/n}.
\end{equation}
Recall that the number of directions is comparable to $R^{1/n}$.
This gives the bound
\begin{equation}
    \|f\|_{L^p} \lesssim R^{(n-1+\frac2n )\frac1p }.
\end{equation}
On the other hand, heuristically, $\mathrm{BR}_R(\sum_{T_{\theta}} f_{T_{\theta}})$ is a wave packet with height comparable to one and essentially supported on $ \bigcup 10T_{\theta}$.
Using the square root cancellation bound on $B_{R^{(n-1)/n}}$, we have
\begin{equation}
\begin{split}
    \|\mathrm{BR}_Rf\|_{L^p} \gtrsim  R^{\frac{n-1}{p}+\frac{1}{2n}}.
\end{split}
\end{equation}
This gives the lower bound
\begin{equation}
    C_p(R) \gtrsim R^{\frac{n-1}{p}+\frac{1}{2n}}R^{-(n-1+\frac2n )\frac1p } \sim R^{\frac1n(\frac12-\frac2p)}. 
\end{equation}
This gives the second bound on the right hand side of \eqref{0702.102}. This explains the lower bound of \eqref{0702.102}.

\appendix
\section{} 
Recall the following result from \cite[Appendix D]{MR4861588}.
\begin{lemma}\label{BGHSlem} 
    For $a\in C_c^\infty (\R)$ supported in an interval $I\subset \R$ and $\phi\in C^\infty(I)$, define the oscillatory integral 
    \[ \mathcal I[\phi,a]:=\int_\R e^{i\phi(s)}a(s)\mathrm{d}s. \]
Let $r\ge 1$. Suppose that for each $j\in \mathbb N_0$, there exist constants $C_j\ge 1$ such that the following conditions hold on $\supp\ a$:
\begin{enumerate}
    \item $|\phi'(s)|>0$,
    \item $|\phi^{(j)}(s)|\le C_j r^{-(j-1)}|\phi'(s)|^j$ for all $j\ge 2$,
    \item $|a^{(j)}(s)|\le C_j r^{-j}|\phi'(s)|^j$ for all $j\ge 0$.
\end{enumerate}
Then for all $N\in\mathbb N_0$ there exists some constant $C(N)$ such that
\[ |\mathcal{I}[\phi,a]|\le C(N) \cdot |\supp\ a|\cdot r^{-N}. \]
    
\end{lemma}

\bibliographystyle{alpha}
\bibliography{reference}

\end{document}